\documentclass[11pt]{amsart}
\usepackage{wrapfig,times,epsfig, graphicx, array, subfigure}
\usepackage[all,knot,arc,poly]{xy}
\usepackage{ifthen}
\numberwithin{figure}{section}
\newtheorem{thm}{Theorem}[subsection]
\newtheorem{cor}[thm]{Corollary}
\newtheorem{lem}[thm]{Lemma}
\newtheorem{prop}[thm]{Proposition}
\theoremstyle{definition}
\newtheorem{defn}[thm]{Definition}
\theoremstyle{remark}
\newtheorem{conj}[thm]{Conjecture}
\newtheorem{rem}[thm]{Remark}
\newtheorem{example}[thm]{Example}
\newtheorem{prob}[thm]{Problem}
\numberwithin{equation}{section}
\numberwithin{figure}{subsection}

\newcommand{\eps}{\varepsilon}
\newcommand{\To}{\longrightarrow}

\newcommand{\Rr}{\mathbb{R}}
\newcommand{\reals}{\mathbb{R}}
\newcommand{\nats}{\mathbb{N}}
\newcommand{\ints}{\mathbb{Z}}
\newcommand{\crs}{\mathrm{cr}}
\newcommand{\sgn}{\mathrm{sgn~}}
\newcommand{\tw}{\mathrm{tw}}
\newcommand{\lk}{\mathrm{lk}}
\newcommand{\cs}{{ ~ \sharp ~}}

\newcommand{\poscrossing}{\hspace{.01cm} \xy (0,0)*{}="A";
(4,0)*{}="B"; (0,4)*{}="C"; (4,4)*{}="D"; (1.5,1.5)*{}="E";
(2.5,2.5)*{}="F"  ;"A"; "E" **\dir{-}; "F"; "D" **\dir{-}
?>*\dir{>}; "C"; "B"
**\dir{-} ?>*\dir{>};
\endxy \hspace{.2cm}}

\newcommand{\negcrossing}{\hspace{.01cm} \xy (0,0)*{}="A";
(4,0)*{}="B"; (0,4)*{}="C"; (4,4)*{}="D"; (1.5,2.5)*{}="E";
(2.5,1.5)*{}="F"  ;"C"; "E" **\dir{-}; "F"; "B" **\dir{-}
?>*\dir{>}; "A"; "D"
**\dir{-} ?>*\dir{>};
\endxy \hspace{.2cm}}

\newcommand{\negcrossingNoArrow}{\hspace{.01cm} \xy (0,0)*{}="A";
(4,0)*{}="B"; (0,4)*{}="C"; (4,4)*{}="D"; (1.5,2.5)*{}="E";
(2.5,1.5)*{}="F"  ;"C"; "E" **\dir{-}; "F"; "B" **\dir{-} ; "A"; "D"
**\dir{-} ;
\endxy \hspace{.2cm}}

\newcommand{\poscrossingNoArrow}{\hspace{.01cm} \xy (0,0)*{}="A";
(4,0)*{}="B"; (0,4)*{}="C"; (4,4)*{}="D"; (1.5,1.5)*{}="E";
(2.5,2.5)*{}="F"  ;"A"; "E" **\dir{-}; "F"; "D" **\dir{-}; "C"; "B"
**\dir{-};
\endxy \hspace{.2cm}}

\newcommand{\smilefrown}{\hspace{0.01cm} \xy (0,0)*{}="A"; (4,0)*{}="B"; (0,4)*{}="C";
(4,4)*{}="D"; "A";"B" **\crv{(2,2)}; "C";"D"
**\crv{(2,2)};
\endxy \hspace{.2cm}}

\newcommand{\smilefrownSmall}{\hspace{0.01cm} \xy (0,0)*{}="A"; (2,0)*{}="B"; (0,2)*{}="C";
(2,2)*{}="D"; "A";"B" **\crv{(1,1)}; "C";"D"
**\crv{(1,1)};
\endxy \hspace{.2cm}}

\newcommand{\infinityfig}{\hspace{0.01cm} \xy (0,0)*{}="A"; (2,0)*{}="B"; (0,2)*{}="C";
(2,2)*{}="D"; "A";"C" **\crv{(1,1)}; "B";"D"
**\crv{(1,1)};
\endxy \hspace{.2cm}}

\newcommand{\tanglea}{\includegraphics[scale = 0.5]{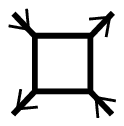}}
\newcommand{\tangleb}{\includegraphics[scale = 0.5]{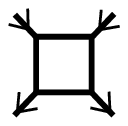}}

\RequirePackage[T1]{fontenc}
\usepackage{ae,aecompl}

\renewcommand{\hrulefill}[1][]{\ifthenelse{\equal{#1}{}}%
  {\leaders\hrule\hfill}%
  {\leaders\hrule height #1\hfill}}

\begin{document}

AMS 57M25  \hrulefill[0pt]{} ISSN 0239-8885

{\center{ \vspace{1cm}  \thispagestyle{empty}

UNIWERSYTET WARSZAWSKI

 INSTYTUT MATEMATYKI

\vspace{2in}

J\'{o}zef H. Przytycki

Survey on recent invariants on classical knot theory

I. Conway Algebras

\vspace{4in}

\address{Department of Mathematics and Computer Science,
Warsaw University\footnote{This address added for e-print.
Actual current address:  George Washington University, Washington,
D.C., \emph{email:} przytyck@gwu.edu} PKiN IXp. 00-901 Warszawa, Poland}

Preprint 6/86

``na prawach r{\c e}kopisu''

\vspace{1cm}

Warszawa 1986

}}
\setcounter{page}{0}
\newpage

\section*{Introduction}
\numberwithin{figure}{section}

The purpose of this survey is to present a new combinatorial
method of constructing invariants of isotopy classes of tame
links. The period of time between the spring of 1984 and the
summer of 1985 was full of discoveries which revolutionized the
knot theory and will have a deep impact on some other branches of
mathematics.  It started by the discovery of Jones of the new
polynomial invariant of links (in May 1984; \cite{Jo_1},
\cite{Jo_2}) and the last big step (which will be described in
this survey) has been made by Kauffman \cite{K_5} in August 1985
when Kauffman applied his method which allowed him to unify almost
all previous work.  This survey is far from being complete, even
if we limit ourselves to the purely combinatorial methods and to
the period May, 1984 -- September, 1985.  (Since then, new
important results have been obtained.)  In particular, we don't
include the very interesting results of Lickorish and Millett
\cite{Li_M_2}, results which link some substitutions in the
Jones-Conway polynomial with old invariants of links.

The survey consists of five parts:  \begin{enumerate}

\item \emph{Diagrams of links and Reidemeister moves.}

This chapter makes the survey almost self-contained and makes it
accessible to non-specialists.

\item \emph{Conway algebras and their invariants of links.}

We consider in this chapter invariants of oriented links which have
the following striking common feature:  If $L_+$, $L_-$, and
$L_\circ$ are diagrams of oriented links which are identical, except
near one crossing point where they are as in Figure \ref{Fig L+ L-
L0}, then the value of the invariant for $L_+$ is uniquely
determined by the values of the invariant for $L_-$ and $L_\circ$,
and the value of the invariant for $L_-$ is uniquely determined by
the values of the invariant for $L_+$ and $L_\circ$.\medskip

\begin{figure}[htbp]

\centering

\includegraphics{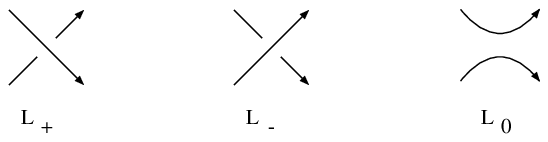}

\caption{} \label{Fig L+ L- L0}

\end{figure}

We construct an abstract algebra (called the Conway algebra) which
formalizes the above approach.

\item \emph{Skein equivalence and properties of the invariants of Conway
type.}

We consider the properties of the invariants under reflection,
mutation, connected and disjoint sums of links.  We analyze closer
the Jones-Conway (HOMFLY) polynomial.

\item \emph{Partial Conway algebras.}

The generalization of Conway algebras, described in this chapter,
allows for the construction of new invariants of links; in
particular, a polynomial of infinitely many variables and a
supersignature.

\item \emph{Kauffman approach.}

The additional diagram $L_\infty$ (Figure \ref{FigLplus_Linfty})
allows the construction of a new link invariant (\cite{B_L_M} and
\cite{Ho}).

\begin{figure}[htbp]

\centering

\includegraphics{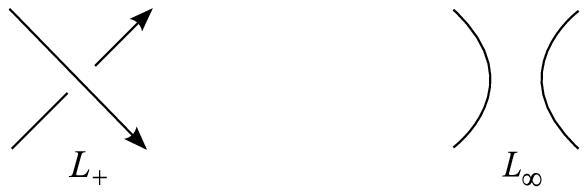}

\caption{} \label{FigLplus_Linfty}

\end{figure}

Kauffman uses regular isotopy of diagrams of links (instead of
isotopy) to build a polynomial invariant of links which generalizes
the previously known one.  We construct an algebraic structure
(Kauffman algebra) which allows us to describe invariants obtained
via the Kauffman method in a unified way.

\end{enumerate}

This survey is a detailed presentation of the eight lectures given
by the author at the University of Zaragoza in February of 1986. The
author would like to thank Jos\'{e} Montesinos for his exceptional
hospitality.

\section{Link diagrams and Reidemeister moves}
\numberwithin{figure}{section}

The classical knot theory studies the position of a circle (knots)
or of several circles (links) in $S^3$ (or $\Rr^3$).  We say that
two links $L_1$ and $L_2$ in $S^3$ are isotopic, written $L_1
\approx L_2$, if there exists an isotopy $F:S^3\times I \To S^3
\times I$ such that $F_0 =\mathrm{Id}$ and $F_1(L_1) = L_2$.  If
the links $L_1$ and $L_2$ are oriented, we assume additionally
that $F_1$ preserves orientations of the links.

We work all the time in the PL category; smooth category could be
considered equally well.  $S^3 = \Rr^3 \cup \infty$ and we can
always assume that the link omits $\infty$.  It is not difficult to
show that two links are isotopic in $\Rr^3$ if and only if they are
isotopic in $S^3$.  Links (up to isotopy) can be represented by
their diagrams on the plane.  Namely, let $p:\Rr^3 \To \Rr^2$ be a
projection and $L \subset \Rr^3$ a link.  A point $P\in p(L) \subset
\Rr^2$ whose preimage, $p^{-1}(P)$, contains more than one point is
called a multiple point.  A projection $p$ is called regular if
\begin{enumerate}

\item There are only finitely many multiple points, and all multiple
points are double points (called crossings), and

\item $P/L : L \to \Rr^2$ is the general position projection (in
some triangulation of $(\Rr^3, L)$ and $\Rr^2$, in which $P/L$ is
simplicial, no double points of $L$ are vertices).

\end{enumerate}

If, for a given regular projection of a link, all over-crossings
(bridges) at every crossing are marked, then the link can be
reconstructed from the projection.   The projection of the link with
just described additional information is called the diagram of the
link.

We call two diagrams equivalent (in oriented or unoriented category)
if they describe isotopic links.  The following theorem of
Reidemeister allows us to work entirely with diagrams.

\begin{thm}Two link diagrams are equivalent if and only if they are
connected by a finite sequence of Reidemeister moves, $\Omega^{\pm
1}_i$ $(i=1,2,3)$ see Figure \ref{Fig Reidemeister}.\end{thm}

\begin{figure}[htbp]

\centering

\includegraphics[width=0.5\textwidth]{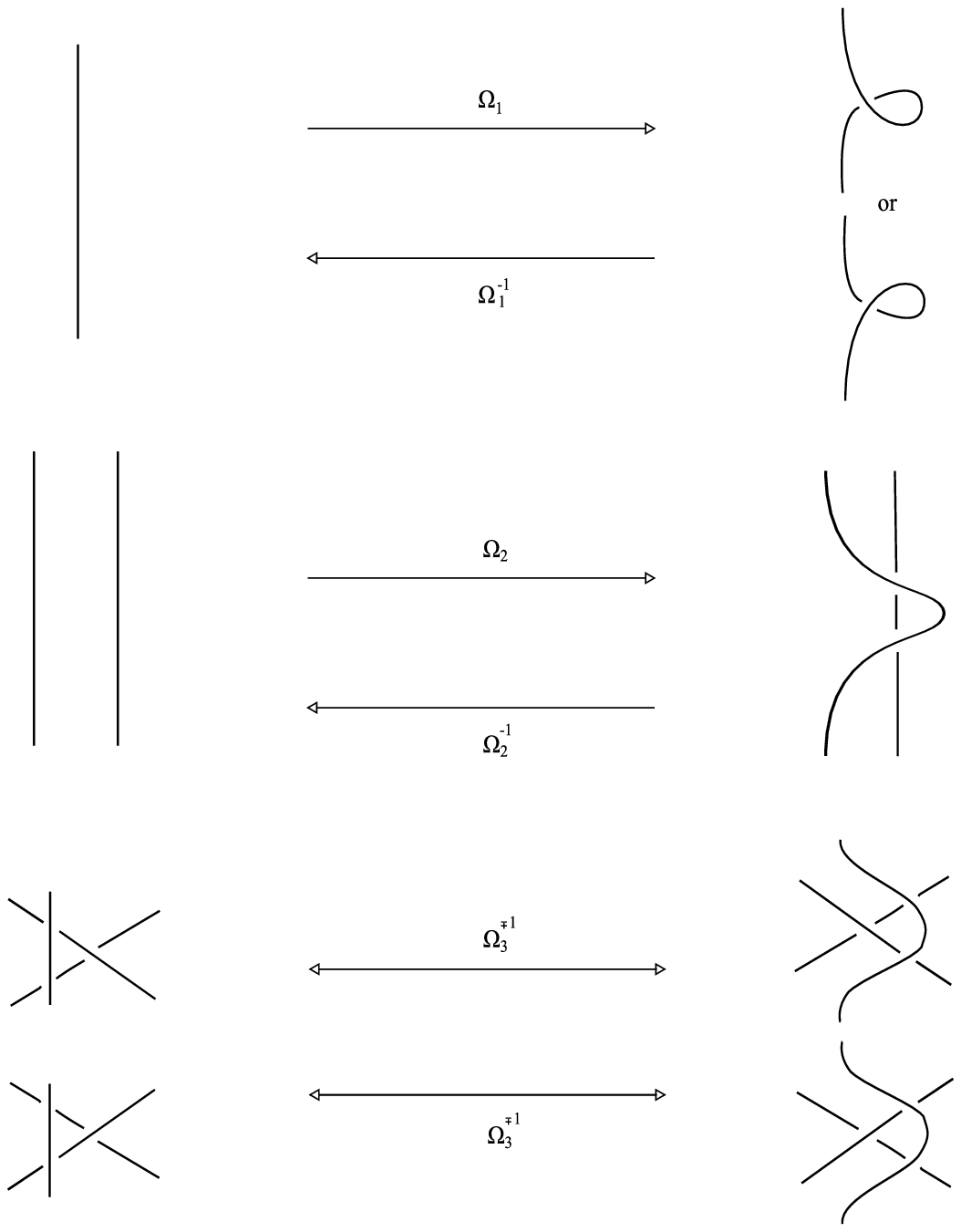}

\caption{} \label{Fig Reidemeister}

\end{figure}
\section{Conway algebras and their invariants of links}
\numberwithin{figure}{subsection}
\numberwithin{equation}{subsection}
\subsection{Conway algebras}\footnote{The equation numbers in this
subsection are off by one from the original.}

Conway \cite{Co}, considering the quick methods of computing the
Alexander polynomial of links, suggested a special normed form of
it (which we call the Conway polynomial) and he showed that the
Conway polynomial, $\nabla_L(z)$, satisfies \begin{enumerate}

\item\label{CA cond 1} $\nabla_{T_1}(z)=1$, where $T_1$ is the
trivial knot

\item\label{CA cond 2}
$\nabla_{L_+}(z)-\nabla_{L_-}(z)=z\nabla_{L_\circ}(z)$, where
$L_+$, $L_-$, and $L_\circ$ are diagrams of oriented links which
are identical, except near one crossing point, where they look
like in Figure \ref{Fig 2.1.0}\footnote{The remaining figures in
this section have their number shifted by one from the original,
where there were two Figures 2.1.1.}.

\end{enumerate}

\begin{figure}[htbp]

\centering

\includegraphics[width=0.5\textwidth]{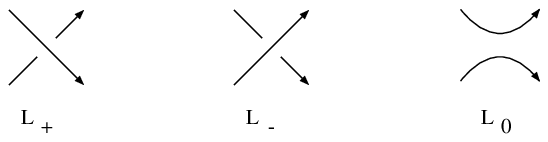}

\caption{} \label{Fig 2.1.0}

\end{figure}

Furthermore, conditions (\ref{CA cond 1}) and (\ref{CA cond 2})
define uniquely $\nabla_L(z)$; \cite{Co}, \cite{K_2}, \cite{Gi},
\cite{B_M}.

At spring of 1984, V.~Jones \cite{Jo_1}, \cite{Jo_2}has shown that
there exists an invariant of oriented links which is a Laurent
polynomial of $\sqrt{t}$ and which satisfies:
\begin{enumerate}

\item $V_{T_1}(t)=1$, and

\item $-tV_{L_+} + \frac{1}{t} V_{L_-} = \left(
\sqrt{t}-\frac{1}{\sqrt{t}}\right) V_{L_\circ}(t).$

\end{enumerate}

It was an immediate idea, after these two examples, that there
exists an invariant of isotopy of oriented links which is a Laurent
polynomial of 2-variables $(P_L(x,y))$ which satisfies:

\begin{enumerate}

\item $P_{T_1}(x,y)=1$, and

\item $xP_{L_+}(x,y) + yP_{L_-}(x,y)=P_{L_\circ}(x,y)$.

\end{enumerate}

In fact, such an invariant exists and it was discovered four
months after the Jones polynomial (in September of 1984) by four
groups of researchers:  R.~Lickorish and K.~Millett, J.~Hoste,
A.~Ocneanu, P.~Freyd, and D.~Yetter \cite{F_Y_H_L_M_O} (and
independently in early December of 1984 by J.~Przytycki and
P.~Traczyk \cite{P_T_1}). We will call this polynomial the
Jones-Conway polynomial or HOMFLY polynomial (after initials of
the inventors of it).

Instead of looking for polynomial invariants of links related to
Figure 2.1.1, we can consider a more general point of view.
Namely, we can look for general invariants of links which have the
following common feature: $w_{L_+}$ is uniquely determined by
$w_{L_-}$ and $w_{L_\circ}$, and also $w_{L_-}$ is uniquely
determined by $w_{L_+}$ and $w_{L_\circ}$.  Invariants with this
property we call Conway type invariants.  Now we will develop this
idea based mainly on the paper \cite{P_T_1}.

Consider the following general situation.  Assume we are given a
set $A$ (called universum) with a sequence of fixed elements,
$a_1, a_2, \ldots$ (i.e. a function $f:\nats \to A$) and two
$2$-argument operations, $|$ and $\ast$, each mapping $A\times A$
into $A$.  (That is, we have an algebra $\mathcal{A} = (A;a_1,
a_2, \ldots, |,\ast)$.)  We would like to construct invariants of
oriented links satisfying the conditions:

$$ \begin{array}{l} w_{L_+} = w_{L_-} | w_{L_\circ} \textrm{,} \\
w_{L_-} = w_{L_+} \ast w_{L_\circ} \textrm{, and} \\ w_{T_n} = a_n
\textrm{, where } T_n \textrm{ is the trivial link with }n\textrm{
components.}
\end{array}$$

\begin{defn} We say that $\mathcal{A} = (A; a_1, a_2, \ldots, |,
\ast)$ is a Conway algebra if the following conditions are
satisfied: \begin{enumerate}

\item[C1.]  $a_n | a_{n+1} = a_n$

\item[C2.]  $a_n \ast a_{n+1} = a_n$

\item[C3.] $(a|b)|(c|d)=(a|c)|(b|d)$

\item[C4.] $(a|b)\ast (c|d) = (a\ast c)|(b\ast d)$

\item[C5.] $(a \ast b)\ast (c\ast d) = (a\ast c)\ast(b\ast d)$

\item[C6.] $(a|b)\ast b =a$

\item[C7.] $(a\ast b)|b =a$

\end{enumerate}

Note that C3 through C5 are transposition properties.

\end{defn}

The following is the main theorem of \cite{P_T_1}.

\begin{thm}\label{Theorem 2.1.1}\footnote{In the original, this
Theorem was mistakenly numbered 2.1.1, but referenced as 2.1.2
throughout.} For a given Conway algebra $\mathcal{A}$, there
exists a uniquely determined invariant, $w$, which attaches an
element $w_L$ from $A$ to every isotopy class of oriented links
and satisfies the conditions \begin{enumerate}

\item $\begin{array}{l}w_{T_n} = a_n\end{array} \textrm{(initial
conditions)}$ \vspace{.3cm}

\item $\begin{array}{l} w_{L_+} = w_{L_-} | w_{L_\circ} \\ w_{L_-} =
w_{L_+} \ast w_{L_\circ} \end{array}  \textrm{(Conway relations)}$

\end{enumerate}

\end{thm}

Before we give the proof, let us write here a few words about the
geometrical meaning of the axioms C1--C7 of Conway algebra.
Relations C1 and C2 are introduced to reflect the following
geometrical relations between the diagrams of trivial links of $n$
and $n+1$ components:

\begin{figure}[htbp]
\centering
\includegraphics{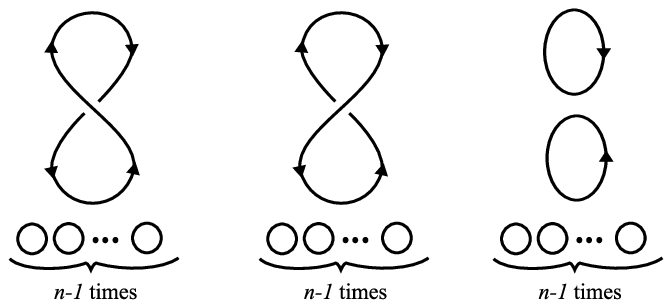}

\caption{}\label{Fig 2.1.1}
\end{figure}

Relations C3--C5 arise when considering rearranging a link at two
crossings of the diagram, but in different order.  It will be
explained in the proof of Theorem \ref{Theorem 2.1.1}.  Relations C6
and C7 reflect the fact that we need the operations $|$ and $\ast$
to be in some respects opposite to one another.

Before giving examples (models) of Conway algebra and sketching
the proof of Theorem \ref{Theorem 2.1.1}, we will show some
elementary properties of Conway algebra.  We have introduced in
the definition seven conditions mainly because of esthetic and
practical reasons (the roles of operations $|$ and $\ast$ are
equivalent). These conditions, however, often depend on one
another:

\begin{lem}\label{Lemma 2.1.3}We have the following dependencies among axioms C1--C7:

\renewcommand{\theenumi}{(\alph{enumi})}
\renewcommand{\labelenumi}{\theenumi}

\begin{enumerate}
\item C1 and C6 $\Rightarrow$ C2
\item C2 and C7 $\Rightarrow$ C1
\item C6 and C4 $\Rightarrow$ C7
\item C7 and C4 $\Rightarrow$ C6
\item C6 and C4 $\Rightarrow$ C5
\item C7 and C4 $\Rightarrow$ C3
\item C5, C6 and C7 $\Rightarrow$ C4
\item C3, C6, and C7 $\Rightarrow$ C4
\end{enumerate}
\end{lem}

\begin{proof}We will prove, as examples, (a), (c), (e), and (g).
\begin{enumerate}
\item[(a)]  $\textrm{C1} \iff a_n | a_{n+1} = a_n \implies
(a_n|a_{n+1})\ast a_{n+1} = a_n \ast a_{n+1}
\overset{\textrm{C6}}{\implies} \newline
\overset{\textrm{C6}}{\implies} a_n=a_n\ast a_{n+1} \iff
\textrm{C2}$.\medskip

\item[(c)] $\textrm{C6} \implies (a|(b|a))\ast(b|a) = a
\overset{\textrm{C4}}{\iff} (a\ast b)|((b|a)\ast a)=a
\overset{\textrm{1.6}}{\implies} \newline
\overset{\textrm{1.6}}{\implies} (a\ast b)|b = a \iff
\textrm{C7}$.\medskip

\item[(e), (g)] $$\begin{array}{lcccc} \textrm{C5} & \iff & (a\ast b)
\ast (c\ast d) & = & (a\ast c) \ast (b \ast d)\bigskip \\ & &
\Downarrow \textrm{C7} & & \textrm{C6} \Uparrow \bigskip \\ & &
(a\ast b) & = &((a\ast c)\ast(b\ast d))|(c\ast d) \bigskip \\ & & &
\Updownarrow \textrm{C7} & \textrm{(or C6 and C4 by (c))} \bigskip \\
\textrm{C4}& \implies & ((a\ast c)|c)\ast((b\ast d)|d) & = & ((a
\ast c)\ast(b\ast d))|(c\ast d) \bigskip \\  & & \textrm{substitute
} a=x|c \textrm{ and } b=y|d &  \Downarrow \textrm{C6} & \bigskip \\
& & (x|c) \ast (y|d) & = & (x\ast y) |(c\ast d) \bigskip \\ & & &
\Updownarrow \bigskip \\ & & & \textrm{C4}
\end{array}$$

\end{enumerate}

\end{proof}

\begin{lem}\label{Lemma 2.1.4} Let us define in each Conway algebra $\mathcal{A}$ and
for each $b\in A$ and action $|_b$ (respectively, $\ast_b$)$:A\to
A$ defined by $|_b (a) = a|b$ (respectively, $\ast_b(a) = a \ast
b$). Then $|_b$ and $\ast_b$ are bijections on $A$. Furthermore,
$|_b$ and $\ast_b$ are inverses one to another, i.e. $|_b \ast_b =
\ast_b |_b = \mathrm{Id}$.
\end{lem}

Lemma \ref{Lemma 2.1.4} follows from conditions C6 and C7.  Now we
will describe some examples of Conway algebras.

\begin{example}[Number of components]\label{Example 2.1.5} Set $A=\nats$ (the set of
natural numbers), $a_i = i$, and $i|j = i\ast j = i$.  Verification
of conditions C1--C7 is immediate (the first letter of each side of
every relation is the same).  This algebra yields the number of
components of the link.
\end{example}

\begin{example}\label{Example 2.1.6}
Set $A=\{0, 1, 2\}$, $a_i \equiv i\mod 3$, the operation $\ast$ is
equal to $|$, and $|$ is given by the following table:$$
\begin{tabular}{c|ccc} $|$ & 0 & 1 & 2  \\ \hline 0 & 1&0&2 \\ 1&0&2&1
\\ 2 &2 &1&0 \end{tabular}$$ The invariant defined by this algebra
distinguishes, for example, the trefoil knot from the trivial knot
(Figure \ref{Fig 2.1.3}).

\begin{figure}[htbp]

\centering
\begin{tabular}{m{3cm}m{1cm}m{3cm}}
\includegraphics[totalheight=2cm]{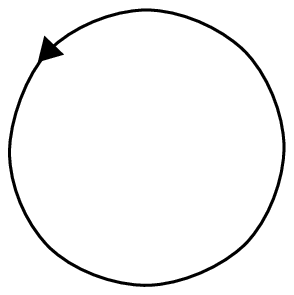} & $\not\approx$ &
  \includegraphics[totalheight=2cm]{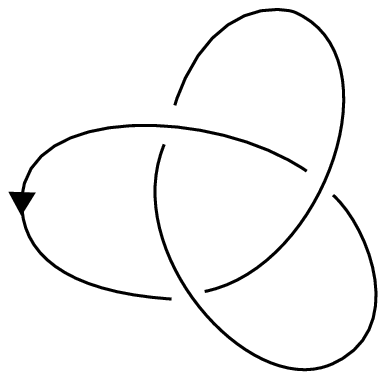} \end{tabular}

\caption{For the trivial knot (on the left), the value of the
invariant is given by $a_1  = 1$. The value of the invariant from
Example \ref{Example 2.1.6} for the right-handed trefoil (on the
right) is $a_1 | ( a_2 | a_1) = 2$.  We write $\not\approx$ for
``not isotopic.''} \label{Fig 2.1.3}

\end{figure}
\end{example}

\begin{example}\label{Example 2.1.7}Set $A= \{1,2,3,4\}$, $a_1 = 1$, $a_2 = 2$, $a_3=4$,
$a_4=1$, $a_5 = 2$, $a_6 = 4$, $\ldots$~.  Operations $|$ and $\ast$
are given by the following tables:  $$\begin{array}{lll}
\begin{tabular}{c|cccc} $|$ & 1 &2&3&4 \\ \hline 1&2&1&4&3 \\
2&3&4&1&2 \\ 3&1&2&3&4 \\ 4&4&3&2&1
\end{tabular} & & \begin{tabular}{c|cccc} $\ast$ &1&2&3&4  \\ \hline
1&3&1&2&4 \\ 2&1&3&4&2 \\ 3&2&4&3&1 \\
4&4&2&1&3\end{tabular}\end{array}$$

The invariant defined by this algebra distinguishes the right-handed
trefoil knot from the left-handed trefoil (Figure \ref{Fig 2.1.4}).

\begin{figure}[htbp]

\centering
\begin{tabular}{m{3cm}m{1cm}m{3cm}}
\includegraphics[totalheight=2cm]{trefoilRight.eps} & $\not\approx$ &
 \includegraphics[totalheight=2cm]{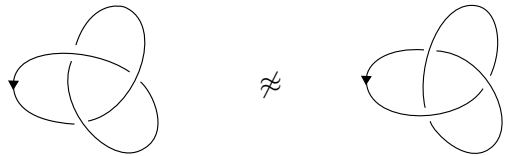}\end{tabular}

\caption{For the right-handed trefoil (on the left) the value of the
invariant is given by $a_1 | (a_2 | a_1) = 4$. The value of the
invariant for the left-handed trefoil (on the right) is $a_1 \ast (
a_2 \ast a_1) = 3$.} \label{Fig 2.1.4}

\end{figure}

\end{example}

T.~Przytycka has found (using a computer) all Conway algebras with
no more than five elements.  If we limit fixed elements to $a_1$ and
$a_2$ and assume $a_1 = 1$, $a_2 = 2$ then we get (up to
isomorphism):

\begin{center}\begin{tabular}{c|c} Number of elements & Number of algebras \\
\hline 2&2 \\ 3& 9 \\ 4& 51 \\ 5 & 204 \end{tabular}\end{center}

\begin{example}[Jones-Conway polynomial]\label{Example 2.1.8}
$A=\ints[x^{\mp}, y^{\mp}]$,  $a_1=1, a_2 = x+y, \ldots, a_i =
(x+y)^{i-1}, \ldots$~.

We define $|$ and $\ast$ as follows:  $w_2 | w_0 = w_1$ and $w_1
\ast w_0 = w_2$, where polynomials $w_1$, $w_2$, and $w_0$ satisfy
the following equation: \begin{equation}\label{Equation 2.1.9}
x w_1 + y w_2 = w_0 \end{equation}

The invariant of knots given by this algebra is the Jones-Conway
polynomial mentioned at the beginning of this part.  In particular,
if we substitute $x=\frac{1}{z}$ and $y=-\frac{1}{z}$ we get the
Conway polynomial, and after the substitution
$x=\frac{-t}{\sqrt{t}-\frac{1}{\sqrt{t}}}$, $y =
\frac{1}{t}\frac{1}{\sqrt{t}-\frac{1}{\sqrt{t}}}$, we get the Jones
polynomial.
\end{example}

Now we will show that the algebra from Example \ref{Example 2.1.8}
is in fact a Conway algebra.

The conditions C1 and C2 follow from the equality $$x(x+y)^{n-1} +
y(x+y)^{n-1} = (x+y)^n.$$

The conditions C6 and C7 follow from the fact that the actions
were defined using one linear equation.  It remains to show C3 (C4
and C5 will follow then by Lemma \ref{Lemma 2.1.3}).  We have from
the definition $$(a|b)|(c|d) = \frac{1}{x}((c|d)-y(a|b)) =
\frac{1}{x}\left(\frac{1}{x}(d-yc) - y\frac{1}{x}(b-ya)\right) =
\frac{1}{x^2}d - \frac{y}{x^2}c - \frac{y}{x^2}b +
\frac{y^2}{x^2}a.$$ Because coefficients of $b$ and $c$ are the
same, if we change the places of $b$ and $c$, the value of the
expression is not changed, which proves condition C3.

It is possible to generalize the algebra of Example \ref{Example
2.1.8} by introducing the new variable $z$ and considering instead
of \ref{Equation 2.1.9}, the equation $$xw_1 +yw_2 = w_0-z.$$
However, one does not get any stronger invariant than the
Jones-Conway polynomial (see Proposition 3.38).

\setcounter{thm}{9}

\begin{example}[Global linking number]\label{Example 2.1.10} Set $A = \nats \times \ints$,
$a_i=(i,0)$, and $$\begin{array}{l} (a,b)|(c,d) =
\left\{\begin{array}{ll} (a, b+1) & \textrm{if }a>c \\ (a,b) &
\textrm{if } a\leq c \end{array} \right. \\ (a,b)\ast (c,d) =
\left\{
\begin{array}{ll}(a, b-1) & \textrm{if } a>c \\ (a,b) & \textrm{if
}a \leq c \end{array}\right.\end{array}$$  The invariant associated
to a link is a pair:  (number of components, global linking number).
\end{example}

It is an easy exercise to read the global linking number from the
diagram.  Namely, call a crossing of type \poscrossing positive
and a crossing of type \negcrossing negative. We will write $\sgn
p = +$ or $-$ depending on whether the crossing $p$ is positive or
negative.  Then for a given diagram $D$, $\mathrm{lk}(D) =
\frac{1}{2} \sum \sgn p_i$, where summation is taken over all
crossings between different components of $D$, is equal to the
global index number of $D$.

Now we will show that the algebra from Example \ref{Example 2.1.10}
is in fact a Conway algebra.  The proof of conditions C1, C2, C6,
and C7 is very easy and we omit it.  We consider condition C3 in
more detail.  From the definition of the operation $|$ we have
$$((a_1,a_2)|(b_1,b_2))|((c_1,c_2)|(d_1,d_2)) =
\left\{\begin{array}{ll}(a_1, a_2+2) & \textrm{if } a_1> b_1
\textrm{ and } a_1>c_1 \medskip \\ (a_1, a_2+1) & \textrm{if
}(a_1>b_1 \textrm{ and } a_1 \leq c_1)\\ & \textrm{ or } (a_1 \leq
b_1 \textrm{ and } a_1>c_1) \medskip \\ (a_1, a_2) & \textrm{if }
a_1 \leq b_1, a_1\leq c_1\end{array}\right.$$

If we exchange the places of $b_i$ and $c_i$ then we get the same
result, so C3 is satisfied.

We will write $L^p_+$, $L^p_-$, and $L^p_\circ$ if we need the
crossing point $p$ to be explicitly specified.

\begin{defn}Let $T$ be a binary tree each of whose vertices
represents a link (trivial links at leaves) in such a way that the
situation at each vertex (except at leaves) looks like:

\begin{figure}[htbp]

\centering

$\xymatrix { & L_+ \ar@{-}[ddl] \ar@{-}[ddr]& & & & L_- \ar@{-}[ddl]
\ar@{-}[ddr]& \\ & & &  \textrm{or} & & & \\ L_- & & L_\circ & & L_+
& & L_\circ}$

\caption{} \label{Fig 2.1.5}

\end{figure}
\end{defn}

In a natural way it yields a binary tree with $a_i$'s at leaves and
$+$'s or $-$'s at other vertices.  We will call it a resolving tree
of the root link.

\begin{figure}[htbp]

\centering
\includegraphics[width=0.15\textwidth]{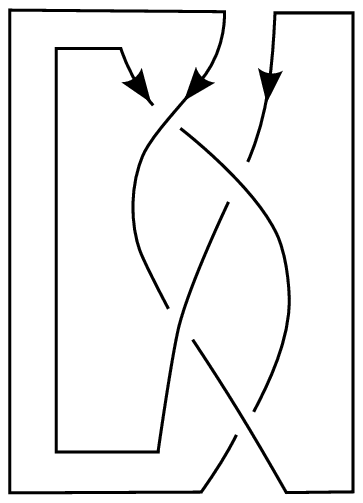}
\caption{}\label{Fig 2.1.6}

\end{figure}

\begin{example}Let $L$ be the figure-eight knot represented by the
diagram in Figure \ref{Fig 2.1.6}.

To determine $w_L$, let us consider the binary tree in Figure
\ref{Fig 2.1.7}.

As is easily seen, the leaves of the tree are trivial links and
every branching reflects a certain operation on the diagram at the
marked crossing point.  To compute $w_L$, it is enough to have the
resolving tree shown in Figure \ref{Fig 2.1.8}.

Here the sign indicates the sign of the crossing point at which the
operation was performed, and the leaf entries are the values of $w$
for the resulting trivial links.  Now we may conclude that $w_L =
a_1|(a_2\ast a_1)$.
\end{example}

\begin{figure}[htbp]

\centering
\includegraphics[width=0.4\textwidth]{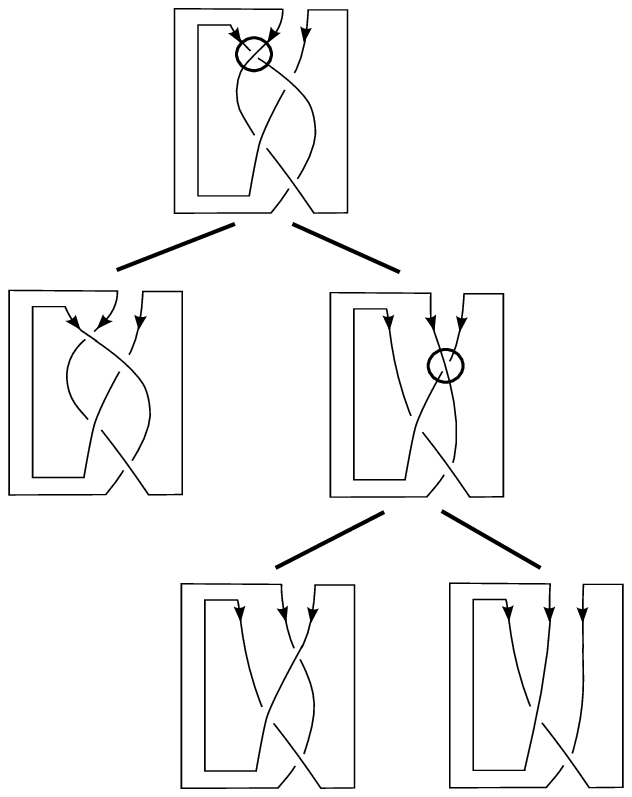}
\caption{}\label{Fig 2.1.7}

\end{figure}

\begin{figure}[htbp]

\centering

$\xymatrix { & + \ar@{-}[dl] \ar@{-}[dr]& & \\ a_1 & & - \ar@{-}[dl]
\ar@{-}[dr]& \\ & a_2 & & a_1 }$

\caption{} \label{Fig 2.1.8}

\end{figure}

There exists a standard procedure to obtain a resolving tree of a
given diagram.  It will be described in the next paragraph and it
will play an essential role in the proof of Theorem \ref{Theorem
2.1.1}.

\subsection{Proof of Theorem \ref{Theorem 2.1.1}}

\setcounter{equation}{1}

\begin{defn} Let $L$ be an oriented diagram of $n$ components and
let $b=(b_1, \ldots, b_n)$ be base points of $L$, one point from
each component of $L$, but not the crossing points.  Then we say
that $L$ is descending with respect to $b$ if the following holds:
If one travels along $L$ (according to the orientation of $L$)
starting from $b_1$, then after having returned to $b_1$ -- from
$b_2, \ldots, $ finally from $b_n$, then each crossing which is
met for the first time is crossed by an over-crossing
(bridge).\end{defn}

It is easily seen that for every diagram $L$ of an oriented link
there exists a resolving tree such that the leaf diagrams are
descending (with respect to appropriately chosen base points).  This
is obvious for diagrams with no crossings at all, and once it is
known for diagrams with less than $n$ crossings we can use the
following procedure for any diagram with $n$ crossings:  Choose base
points arbitrarily and start walking along the diagram until the
first ``bad'' crossing $p$ is met, i.e. the first crossing which is
crossed by an under-crossing (tunnel) when first met.  Then begin to
construct the tree changing the diagram in this point.  If for
example $\sgn p = +$, we get $$\xymatrix { & L=L_+^p \ar@{-}[dl] \ar@{-}[dr] & \\
L_-^p & & L_\circ^p }$$  Then we can apply the inductive hypothesis
to $L^p_\circ$ and we can continue the procedure with $L^p_-$
(walking further along the diagram and looking for the next bad
point).

To prove Theorem \ref{Theorem 2.1.1}, we will construct the function
$w$ as defined on diagrams.  In order to show that $w$ is an
invariant of isotopy classes of oriented links we will verify that
$w$ is preserved by the Reidemeister moves.

We use induction on the number $\mathrm{cr}(L)$ of crossing points
in the diagram.  For each $k \geq 0$ we define a function $w_k$
assigning an element of $A$ to each diagram of an oriented link with
no more than $k$ crossings.  Then $w$ will be defined for every
diagram by $w_L = w_k(L)$ where $k\geq \mathrm{cr}(L)$.  Of course
the function $w_k$ must satisfy certain coherence conditions for
this to work.  Finally, we will obtain the required properties of
$w$ from the properties of the $w_k$'s.

We begin from the definition of $w_0$.  For a diagram $L$ of $n$
components with $\crs(L) = 0$, we put \begin{equation}\label{Eqn
2.2.2} w_0(L) = a_n.\end{equation} To define $w_{k+1}$ and prove
its properties we will use the induction several times.  To avoid
misunderstandings, the following will be called the Main Inductive
Hypothesis (M.I.H.):  We assume that we have already defined a
function $w_k$ attaching an element of $A$ to each diagram $L$ for
which $\crs(L) \leq k$.  We assume that $w_k$ has the following
properties: \begin{equation}\label{Eqn 2.2.3} w_k(U_n) =
a_n\end{equation} for $U_n$ being a descending diagram of $n$
components (with respect to some choice of base points).
\begin{eqnarray} \label{Eqn 2.2.4} w_k(L_+) = w_k(L_-)|w_k(L_\circ) \\ \label{Eqn 2.2.5} w_k(L_-) =
w_k(L_+)\ast w_k(L_\circ),\end{eqnarray} for $L_+$, $L_-$, and
$L_\circ$ related as usually, and \begin{equation}\label{Eqn
2.2.6} w_k(L) = w_k(R(L)), \end{equation} where $R$ is a
Reidemeister move on $L$ such that $\crs(R(L))$ is still at most
$k$.

Then, as any reader may expect, we want to make the Main Inductive
Step (M.I.S.) to obtain the existence of a function $w_{k+1}$ with
analogous properties defined on diagrams with at most $k+1$
crossings.

Before dealing with the task of making the M.I.S.~let us explain
that it will really end the proof of the theorem.  It is clear
that the function $w_k$ satisfying the M.I.H.~is uniquely
determined by properties \ref{Eqn 2.2.3}, \ref{Eqn 2.2.4},
\ref{Eqn 2.2.5}, and the fact that for every diagram there exists
a resolving tree with descending leaf diagrams.  Thus the
compatibility of the functions $w_k$ is obvious and they define a
function $w$ on the diagram.

The function $w$ satisfies the conditions in (2) of Theorem
\ref{Theorem 2.1.1} because the functions $w_k$ satisfy such
conditions.

If $R$ is a Reidemeister move on a diagram $L$, then $\crs(R(L))$
equals at most $k=\crs(L)+2$. Whence $w_{R(L)} = w_k(R(L))$,
$w_L=w_k(L)$, and by properties of $w_k$, $w_k(L) = w_k(R(L))$,
which implies $w_{R(L)} = w_L$.  It follows that $w$ is an
invariant of the isotopy class of oriented links.

Now it is clear that $w$ has the required property (1) too, since
there is a descending diagram $U_n$ in the same isotopy class as
$T_n$ and we have $w_k(U_n) = a_n$.

The rest of this section will be occupied by the M.I.S.  For a given
diagram $D$ with $\crs(D) \leq k+1$ we will denote by $\mathcal D$
the set of diagrams which are obtained from $D$ by operations of the
kind \negcrossingNoArrow ${}^\rightarrow$ \poscrossingNoArrow or
\negcrossingNoArrow  ${}^\rightarrow$ \smilefrown.

Of course, once base points $b=(b_1, \ldots, b_n)$ are chosen for
$D$, then the same points can be chosen as base points for any $L\in
\mathcal D$, provided $L$ is obtained from $D$ by the operations of
the first type only.

Let us define $w_b$ for a given $D$ and $b$, assigning an element
$A$ to each $L\in \mathcal D$.

If $\crs(L) <k+1$ we put \begin{equation}\label{Eqn 2.2.7} w_b(L) =
w_k(L).\end{equation}  If $U_n$ is a descending diagram with respect
to $b$ we put \begin{equation}\label{Eqn 2.2.8} w_b(U_n) = a_n,
\end{equation} ($n$ denotes the number of components).

Now we proceed by induction on the number $b(L)$ of bad crossings
in $L$ (in the symbol $b(L)$, $b$ works simultaneously for ``bad''
and for $b=(b_1, \ldots, b_n)$.  For a different choice of base
points $b' = (b'_1, \ldots, b'_n)$ we will write $b'(L)$.) Assume
that $w_b$ is defined for all $L\in \mathcal D$ such that
$b(L)<t$.  Then for $L$, $b(L) = t$, let $p$ be the first bad
crossing of $L$ (starting from $b_1$ and walking along the
diagram). Depending on $p$ being positive or negative, we have $L
= L_+^p$ or $L=L_-^p$.

We put \begin{equation}\label{Eqn 2.2.9} w_b(L) = \left\{
\begin{array}{ll} w_b(L^p_-)|w_b(L^p_\circ) & \textrm{if }\sgn p = +
\\ & \\ w_b(L^p_+)\ast w_b(L^p_\circ) & \textrm{if }\sgn p =
-\end{array}\right.\end{equation}

We will show that $w_b$ is in fact independent of the choice of $b$
and that it has the properties required from $w_{k+1}$.

\subsubsection*{\underline{Conway relations for $w_b$}}

\ \smallskip

Let us begin with the proof that $w_b$ has properties \ref{Eqn
2.2.4} and \ref{Eqn 2.2.5}.  We will denote by $p$ the considered
crossing point.  We restrict our attention to the case when
$b(L_+^p)>b(L_-^p)$.  The opposite situation is quite analogous.

Now, we use induction on $b(L_-^p)$.  If $b(L_-^p)=0$, then
$b(L_+^p) = 1$, $p$ is the only bad point of $L_+^p$, and by the
defining Equalities \ref{Eqn 2.2.9}, we have
$$w_b(L_+^p)=w_b(L_-^p)|w_b(L_\circ^p),$$ and using C6 we obtain
$$w_b(L_-^p)=w_b(L_+^p) \ast w_b(L_\circ^p).$$

Assume now that the formulae \ref{Eqn 2.2.4} and \ref{Eqn 2.2.5} for
$w_b$ are satisfied for every diagram $L$ such that $b(L_-^p)<t$,
$t\geq 1$.  Let us consider the case $b(L_-^p)=t$.

By the assumption $b(L_+^p)\geq 2$.  Let $q$ be the first bad point
on $L_+^p$.  Assume that $q=p$.  Then by \ref{Eqn 2.2.9} we have
$$w_b(L_+^p)=w_b(L_-^p)|w_b(L_\circ^p).$$

Assume $q\neq p$. Let $\sgn q=+$, for example.  Then by \ref{Eqn
2.2.9} we have
$$w_b(L^p_+)=w_b(L^{pq}_{++})=w_b(L^{pq}_{+-})|w_b(L^{pq}_{+\circ}).$$
But $b(L^{pq}_{--})<t$ and $\crs(L^{pq}_{+\circ})\leq k$, whence by
the inductive hypothesis and M.I.H.~we have
$$w_b(L^{pq}_{+-})=w_b(L^{pq}_{--})|w_b(L^{pq}_{\circ -}), \textrm{
and}$$
$$w_b(L^{pq}_{+\circ})=w_b(L^{pq}_{-\circ})|w_b(L^{pq}_{\circ \circ}),$$
whence
$$ w_b(L^{p}_{+})=\left(w_b(L^{pq}_{--})|w_b(L^{pq}_{\circ
-})\right)|\left(w_b(L^{pq}_{-\circ})|w_b(L^{pq}_{\circ\circ})\right),$$
and by the transposition property C3 \begin{equation}\label{Eqn
2.2.10}
w_b(L^{p}_{+})=\left(w_b(L^{pq}_{--})|w_b(L^{pq}_{-\circ})\right)|\left(w_b(L^{pq}_{\circ
-})|w_b(L^{pq}_{\circ\circ})\right).\end{equation}

On the other hand, $b(L^{pq}_{--})<t$ and $\crs(L_\circ^p)\leq k$,
so using once more the inductive hypothesis and M.I.H.~we obtain
\begin{equation}\label{Eqn 2.2.11}\begin{array}{l}w_b(L^p_-)=w_b(L^{pq}_{-+})=w_b(L^{pq}_{--})|w_b(L^{pq}_{-\circ})
\\ \ \\ w_b(L^p_\circ)=w_b(L^{pq}_{\circ +})=w_b(L^{pq}_{\circ
-})|w_b(L^{pq}_{\circ\circ}).\end{array}\end{equation}

Putting \ref{Eqn 2.2.10} and \ref{Eqn 2.2.11} together, we obtain
$$w_b(L^p_+)=w_b(L^p_-)|w_b(L^p_\circ)$$ as required.  If $\sgn q =
-$, we use C4 instead of C3.  This completes the proof of Conway
relations for $w_b$.

\subsubsection*{\underline{Changing base points}}

\ \smallskip

We will show now that $w_b$ does not depend on the choice of $b$,
provided the order of components is not changed.  It amounts to
the verification that we may replace $b_i$ by $b'_i$ taken from
the same component in such a way that $b'_i$ lies after $b_i$ and
there is exactly one crossing point, say $p$, between $b_i$ and
$b'_i$.  Let $b' = (b_1, \ldots, b'_i, \ldots, b_n)$.  We want to
show that $w_b(L) = w_{b'}(L)$ for every diagram with $k+1$
crossings belonging to $\mathcal D$.  We will only consider the
case $\sgn p = + $; the case $\sgn p = -$ is quite analogous.

We use induction on $B(L)=\max(b(L), b'(L))$.  We consider three
cases. \medskip

\noindent\underline{CBP 1.}  Assume $B(L)=0$.  Then $L$ is
descending with respect to both choices of base points and by
\ref{Eqn 2.2.8},
$$w_b(L) = a_n= w_{b'}(L).$$  \medskip

\noindent\underline{CBP 2.}  Assume that $B(L)=1$ and $b(L) \neq
b'(L)$. This is possible only when $p$ is a self-crossing point of
the $i$th component of $L$.  There are two subcases to be
considered.\medskip

CBP 2(a):  \ $b(L) = 1$ and $b'(L)=0$.  Then $L$ is descending
with respect to $b'$ and by \ref{Eqn 2.2.8},
$$w_{b'}(L)=a_n, \textrm{ and}$$
$$w_b(L)=w_b(L^p_+)=w_b(L^p_-)|w_b(L^p_\circ).$$
Again, we have restricted our attention to the case $\sgn p=+$. Now,
$w_b(L_-^p)=a_n$ since $b(L_-^p)=0$, and $L^p_\circ$ is descending
with respect to a proper choice of base points.  Of course,
$L^p_\circ$ has $n+1$ components, so $w_b(L_\circ^p)=a_{n+1}$ by
\ref{Eqn 2.2.8}.

It follows that $w_b(L) = a_n|a_{n+1}$ and $a_n|a_{n+1} = a_n$ by
C1.\medskip

CBP 2(b): $b(L)=0$ and $b'(L)=1$.  This case can be dealt with
like CBP 2(a).\medskip

\noindent\underline{CBP 3.}  $B(L) = t>1$ or $B(L)=1=b(L)=b'(L)$.

We assume by induction $w_b(K) = w_{b'}(K)$ for $B(K)<B(L)$.  Let
$q$ be a crossing point which is bad with respect to $b$ and $b'$ as
well.  We will consider this time the case $\sgn q =-$.  The case
$\sgn q = +$ is analogous.

Using the already proven Conway relations for $w_b$ and $w_{b'}$ we
obtain $$w_b(L) = w_b(L^q_-) = w_b(L^q_+) \ast w_b(L^q_\circ),
\textrm{ and }$$  $$ w_{b'} = w_{b'}(L^q_-) = w_{b'}(L^q_+) \ast
w_{b'}(L^q_\circ).$$ But $B(L_+^q)<B(L)$ and $\crs(L_\circ^q)\leq
k$, whence by the inductive hypothesis and M.I.H.~hold $$w_b(L^q_+)
= w_{b'}(L_+^q), \textrm{ and }$$  $$w_b(L^q_\circ) =
w_{b'}(L_\circ^q),$$ which imply $w_b(L) = w_{b'}(L)$.  This
completes the proof of this step (C.B.P.).\medskip

Since $w_b$ turned out to be independent of base point changes which
preserve the order of components, we can now consider defined a
function $w^\circ$ which attaches an element of $A$ to every diagram
$L$, $\crs(L)\leq k+1$ with a fixed ordering of components.

\subsubsection{\underline{Independence of $w^\circ$ of Reidemeister
moves} (I.R.M.)}

\ \smallskip

When $L$ is a diagram with a fixed order of components and $R$ is a
Reidemeister move on $L$, then we have a natural ordering of
components on $R(L)$.  We will show now that $w^\circ(L) =
w^\circ(R(L))$.

Of course we assume that $\crs(L), \crs(R(L)) \leq k+1$.

We use the induction on $b(L)$ with respect to properly chosen
base points $b=(b_1, \ldots, b_n)$.  Of course the choice must be
compatible with the given ordering of components.  We choose the
base points to lie outside the part of the diagram involved in the
considered Reidemeister move $R$, so that the same points may work
for the diagram $R(L)$ as well.  We have to consider the three
standard types of Reidemeister moves (Figure \ref{Fig 2.2.1}).

\begin{figure}[htbp]
\centering
\includegraphics[width = .95\textwidth]{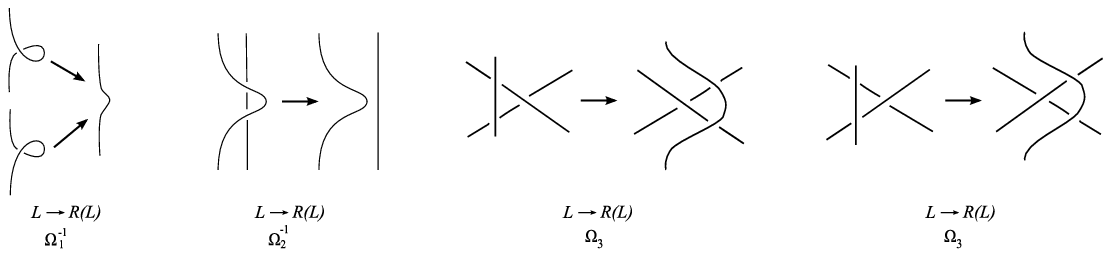}
\caption{}\label{Fig 2.2.1}
\end{figure}

Assume that $b(L)=0$.  Then it is easily seen that also $b(R(L))=0$,
and the number of components is not changed.  Thus, by \ref{Eqn
2.2.8},
$$w^\circ(L) = w^\circ(R(L)).$$

We assume now by induction that $w^\circ(L) = w^\circ(R(L))$ for
$b(L)<t$. Let us consider the case $b(L)=t$.  Assume that there is
a bad crossing $p$ in $L$ which is different from all the
crossings involved in the considered Reidemeister move.  Assume,
for example, that $\sgn p = +$.  Then, by the inductive
hypothesis, we have
\begin{equation}\label{Eqn 2.2.12} w^\circ(L^p_-) =
w^\circ(R(L^p_-)),\end{equation} and by
M.I.H.,\begin{equation}\label{Eqn 2.2.13}w^\circ(L^p_\circ) =
w^\circ(R(L^p_\circ)).\end{equation} Now, by the Conway relation
\ref{Eqn 2.2.4}, which was already verified for $w_0$, we have
$$ w^\circ(L) = w^\circ(L^p_+) = w^\circ(L^p_-)|w^\circ(L^p_\circ)$$ $$ w^\circ(R(L)) = w^\circ(R(L)^p_+) =
w^\circ(R(L)^p_-)|w^\circ(R(L)^p_\circ)$$ whence by \ref{Eqn
2.2.12} and \ref{Eqn 2.2.13} we have $$w^\circ(L) =
w^\circ(R(L)).$$ Obviously $R(L^p_-) = R(L)^p_-$ and $R(L^p_\circ)
= R(L)^p_\circ$.

It remains to consider the case when $L$ has no bad points, except
those involved in the considered Reidemeister move.  We will
consider the three types of moves separately. The most complicated
is the case of a Reidemeister move of the third type.  To deal
with it, let us formulate the following observation:

Whatever the choice of base points is, the crossing point of the
top arc and the bottom arc cannot be the only bad point of the
diagram.

\begin{figure}[htbp]

\centering
\includegraphics{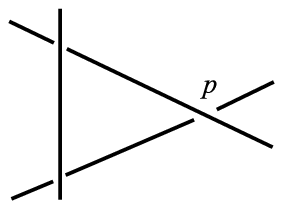}
\caption{}\label{Fig 2.2.2}

\end{figure}

The proof of the above observation amounts to an easy case by case
checking and we omit it.  The observation makes possible the
following induction: we can assume that we have a bad point at the
crossing between the middle arc and the lower or the upper arc.
Let us consider for example the situation described by Figure
\ref{Fig 2.2.2}.

We consider two subcases, according to $\sgn p$ being $+$ or $-$.

Assume $\sgn p = -$.  Then by Conway relations $$w^\circ(L) =
w^\circ(L^p_-) = w^\circ(L^p_+)\ast w^\circ(L^p_\circ)$$
$$w^\circ(R(L)) = w^\circ(R(L)^p_-) =
w^\circ(R(L)^p_+)\ast w^\circ(R(L)^p_\circ).$$ But $R(L)^p_+ =
R(L^p_+)$ and by the inductive hypothesis $$w^\circ(L^p_+) =
w^\circ(R(L^p_+)).$$  Also $R(L)^p_\circ$ is obtained from
$L^p_\circ$ by two subsequent Reidemeister moves of type two (see
Figure \ref{Fig 2.2.3}), whence by M.I.H. $$w^\circ(R(L)^p_\circ) =
w^\circ(L^p_\circ)$$ and the equality $w^\circ(L) = w^\circ(R(L))$
follows.

\begin{figure}[htbp]
\centering \includegraphics[width=0.5\textwidth]{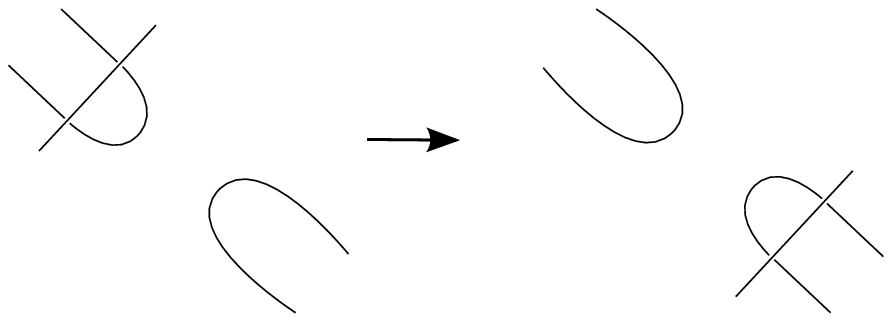}
\caption{}\label{Fig 2.2.3}
\end{figure}

Assume now that $\sgn p=+$.  Then by Conway relations $$w^\circ(L)
= w^\circ(L^p_+) = w^\circ(L^p_-)| w^\circ(L^p_\circ), \textrm{
and}
$$
$$w^\circ(R(L)) = w^\circ(R(L)^p_+) =
w^\circ(R(L)^p_-)| w^\circ(R(L)^p_\circ).$$ But $R(L)^p_- =
R(L^p_-)$ and by the inductive hypothesis $$w^\circ(L^p_-) =
w^\circ(R(L^p_-)).$$

Now, $L^p_\circ$ and $R(L)^p_\circ$ are essentially the same
diagrams (see Figure \ref{Fig 2.2.4}), whence $w^\circ(L^p_\circ) =
w^\circ(R(L)^p_\circ)$ and the equality $w^\circ(L) = w^\circ(R(L))$
follows.

\begin{figure}[htbp]
\centering
\includegraphics[width=0.5\textwidth]{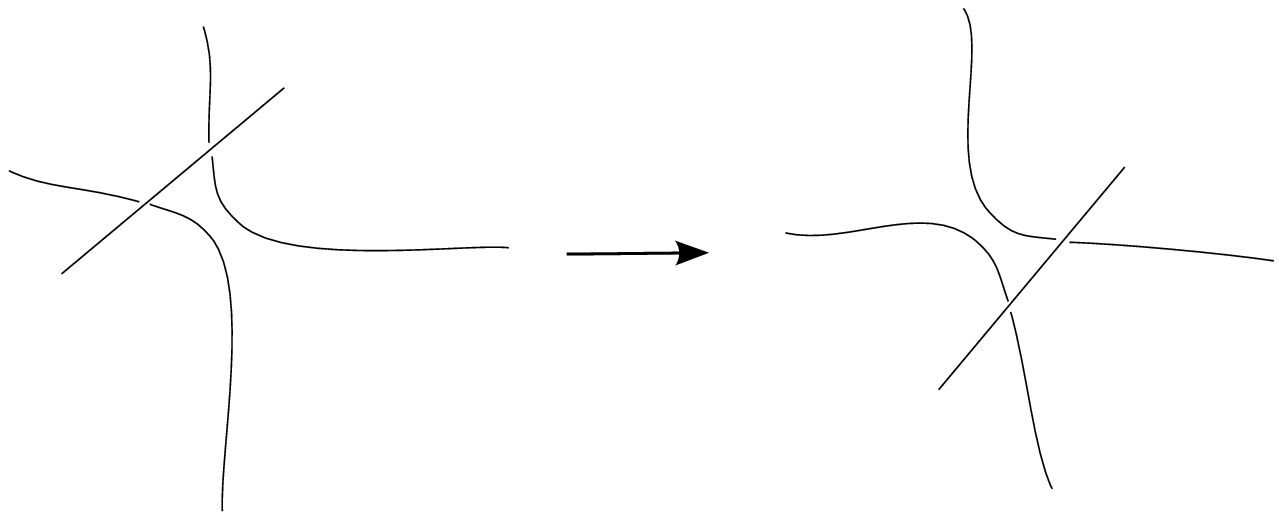}
\caption{}\label{Fig 2.2.4}
\end{figure}

\noindent Reidemeister moves of the first type.

\noindent The base points should always be chosen so that the
crossing point involved in the move is good.

\noindent Reidemeister moves of the second type.

\noindent There is only one case when we cannot chose base points
to guarantee the points involved in the move to be good.  It
happens when the involved arcs are parts of different components
and the lower arc is a part of the earlier component.  In this
case the both crossing points involved are of different signs, of
course. Let use consider the situation shown in Figure \ref{Fig
2.2.5}.

\begin{figure}[htbp]
\centering
\includegraphics{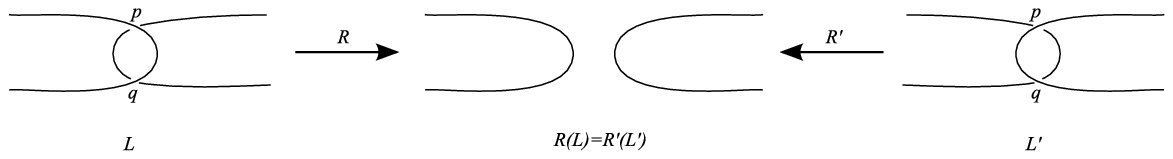}
\caption{}\label{Fig 2.2.5}
\end{figure}

We want to show that $w^\circ(R(L)) = w^\circ(L)$.  But by the
inductive hypothesis we have $$w^\circ(L') =
w^\circ(R'(L'))=w^\circ(R(L)).$$  Using the already proven Conway
relations, formulae C6 and C7 and M.I.H.~if necessary, it can be
proved that $w^\circ(L) = w^\circ(L')$.  Let us discuss in detail
the case involving M.I.H.  It occurs when $\sgn p=+$.   Then we have
$$w^\circ(L)=w^\circ(L^q_-)=w^\circ(L^q_+)\ast w^\circ(L_\circ^q)=
(w^\circ(L^{qp}_{+-})|w^\circ(L^{qp}_{+\circ}))\ast
w^\circ(L^q_\circ).$$ But $L^{qp}_{+-} = L'$ and by
M.I.H.~$w^\circ(L^{qp}_{+\circ})=w^\circ(L^q_\circ)$ (see Figure
\ref{Fig 2.2.6}, here $L^{qp}_{+\circ}$ and $L^q_\circ$ are obtained
from $K$ by a Reidemeister move of the first type).

\begin{figure}[htbp]
\centering
\includegraphics[width=0.8\textwidth]{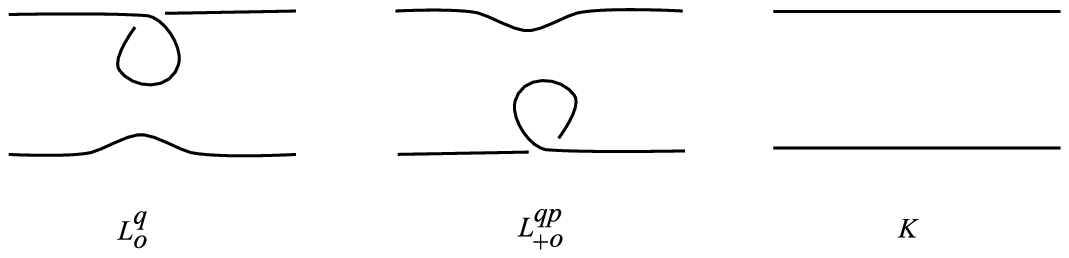}
\caption{}\label{Fig 2.2.6}
\end{figure}

$$w^\circ(L)=w^\circ(L')$$  $$w^\circ(L) = w^\circ(R(L)).$$  The
case: $\sgn p =-$ is even simpler and we omit it.  This completes
the proof of the independence of $w^\circ$ of Reidemeister
moves.\medskip

To complete the Main Inductive Step it is enough to prove the
independence of $w^\circ$ of the order of components.  Then we set
$w_{k+1} = w^\circ$.  The required properties have been already
checked.

\subsubsection*{\underline{Independence of the order of components} (I.O.C.)}

\ \smallskip

It is enough to verify that for a given diagram $L$ ($c(L) \leq
k+1$) and fixed base points $b=(b_1, \ldots, b_i, b_{i+1}, \ldots,
b_n)$ we have $$w_b(L) = w_{b'}(L)$$ where $b' = (b_1, \ldots,
b_{i+1}, b_{i}, \ldots, b_n)$.  This is easily reduced by the
usual induction on $b(L)$ to the case of a descending diagram. To
deal with this case we will choose $b$ in an appropriate way.

Before we do it, let us formulate the following observation:  If
$L_i$ is a trivial component of $L$, i.e. $L_i$ has no crossing
points, neither with itself, nor with other components, then the
specific position of $L_i$ in the plane has no effect on
$w^\circ(L)$; in particular, we may assume that $L_i$ lies
separately from the rest of the diagram:

\begin{figure}[htbp]
\centering
\includegraphics{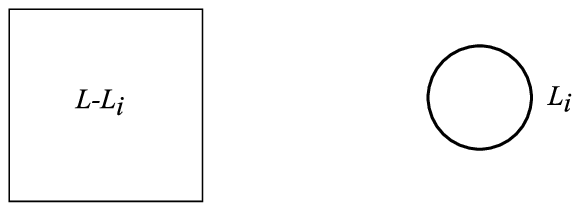}
\caption{}\label{Fig 2.2.7}
\end{figure}

This can be easily achieved by induction on $b(L)$, or better by
saying that it is obvious.

For a descending diagram we will be done if we show that it can be
transformed into another one with less crossings by a series of
Reidemeister moves which do not increase the crossing number.  We
can then use I.R.M.~ and M.I.H.  This is guaranteed by the
following lemma.

\setcounter{thm}{13}

\begin{lem}\label{Lemma 2.2.14} Let $L$ be a diagram with $k$ crossings and a given
ordering of components $L_1, L_2, \ldots, L_n$.  Then either $L$
has a trivial circle as a component or there is a choice of base
points $b=(b_1,\ldots, b_n); ~b_i\in L_i$ such that a descending
diagram $L^d$ associated with $L$ and $b$ (that is, all the bad
crossings of $L$ are changed to good ones) can be changed into a
diagram with less than $k$ crossings by a sequence of Reidemeister
moves not increasing the number of crossings.\end{lem}

\begin{proof} A closed part cut out of the plane by arcs of $L$ is
called an $i$-gon if it has $i$ vertices (see Figure \ref{Fig
2.2.8}). Every $i$-gon with $i\leq 2$ will be called an $f$-gon
($f$ works for few).  Now let $X$ be an innermost $f$-gon, that
is, and $f$-gon which does not contain any other $f$-gon inside.

\begin{figure}[htbp]
\centering
\includegraphics{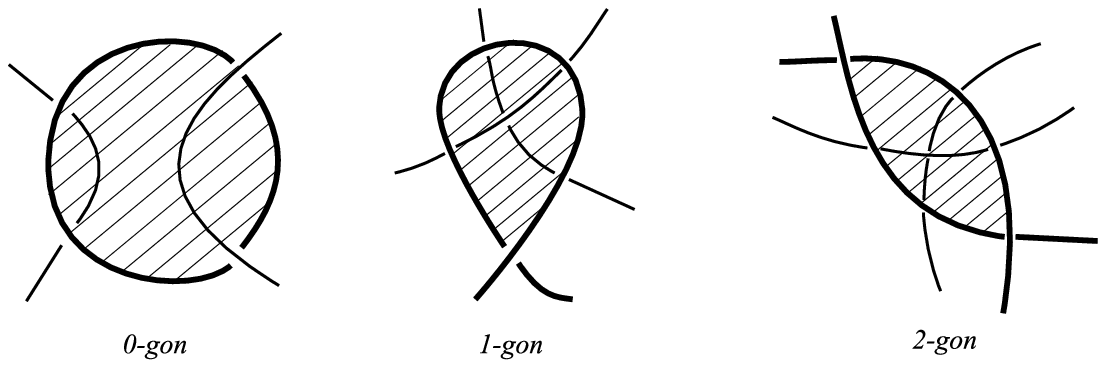}
\caption{}\label{Fig 2.2.8}
\end{figure}

If $X$ is a $0$-gon we are done because $\partial X$ is a trivial
circle.  If $X$ is a $1$-gon then we are done because
$\mathrm{int} X \cap L = \emptyset$ so we can perform on $L^d$ a
Reidemeister move which decreases the number of crossings of $L^d$
(Figure \ref{Fig 2.2.9}).

\begin{figure}[htbp]
\centering
\includegraphics{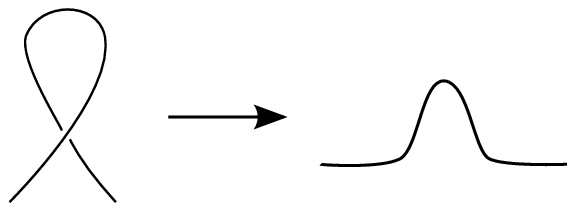}
\caption{}\label{Fig 2.2.9}
\end{figure}

Therefore, we assume that $X$ is a $2$-gon.  Each arc which cuts
$\mathrm{int} X$ goes from one edge to another. Furthermore, no
component of $L$ lies fully in $X$ so we can choose base points
$b=(b_1, \ldots, b_n)$ lying outside $X$.  This has important
consequences:  If $L^d$ is an untangled diagram associated with
$L$ and $b$ then each $3$-gon in $X$ supports a Reidemeister move
of the third type (i.e. the situation of the Figure \ref{Fig
2.2.10} is impossible).

\begin{figure}[htbp]
\centering \includegraphics{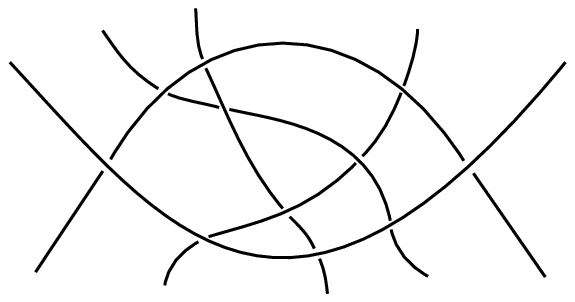} \caption{}\label{Fig
2.2.10}
\end{figure}

Now we will prove Lemma \ref{Lemma 2.2.14} by induction on the
number of crossings of $L$ contained in the 2-gon $X$ (we denote
this number by $c$).

If $c=2$ then $\mathrm{int} X\cap L =\emptyset$ and we are done by
the previous remark (2-gon $X$ can be used to make the
Reidemeister move of the second type on $L^d$ and to reduce the
number of crossings in $L^d$).

Assume that $L$ has $c>2$ crossings in $X$ and that Lemma \ref{Lemma
2.2.14} is proved for less than $c$ crossings in $X$. In order to
make the inductive step we need the following fact.

\begin{prop}\label{Prop 2.2.15} If $X$ is an innermost 2-gon with
$\mathrm{int}X\cap L\neq \emptyset$ then there is a 3-gon,
$\triangle \subset X$ such that $\triangle \cap \partial X \neq
\emptyset$, $\mathrm{int}\triangle \cap L\neq\emptyset$.\end{prop}

Before we prove Proposition \ref{Prop 2.2.15}, we will show how
Lemma \ref{Lemma 2.2.14} follows from it.

We can perform the Reidemeister move of the third type using the
3-gon $\triangle$ and reduce the number of crossings of $L^d$ in
$X$ (compare Figure \ref{Fig 2.2.11}).

\begin{figure}[htbp]
\centering
\includegraphics[width=0.5\textwidth]{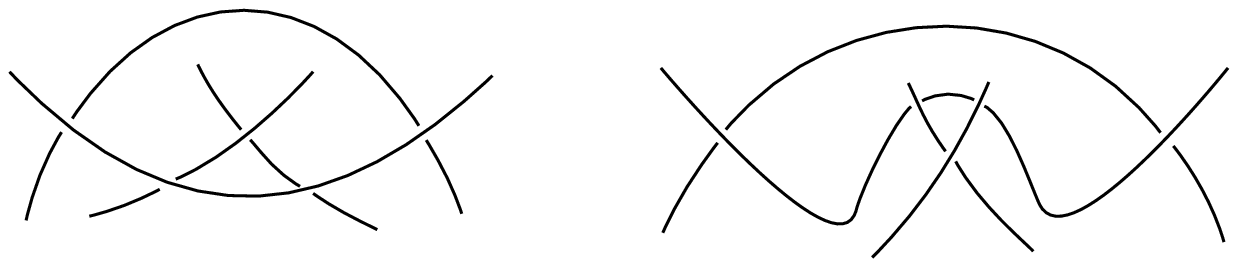}
\caption{}\label{Fig 2.2.11} \end{figure}

Now either $X$ is an innermost $f$-gon with less than $c$ crossings
in $X$ or it contains an innermost $f$-gon with less than $c$
crossings in it.  In both cases we can use the inductive hypothesis.

Instead of proving Proposition \ref{Prop 2.2.15}, we will show a
more general fact, which has Proposition \ref{Prop 2.2.15} as a
special case.

\begin{prop}\label{Prop 2.2.16}  Consider a 3-gon $Y=(a,b,c)$ such
that each arc which cuts it goes from the edge $\overline{ab}$ to
the edge $\overline{ac}$ without self-intersections (we allow $Y$ to
be a 2-gon considered as a degenerate 3-gon with $\overline{bc}$
collapsed to a point.  Furthermore, let $\mathrm{int} Y$ be cut by
some arc.  Then there is a 3-gon $\triangle\subset Y$ such that
$\triangle \cap \overline{ab} \neq \emptyset $ and $\mathrm{int}
\triangle$ is not cut by any arc.\end{prop}

\begin{proof}[Proof of Proposition \ref{Prop 2.2.16}]  We proceed by
induction on the number of arcs in $\mathrm{int} Y\cap L$ (each such
arc cuts $\overline{ab}$ and $\overline{ac}$).  For one arc it is
obvious (Figure \ref{Fig 2.2.12}).  Assume it is true for $k$ arcs
($k\geq 1$)  and consider the $(k+1)$st arc $\gamma$.  Let
$\triangle_\circ = (a_1,b_1,c_1)$ be a 3-gon from the inductive
hypothesis with and edge $\overline{a_1 b_1}\subset \overline{ab}$
(Figure \ref{Fig 2.2.13}).

\begin{figure}[htbp]
\centering
\includegraphics[width=0.3\textwidth]{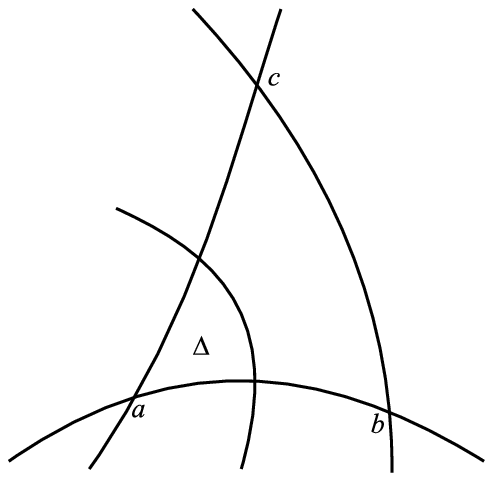}
\caption{}\label{Fig 2.2.12}
\end{figure}

If $\gamma$ does not cut $\triangle_\circ$ or it cuts
$\overline{a_1b_1}$ we are done (Figure \ref{Fig 2.2.13}).
Therefore let us assume that $\gamma$ cuts $\overline{a_1c_1}$ (in
$u_1$) and $\overline{b_1c_1}$ (in $w_1$).  Let $\gamma$ cut
$\overline{ab}$ in $u$ and $\overline{ac}$ in $w$ (Figure \ref{Fig
2.2.14}).

\begin{figure}[htbp]
\centering
\includegraphics[width=0.4\textwidth]{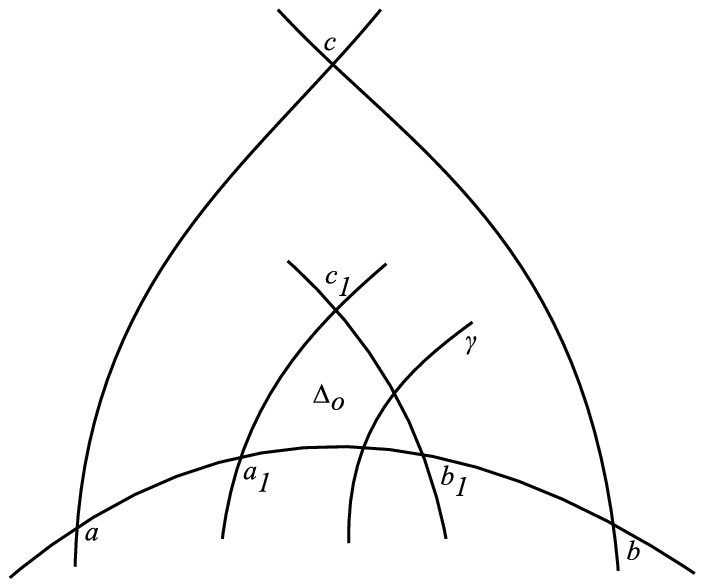}
\caption{}\label{Fig 2.2.13}
\end{figure}

We have to consider two cases: \begin{enumerate} \item[(a)]
$\overline{uu}_1\cap \mathrm{int}\triangle_\circ=\emptyset$ (so
$\overline{ww}_1 \cap \mathrm{int}\triangle_\circ = \emptyset$);
Figure \ref{Fig 2.2.14}.

Consider the 3-gon $ua_1u_1$.  No arc can cut the edge
$\overline{a_1u_1}$ so each arc which cuts the 3-gon $ua_1u_1$
cuts the edges $\overline{ua}_1$ and $\overline{uu}_1$.

Furthermore, this 3-gon is cut by less than $k+1$ arcs so by the
inductive hypothesis there is a 3-gon $\triangle$ in $ua_1u_1$
with an edge on $\overline{ua_1}$ the interior of which is not cut
by any arc.  The $\triangle$ satisfies the thesis of Proposition
\ref{Prop 2.2.16}.

\item[(b)] $\overline{uw_1}\cap\mathrm{int}\triangle_\circ =
\emptyset$ (so $\overline{wu_1}\cap\mathrm{int}\triangle_\circ =
\emptyset$). In this case we proceed like in case (a).
\end{enumerate}

This completes the proof of Proposition \ref{Prop 2.2.16} and hence
the proof of Lemma \ref{Lemma 2.2.14}.

\begin{figure}[htbp]
\centering
\includegraphics[width=0.4\textwidth]{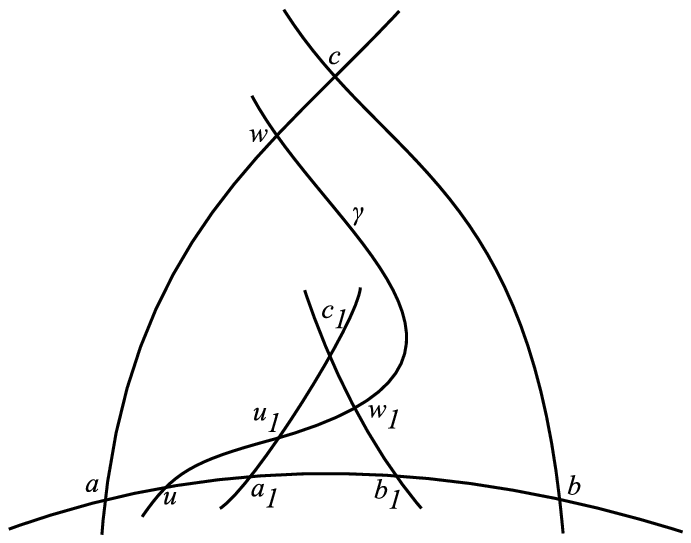}
\caption{}\label{Fig 2.2.14}
\end{figure}

\end{proof}
\end{proof}

\vspace{2in}

\ \ \ \ \ \ \ \ \ \ \ \ \


\newpage \thispagestyle{empty}


AMS 57M25  \hrulefill[0pt]{} ISSN 0239-8885

{\center{ \vspace{1cm}

UNIWERSYTET WARSZAWSKI

 INSTYTUT MATEMATYKI

\vspace{2in}

J\'{o}zef H. Przytycki

Survey on recent invariants on classical knot theory

II. Skein equivalence and properties of invariants \\ of Conway
type. Partial Conway algebras

\vspace{3in}

%


Preprint 8/86

``na prawach r{\c e}kopisu''

\vspace{1cm}

Warszawa 1986

}}
\newpage

\section{Skein equivalence and properties of invariants of Conway type}
\setcounter{thm}{0} \numberwithin{thm}{section}
\numberwithin{figure}{section} \numberwithin{equation}{section}

Let us start this chapter by introducing the equivalence relation on
oriented links which identifies links which cannot be distinguished
by an invariant of Conway type.  This relation is called the skein
equivalence and is denoted by $\sim_S$ (\cite{Co}, \cite{Li_M_1},
\cite{P_T_1}).

\begin{defn}  $\sim_S$ is the smallest equivalence relation on
isotopy classes of oriented links which satisfies the following
condition:

Let $L_1$ (respectively, $L_2$) be a diagram of a link $L_1$
(respectively $L_2$) such that $p_1$ and $p_2$ are crossings of the
same sign and $$(L'_1)^{p_1}_{-\sgn p_1} \sim_S (L'_2)^{p_2}_{-\sgn
p_2}, \textrm{ and }(L'_1)^{p_1}_{\circ} \sim_S
(L'_2)^{P_2}_{\circ}$$ then $L_1 \sim_S L_2$.\end{defn}

From the definition we get on the spot:

\begin{lem}  Two oriented links are not skein equivalent iff there
exists an invariant of Conway type which distinguishes them.  In
particular, assigning to an oriented link its skein equivalence
class ins an invariant of Conway type.\end{lem}

Skein equivalence can also be described as a ``limit'' of the
sequence of relations.  Namely,

\begin{itemize}

\item[$\sim_0$] ~$L_1 \sim_0 L_2 \textrm{ iff } L_1 \textrm{ is
isotopic to }L_2,$ and

\item[$\sim_i$] ~is the smallest equivalence relation on
oriented links which satisfies the condition:  Let $L'_1$
(respectively $L'_2$ be a diagram of a link $L_1$ (respectively
$L_2$ with a given crossing $p_1$ (respectively $p_2$), $\sgn p_1 =
\sgn p_2$ and $(L'_1)^{p_1}_{-\sgn p_1} \sim_{i-1}
(L'_2)^{p_2}_{-\sgn p_2}$ and $(L'_1)^{p_1}_\circ \sim_{i-1}
(L'_2)^{p_2}_\circ$ then $L_1 \sim_i L_2$.

\end{itemize}
Now it is easy to show that the smallest relation which contains all
$\sim_i$ relations is the skein equivalence relation.  We can weaken
the relations $\sim_i$ not assuming that they are equivalence
relations.  Namely, we introduce $\approx_0, \approx_1, \ldots,
\approx_i,\ldots, \approx_\infty$ as follows:

\begin{itemize}

\item[$\approx_0$] $=\sim_0$, and

\item[$\approx_i$] $L_1 \approx_i L_2$ iff there exist diagrams
$L'_1$ for $L_1$ and $L'_2$ for $L_2$ with crossings $p_1$ and $p_2$
respectively such that $\sgn p_1 = \sgn p_2$ and
$(L'_1)^{p_1}_{-\sgn p_1} \approx_{i-1} (L'_2)^{p_2}_{-\sgn p_2}$
and $(L'_1)^{p_1}_{\circ} \approx_{i-1} (L'_2)^{p_2}_{\circ}$, and

\item[$\approx_\infty$]  is the smallest equivalence relation on
oriented links which contains all relations $\approx_i$.

\end{itemize}

\begin{prob}\begin{enumerate}
\item[(a)] Are there links which are skein ($\sim_S$) equivalent
but not $\approx_\infty$ equivalent?

\item[(b)] Are there links which are $\approx_\infty$ equivalent
but are not $\approx_i$ equivalent for any finite $i$?

\item[(c)] For which $i>0$ do there exist links which are $\sim_i$
equivalent but are not $\approx_i$ equivalent?

\end{enumerate}\end{prob}

Let us come back now to invariants of Conway type and to the skein
equivalence.  We start from examples of links which are not isotopic
but which are skein equivalent.

\begin{lem} If $-L$ denotes the link we get from the link $L$ by
changing orientation of each component of $L$ then $-L \sim_S$.  In
particular, for the Jones-Conway polynomial $P(x,y)$, $P_{-L}(x,y) =
P_L(x,y)$.\end{lem}

\begin{proof}The proof is immediate if one notices that the sign of a
crossing is not changed when we change $L$ to $-L$.  So we can build
the resolving tree (the same for $L$ and $-L$) proving that $L
\approx_{\crs{L}-1} -L$ where $\crs(L)$ is the minimal number of
crossings of diagrams of $L$. \end{proof}

\begin{example}\label{Example 3.5}  The links $L_1$ and $L_2$ from Figure \ref{Fig 3.1}
are skein equivalent (if we build a resolving tree starting from the
marked crossings then we even show that $L_1 \approx_1 L_2$).  $L_1$
can be distinguished from $L_2$ by considering global linking
numbers of its sublinks.

\begin{figure}[htbp]
\centering
\includegraphics{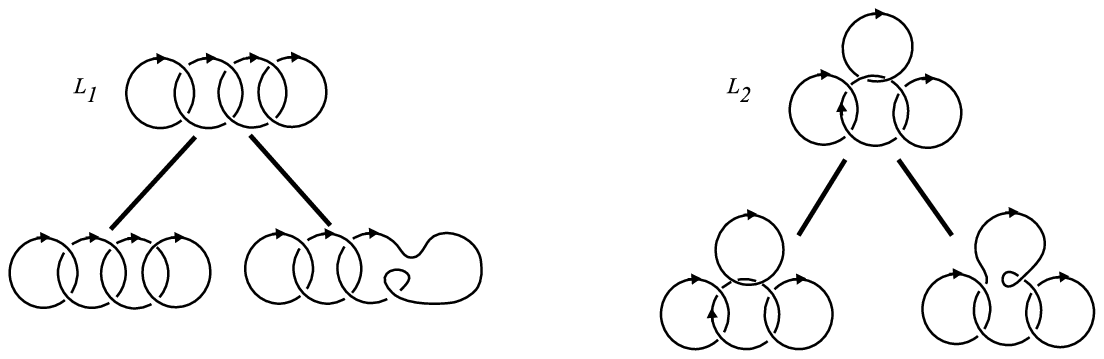}
\caption{}\label{Fig 3.1}
\end{figure}

\end{example}

For further examples we need the definition of a tangle and a
mutation.

\begin{defn}[\cite{Li_M_1}] \begin{enumerate}
\item[(a)] A tangle is a part of a diagram of a link with two
inputs and two outputs (Figure \ref{Fig 3.2}(a)).  It depends on
an orientation of the diagram which arcs are inputs and which are
outputs.  We distinguish tangles with neighboring inputs (Figure
\ref{Fig 3.2}(b)) and alternated tangles (Figure \ref{Fig
3.2}(c)).

\begin{figure}[htbp]
\centering
\includegraphics[scale= 0.8]{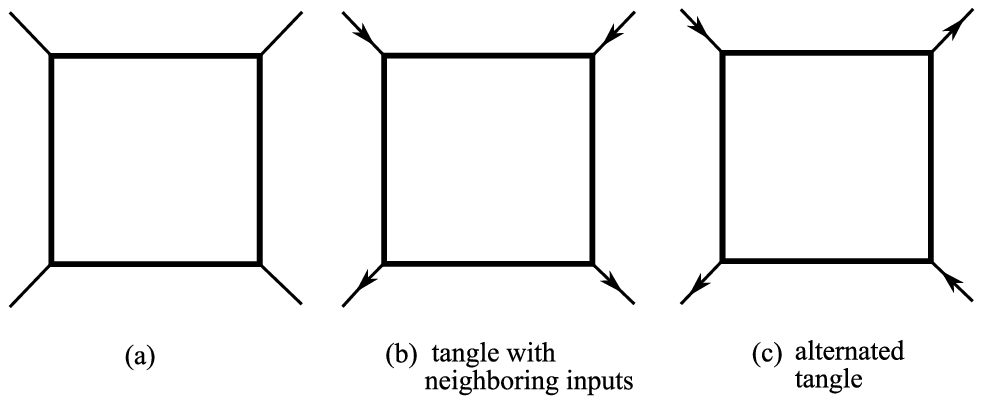}
\caption{}\label{Fig 3.2}
\end{figure}

\item[(b)] Let $L_1$ and $L_2$ be oriented diagrams of links.
Then $L_2$ is a mutation of $L_1$ if $L_2$ can be obtained from
$L_1$ by the following process:  \begin{enumerate}

\item[(i)] Remove from $L_1$ an inhabitant $T$ of a tangle $B$.

\item[(ii)] Rotate $T$ through angle $\pi$ about the ventral axis
(perpendicular to the plane of the diagram) or about the
horizontal or vertical axis of the tangle and iff necessary change
the orientation of $T$ (so that inputs and outputs are preserved).

\item[(iii)] place the new inhabitant into the tangle $B$ to get
$L_2$.
\end{enumerate}
\end{enumerate}
\end{defn}

\begin{lem}[\cite{Li_M_1}, \cite{Hos-1}, \cite{Gi}]\label{Lemma 3.7}
If $L_1$ and
$L_2$ are links whose some diagrams differs by a mutation then $L_1
\sim_S L_2$.  In fact we have $L_1\approx_{\crs -1}L_2$ where $\crs$
is the number of crossings in the mutated tangle of the diagram
(here $\approx_{-1} = \approx_{0}$).\end{lem}

\begin{proof} For $\crs \leq 1$ we rotate one of the tangles from
Figure \ref{Fig 3.3} (up to trivial circles in the tangle) and such
a mutation does not change the isotopy class of a link.  Then we use
in the proof the standard induction on $\crs$ and the minimal number
of bad crossings in the tangle (similarly as in the proof of Theorem
\ref{Theorem 2.1.1}).

\begin{figure}[htbp]
\centering
\includegraphics[scale = 0.8]{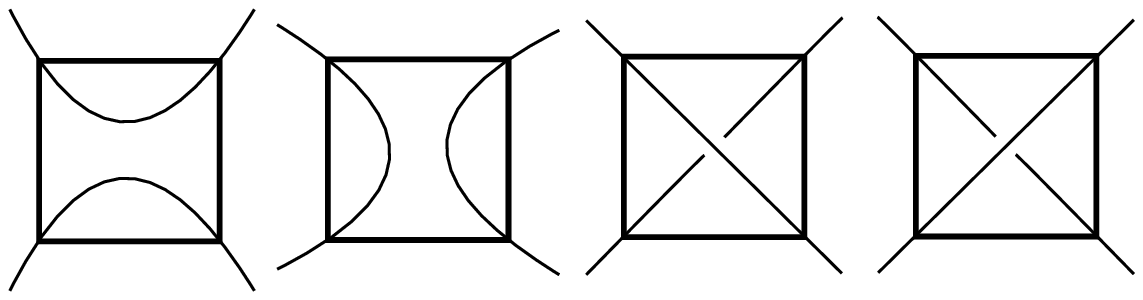}
\caption{}\label{Fig 3.3}
\end{figure}

\end{proof}

\begin{example} The Conway knot (Figure \ref{Fig 3.4}) and the
Kinoshita-Terasaka knot (Figure \ref{Fig 3.5}) are mutants of one
another (the rotated tangle is shown on Figures \ref{Fig 3.4} and
\ref{Fig 3.5}).  Therefore, these knots are skein equivalent (even
$\approx_1$ equivalent; just start to build the resolving tree from
the marked crossings).  D.~Gabai \cite{Ga} has shown that these
knots have different genera so they are not isotopic (R.~Riley
\cite{Ri} was the first to distinguish these knots).\end{example}

\begin{figure}[htbp]
\centering

\subfigure[Conway knot.]{\includegraphics[scale =
0.9]{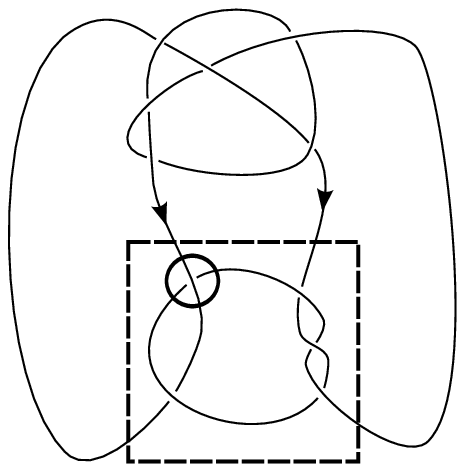}\label{Fig 3.4}} \hspace{2cm}
\subfigure[Kinshina-Terasaka knot.]{\includegraphics[scale =
0.9]{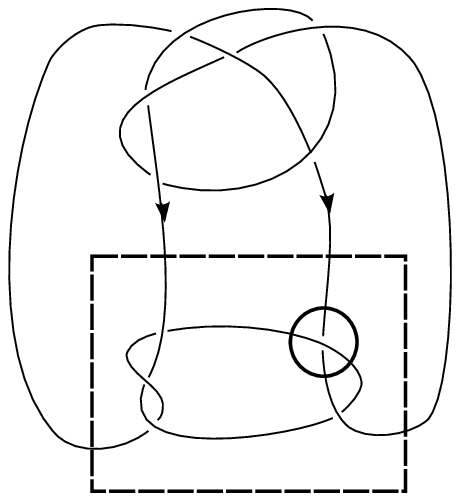}\label{Fig 3.5}}

\caption{}

\end{figure}

\begin{example}[\cite{Li_M_1}, \cite{Hos-1}] In Figure \ref{Fig 3.6}
is shown a pretzel link that will be denoted $L(p_1^{\eps(1)},
p_2^{\eps(2)}, \ldots , p_n^{\eps(n)})$.

\setcounter{figure}{5}

\begin{figure}[htbp]
\centering
\includegraphics[scale = 0.8]{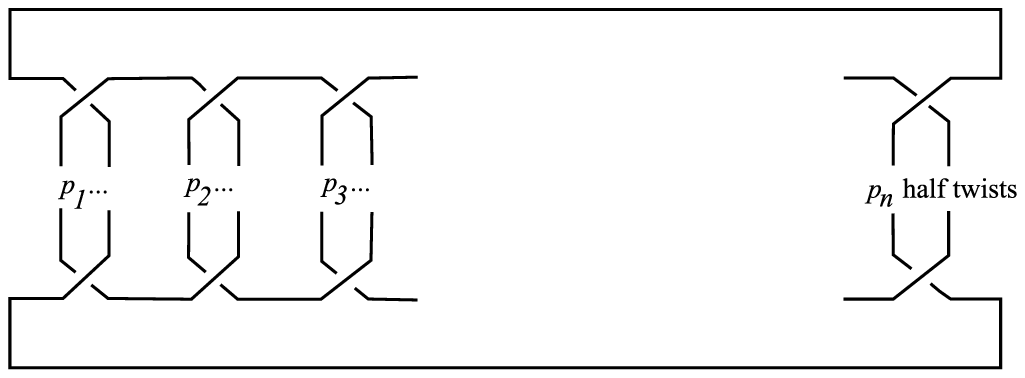}
\caption{}\label{Fig 3.6}
\end{figure}

The $i$th vertical strip has $p_i$ half twists.  the superscript
$\eps(i)$ is 1 if all the crossings on the $i$th strip are positive
and $-1$ if they are negative.  Note that $\eps(i)$ depends on the
choice of orientation of the various components; for a given $(p_1,
p_2, \ldots, p_n)$ an arbitrary choice of $\eps$ may not be
possible.  To have a pretzel link oriented assume that the upper arc
is oriented from the right into the left (compare Figure \ref{Fig
3.7}).  It follows from Lemma \ref{Lemma 3.7} that for any
permutation $\partial \in S_n$, $$L(p_1^{\eps(1)}, p_2^{\eps(2)},
\ldots, p_n^{\eps(n)})$$ is skein equivalent to
$$L(p_{\partial(1)}^{\eps({\partial(1)})},
p_{\partial(2)}^{\eps({\partial(2)})}, \ldots,
p_{\partial(n)}^{\eps({\partial(n)})}),$$ because we achieve the
second link from the first by a finite sequence of mutations.  In
particular, we can travel from the pretzel link of two components,
$L(3,5,3,-5^{-1}, -3^{-1}, -3^{-1})$, (see Figure \ref{Fig 3.7}), to
its mirror image, $L(-3^{-1},-5^{-1}, -3^{-1}, 5,3,3)$, using a
finite number of mutations, however, these links are not isotopic
(see \cite{B_Z}).\end{example}

\begin{figure}[h]
\centering
\includegraphics[scale = 0.8]{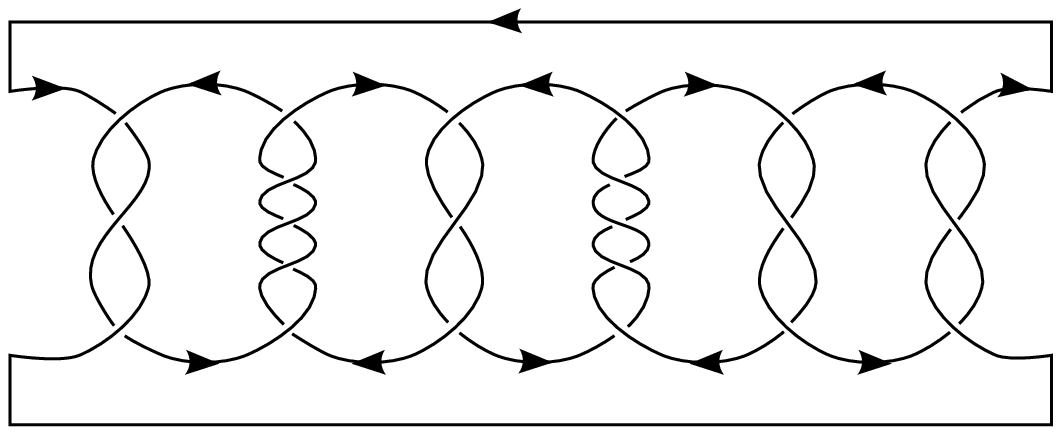}
\caption{}\label{Fig 3.7}
\end{figure}

\begin{example}\label{Example 3.10}Consider a diagram, $D$, of a link with two
alternating tangles.  We assume the following convention:\medskip

\includegraphics[scale = .6]{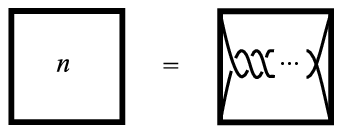}, with $n$ half-twists in the second box,
and

\medskip
\includegraphics[scale = .6]{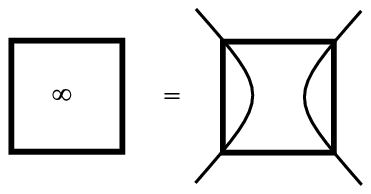}.

Let $D(n,m)$ denote the diagram obtained from $D$ by putting $m$ in
the first tangle and $n$ into the second.  Assume  that $D(\infty,
n)$ is skein equivalent to $D(m,\infty)$ for every $m$ and $n$. Then
for $m+n = m'+n'$ and $m \equiv m' (\mod 2)$, $D(m,n) \sim_S D(m',
n')$.

Examples of diagrams which satisfy the above conditions were found
by T.~Kanenobu \cite{Ka_1, Ka_2, Ka_3} (Figure \ref{Fig 3.8}).

\begin{figure}[htbp]
\centering
\includegraphics{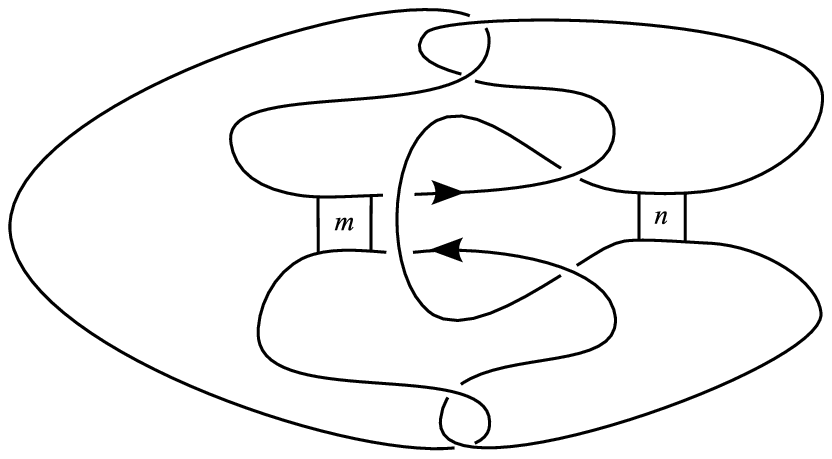}
\caption{}\label{Fig 3.8}
\end{figure}

In this example, $D(\infty, m)$ and $D(n, \infty)$ are trivial links
of 2 components.  Kanenobu \cite{Ka_1} has shown (using Jones-Conway
polynomial and the structure of the Alexander module) that $D(2m,
2n)$ is isotopic to $D(2m', 2n')$ iff $(m,n) = (m', n')$ or $(m,n) =
(n',m')$.
\end{example}

To show the statement from Example \ref{Example 3.10} one should use
the standard induction on $|m-m'|$.

The next example and its story are taken from the Lickorish and
Millett paper \cite{Li_M_1}.

Using a computer, M.B.~Thistlethwaite has shown that amongst the
12966 knots with at most 13 crossings, there are thirty with the
Conway polynomial $\nabla_L(z) = 1+2z^2+2z^4$.  Examination of these
failed to find a pair of knots distinguished by the Jones-Conway
polynomial, but not by the Jones polynomial.  However, an outcome of
that search produced the following example.

\begin{example}\label{Example 3.11}\cite{Li_M_1}
Consider the knots in Figure \ref{Fig 3.9}.

\begin{figure}[htbp]
\centering

\subfigure[$8_8$]{\includegraphics{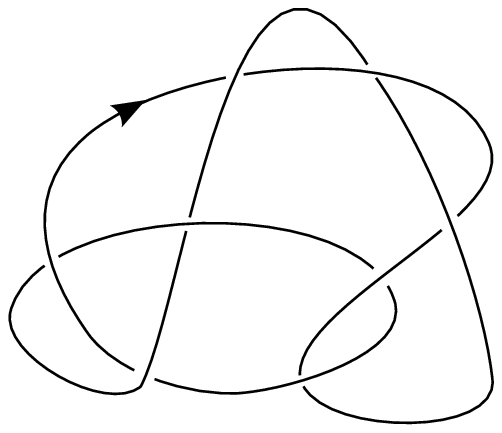}} \hspace{.8cm}
\subfigure[$10_{129}$]{\includegraphics{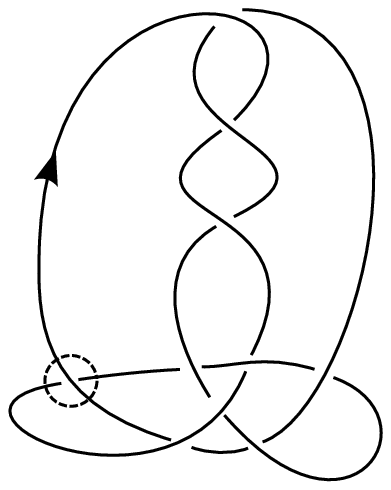}}\hspace{0.8cm}
\subfigure[$13_{6714}$]{\includegraphics{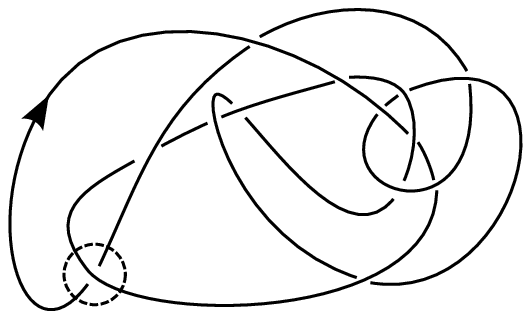}}
\caption{}\label{Fig 3.9}
\end{figure}

Now changing the encircled crossing of $13_{6714}$ produces
$10_{129}$, and nullifying (smoothing) that crossing produces
$T_2$, the trivial link of 2 components.  Similarly, changing the
encircled crossing in $10_{129}$ give $8_8$ and smoothing it gives
$T_2$. Hence, we have triples $(13_{6714}, 10_{129}, T_2)$ and
$(8_8, 10_{129}, T_2)$, both of the form $(L_+, L_-, L_\circ)$.
Therefore $8_8$ and $13_{6714}$ are skein equivalent.  The knots
of Figure \ref{Fig 3.9} are slice knots and so have zero
signature. Furthermore, $8_8$ is the only knot with $L(25, 11)$ as
its double branched cover ($8_8$ is the 2-bridge knot of type
$K_{11/25}$); see \cite{Hod}.  Because $8_8$ is not isotopic to
$13_{6714}$ (first shown by Thistlethwaite); (also Kauffman
polynomial, Chapter 5, distinguishes these knots) therefore $8_8$
and $13_{6714}$ have different double branched covers.  From this
it can be gained that $8_8$ cannot be obtained from $13_{6714}$ by
a finite sequence of mutations (because a mutation does not change
the double cover of a knot).  the same observation can be gained
using the Kauffman polynomial (Lemma 5.9(e)).

Lickorish and Millett found in \cite{Li_M_1} that $8_8$ and
$\overline{10}_{129}$ (the mirror image of $10_{129}$) have the same
Jones-Conway polynomial and they asked whether they are skein
equivalent.  Kanenobu has given the positive answer to this question
showing that the knots $8_8$, $10_{129}$, and $13_{6714}$ are
special cases of his $D(m, n)$ knots \cite{Ka_2}.
\end{example}

\begin{prop}\label{Prop 3.12} $8_8 \approx D(0, -1)$, $10_{129}
\approx D(2, -1)$, and $13_{6714} \approx D(2, -3)$, where $\approx$
denotes isotopy.\end{prop}

\begin{proof} Just by checking the needed equalities. \end{proof}

This allows us to answer the first part of Question 10
\cite{Li_M_1}:

\begin{cor}[\cite{Ka_3}] The knots $8_4$ and $\overline{10}_{129}$ are
skein equivalent but they have different unknotting
numbers.\end{cor}

\begin{proof} It can be easily shown that $10_{129}$ has unknotting
number 1.  For the proof that $8_8$ has unknotting number 2 we refer
to \cite{Li_M_1} (see also \cite{Ka_M}).\end{proof}

Examples which we have described so far have shown limitations of
invariants of Conway type.  However, the fact is that, for example,
the Jones-Conway polynomial is better than the Jones polynomial and
the Conway polynomial.  In fact, the Jones-Conway polynomial is the
stronger invariant.  This is confirmed by the following example
which comes from the Thistlethwaite tabulations (see \cite{Li_M_1}).

\begin{example} Consider the knot shown in Figure \ref{Fig 3.10} ($11_{388}$ in
\cite{Pe}).  We have $P_{11_{388}}(x,y)\neq
P_{\overline{11}_{388}}(x,y)$, but $V_{11_{388}}(t) =
V_{\overline{11}_{388}}(t)$ and $\nabla_{11_{388}}(z) =
\nabla_{\overline{11}_{388}}(z)$.

\begin{figure}[htbp]
\centering

$11_{388}$ \includegraphics{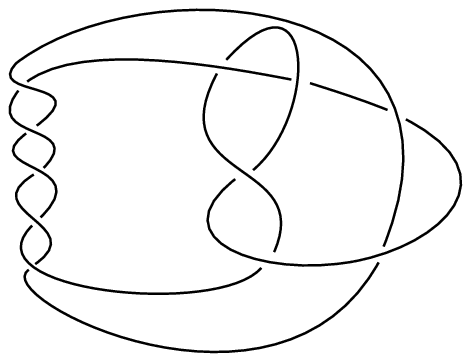}

\caption{}\label{Fig 3.10}
\end{figure}
\end{example}
\begin{proof} Check the values of the invariants for $11_{388}$ in
the table and use the following Lemma:

\begin{lem}\label{Lemma 3.15}If the link $\overline L$ is the mirror image of the link
$L$, then the Jones-Conway polynomial satisfies $$P_{\overline
L}(x,y) = P_L(y,x).$$

In particular, for the Jones polynomial we have $V_{\overline{L}(t)}
= V_{L}(\frac{1}{t})$, and for the Conway polynomial we have
$\nabla_{\overline L}(z)= \nabla_L(-z)$.\end{lem}

The proof of the lemma is an easy consequence of the observation that
the sign of each crossing is changed if we move from $L$ to
$\overline L$.

\end{proof}

The idea of Lemma \ref{Lemma 3.15} can be partially generalized to
other invariants yielded by a Conway algebra.

\begin{lem}\label{Lemma 3.16}Let $\mathcal{A} = \{A; a_1, a_2, \ldots, |, \ast\}$ be
a Conway algebra such that there exists an involution $\tau:A\to A$
which satisfies \begin{enumerate}

\item $\tau(a_i) = a_i$, and
\item $\tau(a|b) = \tau(a) \ast \tau(b)$.\end{enumerate}

Then the invariant, $A_L$, yielded by the algebra satisfies
$$A_{\overline{L}} = \tau(A_L).$$\end{lem}

In Examples \ref{Example 2.1.5} and \ref{Example 2.1.6}, $\tau$ is
the identity.  In Example \ref{Example 2.1.8}, which defines the
Jones-Conway polynomial, $\tau(P(x,y)) = P(y,x)$, and in Example
\ref{Example 2.1.10}, $\tau(n,z)=(n, -z)$.  On the other hand, the
algebra from Example \ref{Example 2.1.7} does not have such an
involution.

\begin{rem} We can build a Conway algebra using terms (words over
the alphabet \newline $a_1, a_2, \ldots, |, \ast, (, )$ which are
sensible).

In this algebra $\tau$ exists and is uniquely determined by the
conditions in Lemma \ref{Lemma 3.16}.  To prove this, it is enough
to observe that $\tau$ maps axioms of a Conway algebra into axioms.
On the other hand, this algebra, $\mathcal A_u$, is the universal
Conway algebra, that is, for any other Conway algebra $\mathcal A$
there is a unique homomorphism $\mathcal{A}_u \to
\mathcal{A}$.\end{rem}

\begin{rem} It may happen that for each pair $u,v \in A$ there
exists exactly one $w \in A$ such that $v|w = u$ and $u\ast w = v$.
Then we can introduce a new operation $\circ:A\times A \to A$
putting $u\circ v = w$ (we have such a situation in Examples
\ref{Example 2.1.6}, \ref{Example 2.1.7}, and \ref{Example 2.1.8},
but not in Examples \ref{Example 2.1.5} and \ref{Example 2.1.10}).
Then $a_n = a_{n-1} \circ a_{n-1}$.  We can interpret $\circ$ as
follows:  If $w_1$ is the invariant of $L_+$ and $w_2$ of $L_-$,
then $w_1 \circ w_2$ is the invariant of $L_\circ$.  If the
operation $\circ$ is well-defined we can find an easy formula for
invariants of connected and disjoint sums of links.\end{rem}

\begin{thm}\label{Theorem 3.19}If $L=L_1\sqcup L_2$ (a disjoint sum)
then $P_{L_1\sqcup L_2}(x,y) = (x+y)P_{L_1}(x,y) \cdot
P_{L_2}(x,y)$, where $P_L(x,y)$ denotes the Jones-Conway polynomial
of $L$.\end{thm}

\begin{proof} There is a diagram of $L$ in which $L_1$ is disjoint
from $L_2$.  It is a splittable diagram.  We will show Theorem
\ref{Theorem 3.19} for splittable diagrams.  We use the induction on
pairs $(\crs(L), b(L))$ ordered lexicographically; $\crs{L}$ denotes
the number of crossings and $b(L)$ the minimal number of bad
crossings over all choices of base points.

For $b(L) =0$, the theorem holds because $L$ is a trivial link of
$n(L)$ components and $L_1$ and $L_2$ are trivial links of $n(L_1)$
and $n(L_2)$ components respectively, and by the definition
$$P_L(x,y) = (x+y)^{n(L)-1} = (x+y)(x+y)^{n(L_1)-1} (x+y)^{n(L_2)-1}
= (x+y)P_{L_1}(x,y)\cdot P_{L_2}(x,y).$$

Assume that we have shown the theorem for splittable diagrams which
satisfy $(\crs(L), b(L))<(c,b)$, $b \neq 0$, and consider a diagram
$L$ with $(\crs(L),b(L))=(c,b)$.

Let $p$ be a bad crossing of $L$.   Consider first the case $p \in
L_1$, $\sgn p = +$.  For $L^p_-$ and $L^p_\circ$, the theorem is
true by an inductive hypothesis.  Therefore:
\begin{eqnarray*} & P_L(x,y)=P_{L^p_+}(x,y)=
\frac{1}{x}\left(P_{L_\circ^p}(x,y)-yP_{L^p_-}(x,y)\right) = \\
& \frac{1}{x}\left( (x+y)P_{(L_1)^p_\circ}(x,y)\cdot P_{L_2}(x,y)
-y(x+y)P_{(L_1)^p_-}(x,y)\cdot P_{L_2}(x,y)\right) =
\\ & (x+y) P_{L_2}(x,y)\cdot \left( \frac{1}{x}(P_{(L_1)^p_\circ}(x,y) -y P_{(L_1)^p_-}(x,y))\right)= (x+y)P_{L_2}(x,y)\cdot P_{L_1}(x,y),\end{eqnarray*}
which completes the proof of the theorem in the considered case.  In
other cases, we proceed similarly.
\end{proof}

\begin{cor}\label{Cor 3.20} If $L=L_1 \cs L_2$ (connected sum),
then $$P_L(x,y) = P_{L_1}(x,y)\cdot P_{L_2}(x,y).$$
\end{cor}

\begin{proof}  There is a diagram of $L$ as in Figure \ref{Fig
3.11}.  Rotate $L_2$ to get diagrams $L_+$ and $L_-$ as in Figure
\ref{Fig 3.12}.  Of course $L_+$ and $L_-$ are isotopic to $L$, and
$L_\circ$ (Figure \ref{Fig 3.12} is the disjoint sum of $L_1$ and
$L_2$.

\begin{figure}[htbp]
\centering
\includegraphics{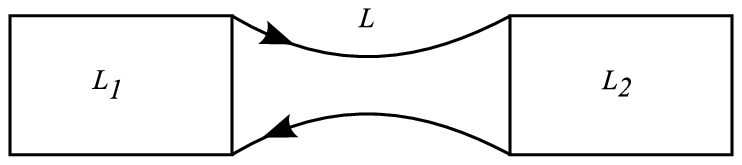}
\caption{}\label{Fig 3.11}
\end{figure}

\begin{figure}[htbp]
\centering
\includegraphics{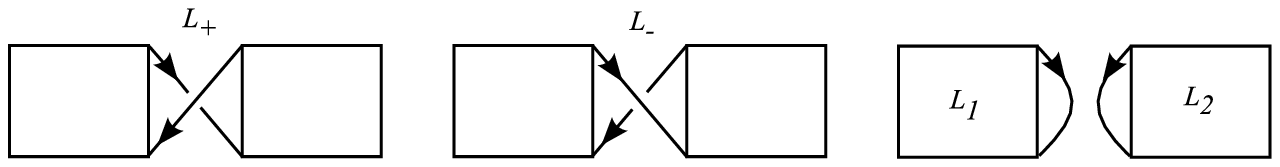}
\caption{}\label{Fig 3.12}
\end{figure}

Therefore $$xP_L(x,y)+yP_L(x,y) = P_{L_1 \sqcup L_2}(x,y)$$ and from
this we have $$(x+y)P_{L_1 \cs L_2}(x,y) = P_{L_1\sqcup L_2}(x,y).$$
This formula and Theorem \ref{Theorem 3.19} give us Corollary
\ref{Cor 3.20}.  \end{proof}

Theorem \ref{Theorem 3.19} and Corollary \ref{Cor 3.20} can be
partially generalized to the case of invariants yielded by an Conway
algebra with the operation $\circ$.  First, observe that if we add
the trivial knot to the given link $L$ then we get instead $A_L$ the
value $A_L \circ A_L$ (or $A_L^2$); Figure \ref{Fig 3.13}.  In
particular, we get known equality $a_i^2 = a_{i+1}$.  More
generally, considering Figure \ref{Fig 3.12}, we get

\setcounter{equation}{20}
\begin{equation}\label{Equation 3.21} A_{L_1 \sqcup L_2} = A^2_{L_1 \cs
L_2}.\end{equation}

\begin{figure}[htbp]
\centering
\includegraphics{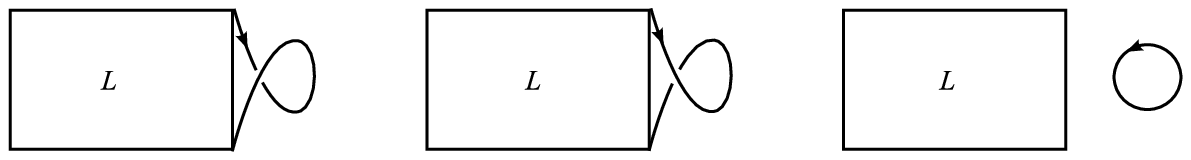}
\caption{}\label{Fig 3.13}
\end{figure}

Using a similar method to that of Theorem \ref{Theorem 3.19} and
Corollary \ref{Cor 3.20}, one can prove the following lemma.

\setcounter{thm}{21}

\begin{lem}\label{Lemma 3.22} Let the Conway algebra $\mathcal A$
have the action $\circ$ and let for each $w\in A$ exist a
homomorphism (operations $|$ and $\ast$ are preserved)
$\varphi_w:A\to A$ such that $\varphi_w(a_1)=w$,
$\varphi_w(a_2)=w^2$, $\varphi_w(a_3)=w^4, \ldots$~.  Then
$$\begin{array}{l}A_{L_1 \cs L_2} =\varphi_{A_{L_1}})(A_{L_2}) =
\varphi_{A_{L_2}}(A_{L_1}) \\ A_{L_1 \sqcup L_2} = (A_{L_1 \cs
L_2})^2.\end{array}$$
\end{lem}

Conway algebras from Examples \ref{Example 2.1.6}, \ref{Example
2.1.7}, and \ref{Example 2.1.8} satisfy the assumptions of Lemma
\ref{Lemma 3.22}.

\begin{prob}\label{Problem 3.23}\begin{enumerate}
\item[(a)]\label{Problem 3.23 1} Consider the equation $a|x = b$
in the universal Conway algebra.  Can is possess more than one
solution? (The equation $a_1 | x = a_2 $ has no solutions.)

\item[(b)]\label{Problem 3.23 2} Assume that for some diagrams of
links $L$ and $L'$ and for some crossings hold $L_+ \sim_S L'_+$
and $L_- \sim_S L'_-$.  Does the equality $L_\circ \sim_S
L'_\circ$ hold? \end{enumerate}
\end{prob}

The following theorem of S.~Bleiler and M.~Scharlemann \cite{B_S}
can be thought of as the first step to solve Problem \ref{Problem
3.23}(b).

\begin{thm}\label{Theorem 3.24} Let $L_+$, $L_-$, and $L_\circ$ be diagrams of links in
standard notation (Figure \ref{Fig L+ L- L0}).  Then
\begin{enumerate}

\item[(a)] If $L_+$ and $L_-$ represent trivial links then
$L_\circ$ also represents a trivial link.

\item[(b)] If $L_-$ and $L_\circ$ represent trivial links and we
consider a self-crossing of some component of $L_-$, then $L_+$ is
a trivial link.

\item[(c)] If $L_-$ and $L_\circ$ represent trivial links and we
consider a crossing of different components of $L_-$ then $L_+$ is
isotopic to the link which consists of Hopf link and a trivial
link (Figure \ref{Fig 3.14}).
\end{enumerate}
\end{thm}

\begin{figure}[htbp]
\centering
\includegraphics{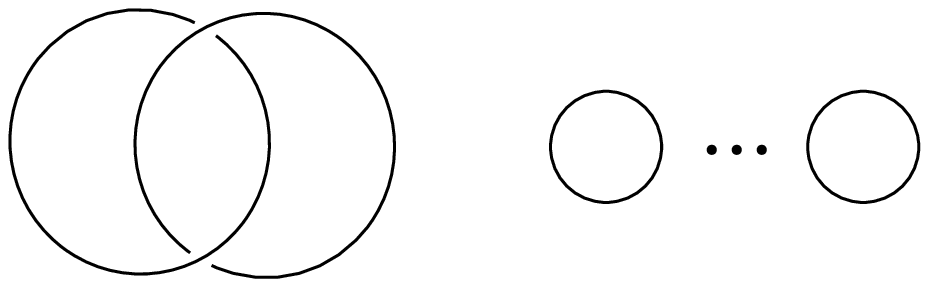}
\caption{}\label{Fig 3.14}
\end{figure}

For the proof, we refer to \cite{B_S}.

Lickorish and Millett \cite{Li_M_1} have generalized Corollary
\ref{Cor 3.20} into the case of the sum of (alternating) tangles
(see \cite{Li_1}).  This is the two-variable analogue of the
numerator-denominator formula of Conway \cite{Co} for the Conway
polynomial.

\begin{prop}\label{Prop 3.25}Let $A$ and $B$ be two alternating tangles.  Let $A+B$
denote the tangle of Figure \ref{Fig 3.15}.

\begin{figure}[htbp]
\centering
\includegraphics{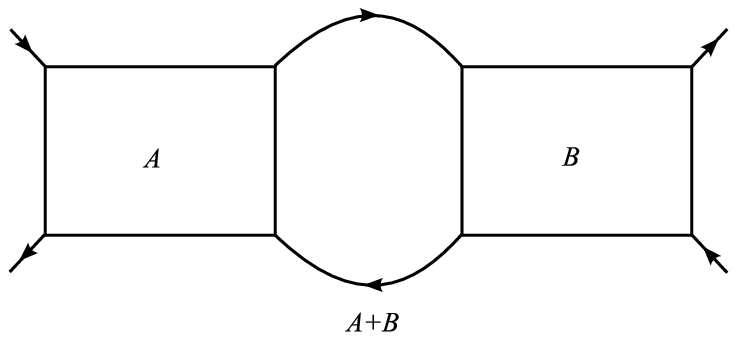}
\caption{}\label{Fig 3.15}
\end{figure}

The numerator of $A$, $N(A)$, is the link shown in Figure \ref{Fig
3.16}(a), and the denominator of $A$, $D(A)$, is shown in Figure
\ref{Fig 3.17}(b).

\begin{figure}[htbp]
\centering

\subfigure[]{\includegraphics{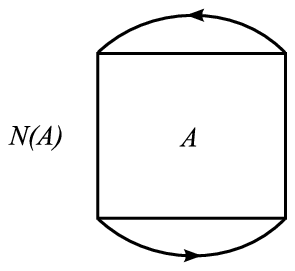}} \hspace{2.4cm}
\subfigure[]{\includegraphics{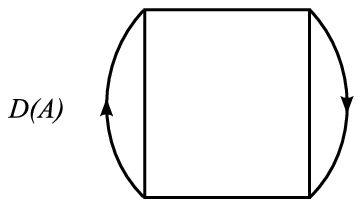}}

\caption{}\label{Fig 3.16}\label{Fig 3.17}
\end{figure}

Finally, $A^N$ and $A^D$ denote the values of the Jones-Conway
polynomial for $N(A)$ and $D(A)$ respectively.

Then
\begin{enumerate}
\item[(a)] $(1-(x+y)^2)(A+B)^N =
(A^NB^D+A^DB^N)-(x+y)(A^NB^N+A^DB^D)$

\item[(b)] $(A+B)^D=A^D B^D$.
\end{enumerate}
\end{prop}

\begin{proof}Part (b) is exactly Corollary \ref{Cor 3.20}.  To prove
part (a) of Proposition \ref{Prop 3.25} we use the induction on
$(\crs(B), b(B))$, (number of crossings in $B$, minimal number of
bad crossings in $B$) similarly as in Theorem \ref{Theorem 3.19}.
We can find, for the tangle $B$, a resolving tree the leaves of
which are the tangles shown on Figure \ref{Fig 3.18} possibly with
some trivial circles.  The same trivial circles appear in $A+B$ so
they can be omitted in further considerations.
\setcounter{figure}{17}

\begin{figure}
\centering
\includegraphics{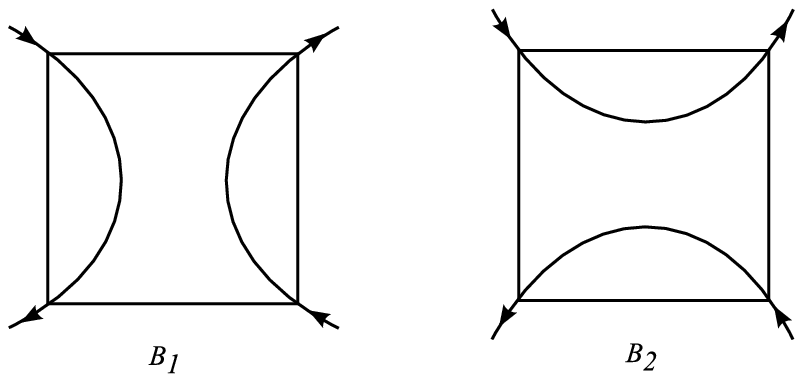}
\caption{}\label{Fig 3.18}
\end{figure}

$N(B_1)$ and $D(B_2)$ are trivial knots and $D(B_1)$ and $N(B_2)$
are trivial links of two components.  Furthermore $N(A+B_1) = D(A)$
and $N(A+B_2) = N(A)$.  Therefore $B^N_1 = B^D_2 = 1$, $B^D_1 =
B^N_2 = (x+y)$, $(A+B_1)^N = A^D$, and $(A+B_2)^N = A^N$.

From these it follows that $$(1-(x+y)^2)(A+B_1)^N = (1-(x+y)^2 A^D =
(A^N(x+y)+A^D)-(x+y)(A^N +A^D(x+y)).$$  Similarly,
$$(1-(x+y)^2)(A+B_2)^N = (1-(x+y)^2 A^N = (A^N+(x+y)A^D)-(x+y)((x+y)A^N
+A^D).$$

Thus we have proved Proposition \ref{Prop 3.25}(1) in the case of
$B=B_1$ or $B_2$.  Now immediate verification shows that if the
formula holds for $B_-$ and $B_\circ$ then it holds for $B_+$ and
similarly, if it holds for $B_+$ and $B_\circ$ then it holds for
$B_-$.  This allows us to perform the inductive step and complete
the proof of Proposition \ref{Prop 3.25}.\end{proof}

\begin{cor}[\cite{Co}]\label{Corollary 3.26} Let us define the
ration $F(A)$ of the tangle $A$ as follows:  $$F(A) =
\frac{\nabla_{N(A)}(z)}{\nabla_{D(A)}(z)},$$ where $\nabla(z)$ is
the Conway polynomial and the common factor of the numerator and the
denominator is not reduced.  Then $F(A+B) = F(A)+F(B)$.\end{cor}

\begin{example}\label{Example 3.27} Let $A$ be the tangle from
Figure \ref{Fig 3.19}.  The $F(A) = \frac{z}{1}$.\end{example}

\begin{figure}[htbp]
\centering
\includegraphics{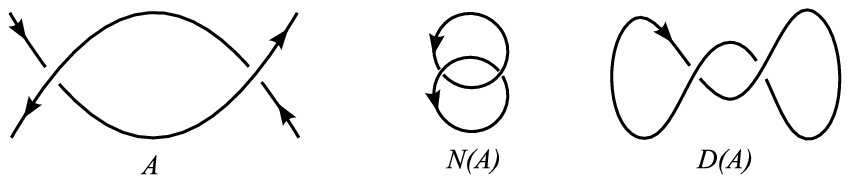}
\caption{}\label{Fig 3.19}
\end{figure}

\begin{prob}\label{Problem 3.28} Let $\mathcal A$ be a Conway
algebra for which there are the operation $\circ$ and the
homeomorphism $\varphi_w$.  Find the value of the invariant yielded
by the algebra for the numerator of a sum of two tangles.\end{prob}

J.~Birman \cite{Bi_2} (and independently M.~Lozano and H.~Morton)
found examples of knots which are not isotopic but which have the
same Jones-Conway polynomial.  Lickorish and Millett \cite{Li_M_1}
observed that these knots are not skein equivalent because they
have different signature.

Signature and its generalizations will be considered in the next
chapter, there we will show that the examples mentioned above are
algebraically equivalent (i.e. cannot be distinguished by any
invariant yielded by a Conway algebra).  We follow the paper
\cite{P_T_2}.

We will consider oriented links in the form of closed braids.  We
will use notation and terminology of Murasugi \cite{Mu_1} (see
also \cite{Bi_1}).  In particular for 3-braids $\Delta =
\sigma_1\sigma_2\sigma_1$ ($\sigma_1$ and $\sigma_2$ are shown on
Figure \ref{Fig 3.20}; the notations reflect the actual fashion
that positive braids have all crossings positive).

\begin{figure}[htbp]
\centering
\includegraphics{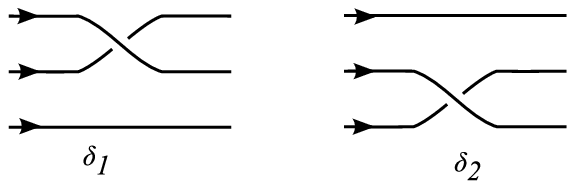}
\caption{}\label{Fig 3.20}
\end{figure}

We start from the first family of Birman examples.

\begin{thm}\label{Theorem 3.29} Let $\gamma$ be a 3-string braid
$\sigma^{a_1}_1 \sigma^{a_2}_2 \ldots \sigma^{a_{2k-1}}_1
\sigma_2^{a_{2k}}$ such that the sum of the exponents of $\gamma$,
$\displaystyle e(\gamma) = \sum^{2k}_{i=1} a_i$, is equal to 0. Then
the closed braid $\hat{\gamma}$ cannot be distinguished from its
mirror image $\hat{\overline{\gamma}}$ by the invariant yielded by
any Conway algebra.
\end{thm}
\begin{proof}Consider a Conway algebra $\mathcal{A}= (A, a_1, a_2,
\ldots,|,\ast)$.  First we will formulate a lemma which is crucial
to our proof of the theorem, then we will show how the theorem
follows from the lemma, and then we will prove the lemma.

We will use the following notation:  if $\gamma = \gamma_1
\sigma^a_i \gamma_2 \sigma^b_j \gamma_3$ is a 3-braid, then
$A_{a+p,b+q}$ denotes the value of the invariant of the closed
braid $\hat{\gamma}_{a+p, b+q}$, where $\gamma_{a+p,b+q} =
\gamma_1 \sigma^{a+p}_i \gamma_2 \sigma^{b+q}_j \gamma_3$.
Strictly speaking, $\gamma_1$, $\gamma_2$, $\gamma_3$, $i$, and
$j$ should be explicitly given in the notation, but we adopt a
rather informal convention of treating $a$ and $b$ as recognizing
signs for them. We will also use a natural convention of writing
$\gamma_{-b, -a}$ for the mirror image of $\gamma_{a,b}$.

We have the following obvious equalities: $$A_{a,b}=
A_{a-2,b}|A_{a-1,b} = (A_{a-2,b+2} \ast A_{a-2, b+1})|A_{a-1,b}$$
$$A_{c,d}=A_{c,d+2}\ast A_{c,d+1} = (A_{c-2,d+2}|A_{c-1,d+2})\ast
A_{c,d+1}.$$

Let us formulate our lemma.

\begin{lem}\label{Lemma 3.30}\begin{enumerate}
\item[(a)] If $A_{a-2, b+2} = A_{c-2,d+2}$ then we have the
following equivalence: $$(w_2 \ast A_{a-2, b+1})|h_2 =
(w_1|A_{c-1,d+2})\ast h_1 \iff (w_2 \ast A_{c-3, d+2})|h_2 =
(w_1|A_{a-2,b+3})\ast h_1,$$ where $w_1,w_2,h_1, h_2 \in A$.

\item[(b)] If $A_{a-2, b+2} = A_{c-2,d+2}$ and $A_{a-3,b+3} =
A_{c-3,d+3}$ then we have the following equivalence: $$(w_2 \ast
A_{a-2, b+1})|h_2 = (w_1|A_{c-1,d+2})\ast h_1 \iff (w_2 \ast
A_{a-4, b+3})|h_2 = (w_1|A_{c-3,d+4})\ast
h_1.$$\end{enumerate}\end{lem}

First, we show how to prove Theorem \ref{Theorem 3.29} using Lemma
\ref{Lemma 3.30}.  Let $\gamma$ be a cyclically reduced word,
$\gamma = \sigma^{a_1}_1 \sigma^{a_2}_2 \ldots \sigma^{a_{2k-1}}_1
\sigma_2^{a_{2k}}$ with $|a_i|>0$ and $e(\gamma)=0$ (the sum of
the exponents).  We define the complexity of $\gamma$ to be the
pair $(\crs(\gamma), p(\gamma))$, where $\crs(\gamma)$ is the sum
of absolute values of exponents of $\gamma$ (i.e. the number of
crossing points of the closed braid $\hat{\gamma}$) and
$p(\gamma)$ is the number of pairs of exponents $a_i,a_{i+1}$ in
$\gamma$ having the same sign (in the cyclic word, we consider
also the pair $a_{2k}, a_1$).  We will prove the theorem by
induction first on $\crs(\gamma)$, then on $p(\gamma)$.  The
theorem is obviously true for $\crs(\gamma)=0$.  Also for
$\crs(\gamma) = 2k$ and $p(\gamma) =0$, we have $\gamma$ isotopic
to its mirror image (we have a cyclic word of the form $1, -1, 1,
-1, \ldots$ in this case).

It is easy to see that if $\crs(\gamma)-2k+p(\gamma)>0$, then we
can choose $a$ and $b$ ($a>0$, $b<0$), two of the exponents of
$\gamma$, in such a way that either \begin{enumerate} \item[(a)]
$\crs(\gamma_{a-2,b+2})<\crs(\gamma)$

or

\item[(b)] $\crs(\gamma_{a-2, b+2}) = \crs(\gamma) \textrm{ and
}p(\gamma_{a-2,b+2})<p(\gamma).$\end{enumerate} In both cases we
have by inductive assumption that $A_{a-2, b+2} = A_{-b-2, -a+2}$
(according to the adopted notation, $A_{-b-2, -a+2}$ is the value
of the invariant for the mirror of $\hat{\gamma}_{a-2, b+2})$.  We
have also $$A_{a,b} = (A_{a-2,b+2}\ast A_{a-2, b+1})|A_{a-1, b}$$
$$A_{-b,-a} = (A_{-b-2,-a+2}\ast A_{-b-1, -a+2})|A_{-b,-a+1},$$
and we want to prove $A_{a,b}=A_{-b, -a}$.

\setcounter{equation}{30}

Let us consider the cyclic word $\gamma_{a-2, b+1}$.  We have
either that (a) it consists of one letter, and then $A_{a-2,
b+1}=A_{-b-1, -a+2}=a_2 \in A$, and we need to prove the equality
\begin{equation}\label{Equation 3.31} (w\ast a_2)|A_{a-1,b} = (w|a_2)\ast A_{-b,
-a+1},\end{equation} where $w=A_{a-2, b+2} =A_{-b-2, -a+2}$, or

(b) it has exponents $p$ and $q$ such that $|p|\geq 2$, $p-q<0$. We
will consider the case $p\geq 2$, $q\leq -1$.

We will now use symbols like $A_{p+x, q+y}$ for values of the
invariant for closed braids obtained from $\gamma_{a-2, b+1}$ (not
from $\gamma$!) by changing the exponents $p$ and $q$, we also use
$\gamma_{p+x, q+y}$ for the related braids.  Using this notation we
obtain $$A_{a-2, b+1}=A_{p,q},$$ and by the inductive assumption
$$A_{p, q+1}= A_{-q-1,-p},$$ $$A_{p-1, q+2}=A_{-q-2, -p+1}, $$
because $\crs(\gamma_{p, q+1})$, $\crs(\gamma_{p-1,q+2})$ <
$\crs(\gamma)$ and $e(\gamma_{p,q+1})=e(\gamma_{p-1, q+2})=0$.

We are now in a position to apply Lemma \ref{Lemma 3.30}(2) taking
$a=p+2$, $b=q-1$, $c=-q+1$, $d=-p-2$, $w_1=w_2=A_{a-1, b+2}=A_{-b-2,
-a+2}$, $h_2=A_{a-1, b}$, and $h_1 = A_{-b, -a+1}$.  We obtain the
following equivalence: $$ (w_2 \times A_{p,q})|h_2 =  (w_1|A_{-q,
-p}\ast h_1 \iff (w_2 \times A_{p-2, q+2})|h_2 = (w_1 |A_{-q-2,
-p+2})\ast h_1.$$

We can repeat the procedure until we are reduced to proving the
equality \ref{Equation 3.31}.  The same argument works for $p \geq
1$, $q \leq -2$, the only change is that we are able in this case to
diminish $|q|$, not $|p|$.  In order to prove \ref{Equation 3.31}
let us consider equalities $(w\times A_{a-1,b})|A_{a-1,b} = w =
(w|A_{-b, -a+1}) \ast A_{-b,-a+1}$, which are true by C6 and C7.
Applying Lemma \ref{Lemma 3.30} in a manner similar to the above
used we obtain a sequence of equivalent equalities ending with
\ref{Equation 3.31}, thus \ref{Equation 3.31} is true, which
completes the proof of Theorem \ref{Theorem 3.29}.

It remains only to prove Lemma \ref{Lemma 3.30}.

Consider the equality \begin{equation}\label{Equation 3.32} (w_2
\ast A_{a-2, b+1}|h_2 = (w_1|A_{c-1, d+2})\ast h_1.\end{equation}

Multiplying both sides of \ref{Equation 3.32} by $|h_2)\ast A_{c-1,
d+2}$, we get an equality equivalent to \ref{Equation 3.32}:
\begin{equation}\label{Equation 3.33} (((w_2 \ast A_{a-2,
b+1})|h_2)|h_1) \ast A_{c-1, d+2} = w_1\end{equation} (we applied C7
and C6 to the right side).

We will now consider a series of equalities obtained by transforming
the formula for the left side of \ref{Equation 3.33}.  This will be
done by applying transposition properties, C6 and C7.  For the
reader's convenience the axiom applied will always be marked. In
case of transpositions, we will also mark the elements to be
transposed by setting them in boldface.

The left side of \ref{Equation 3.33} is equal to (we begin by
replacing $A_{c-1, d+2}$ with $A_{c-3, d+2}|A_{c-2, d+2}$).  Then we
have \begin{eqnarray*} & [((w_2*(A_{a-2, b+3} \ast A_{a-2,
b+2}))|h_2)|\mathbf{h_1}]\ast [\mathbf{A_{c-3, d+2}}|A_{c-2, d+2}] \\
\overset{C3}{=} & (((w_2\ast (A_{a-2,b+3} \ast A_{a-2, b+2}))|h_2)
\ast A_{c-3, d+2})|(h_1 \ast A_{c-2, d+2}) \\ \overset{C7}{=} &
([(w_2 \ast (A_{a-2, b+3} \ast A_{a-2, b+2}))|\mathbf{h_2}]\ast
[\mathbf{(A_{c-3,
d+2}\ast A_{c-2, d+2})}|A_{c-2, d+2}])|(h_1\ast A_{c-2,d+2}) \\
\overset{C4}{=} & (([w_2\ast \mathbf{(A_{a-2,b+3}\ast A_{a-2,
b+2})}]\ast [\mathbf{A_{c-3, d+2}} \ast A_{c-2,d+2}])|(h_2\ast
A_{c-2, d+2}))|(h_1 \ast A_{c-2, d+2}) \\ \overset{C5}{=} & ([(w_2
\ast A_{c-3, d+2}) \ast \mathbf{((A_{a-2,b+3}\ast A_{a-2, b+2}) \ast
A_{c-2,d+2})}]|(\mathbf{h_2}\ast A_{c-2, d+2}))|(h_1 \ast A_{c-2, d+2}) \\
\overset{C4}{=} & ([(w_2 \ast A_{c-3, d+2})
|h_2]\ast(((A_{a-2,b+3}\ast A_{a-2, b+2}) \ast A_{c-2,d+2})| A_{c-2,
d+2}))|(h_1 \ast A_{c-2, d+2})\\ \overset{C7}{=} & [((w_2 \ast
A_{c-3, d+2})|h_2)\ast \mathbf{(A_{a-2, b+3}\ast A_{a-2,
b+2})}]|[\mathbf{h_1}\ast A_{c-2, d+2}] \\ \overset{C4}{=} & (((w_2
\ast A_{c-3, d+2})|h_2) |h_1 )\ast ((A_{a-2, b+3}\ast A_{a-2,
b+2})|A_{c-2, d+2})
\end{eqnarray*}

But we have assumed $A_{a-2, b+2} = A_{c-2, d+2}$, so applying C7 we
obtain $(((w_2 \ast A_{c-3, d+2})|h_2)|h_1)\ast A_{a-2, b+3}$ equal
to the left side of \ref{Equation 3.33}.

Thus \ref{Equation 3.32} is equivalent to $$(w_2\ast A_{c-3,
d+2})|h_2 = (w_1 |A_{a-2, b+3})\ast h_1$$ which completes the proof
of Lemma \ref{Lemma 3.30}(1).  If we repeat the above argument once
more we will get Lemma \ref{Lemma 3.30}(2).

\end{proof}

If $\gamma$ from Theorem \ref{Theorem 3.29} has normal form
$\Delta^{2n}\gamma_\circ$, $n\neq 0$. and $\gamma \in \Omega_6$
(i.e. $\gamma = \Delta^{2n}\sigma^{-p_1}_1 \sigma^{q_1}_2 \ldots
\sigma^{p_k}_1 \sigma_2^{q_k}$, $p_i, q_i, k > 0$ \cite{Mu_1}),
then $\hat{\gamma}$ has non-zero signature (see \cite{Mu_1, Bi_2};
it can be shown that $\epsilon$ from Proposition 11.1 of
\cite{Mu_1} is equal to 0).

Furthermore, the determinants of links from $\Omega_6$ are not zero
\cite{Mu_1}, so the signature of these links is a skein invariant
(see part 4).  Theorem \ref{Theorem 3.29} gives a class of pairs of
links which are not skein equivalent but have the same invariant in
every Conway algebra.

We can work similarly with links from Proposition 2 and Lemma 4 of
\cite{Bi_2}.

\setcounter{thm}{33}

\begin{thm}\label{Theorem 3.34}Let $\gamma$ be a 3-string braid $\sigma^{a_1}_1
\sigma^{a_2}_2 \ldots \sigma^{a_{2k-1}}_1 \sigma^{a_{2k}}_2$ such
that $e(\gamma)=6r$.  Let $\mathcal{B} = \Delta^{4r}\gamma^{-1}$.
Then the closed 3-braid $\hat{\gamma}$ cannot be distinguished
from $\hat{\mathcal B}$ by the invariant yielded by any Conway
algebra.
\end{thm}

\begin{proof} For $e(\gamma)=0$ it is Theorem \ref{Theorem 3.29}.
Assume $e(\gamma) > 0$ (the case $e(\gamma)<0$ is quite
analogous). Theorem \ref{Theorem 3.34} holds for the link of three
components $\gamma = \Delta^{2r}$ and for the knot
$\Delta^{2(r-1)}\sigma^2_1\sigma^2_2\sigma_1\sigma_2$
($\hat{\gamma}$ is isotopic to $\hat{\mathcal{B}}$ in both cases).
Now we will proceed by induction on some ``complication'' which
measures the distance between given $\gamma$ and $\Delta^{2r}$ or
$\Delta^{2(r-1)}\sigma^2_1\sigma^2_2\sigma_1\sigma_2$.  Namely,
our complication associated to a word $\gamma$ is the triplet
$(\crs(\gamma), s(\gamma), d(\gamma))$ where $s(\gamma)$ is equal
to $e(\gamma)-p$ ($p$ denotes the number of monomials in the
cyclically reduced word $\gamma$), and $d(\gamma)$ is the number
of exponents $a_i$ equal to 2.

The rest of the proof reminds that of Theorem \ref{Theorem 3.29} but
differs in details.  We will give a sketch of the proof.

If $\crs(\overline{\gamma}) = e(\gamma)$ and $s(\gamma) \leq 2$.  We
have done (the theorem holds for $\Delta^{2r}$,
$\Delta^2(r-1)\sigma^3_1\sigma_2\sigma_1\sigma_2$ and its conjugates
in $B_3$, $\Delta^2(r-1)\sigma^2_1\sigma^2_2\sigma_1\sigma_2$).
Consider $\gamma$ with $\crs(\gamma)>e(\gamma)$ or $\crs(\gamma)=
e(\gamma)$ with $s(\gamma)> 2$, and assume Theorem \ref{Theorem
3.34} holds for $\gamma$'s with a smaller complication.  Now either
$\gamma$ is conjugate in $B_3$ to a word with a smaller complication
or we can write $\gamma =
\gamma_1\sigma^a_i\gamma_2\sigma^b_j\gamma_3$ in such a way that
$\gamma^{a-2,b+2}$ has a smaller complication that $\gamma$. Namely,
let the cyclically reduced $\gamma = \sigma^{a_1}_{\epsilon(1)}
\sigma^{a_2}_{\epsilon(2)} \ldots \sigma^{a_p}_{\epsilon(p)}$.  Then
either:
\begin{enumerate}

\item[(i)] $\crs(\gamma) > e(\gamma)$. Then there exist $i$ and
$j$ such that $a_i \geq 2$, $a_j < -1$ and we put $a=a_i$,
$b=a_j$. Then $\crs(\gamma_{a-2, b+2} \leq \crs(\gamma)-2$.

\item[(ii)] All $a_t >0$ and there exists $i$ with $a_i \geq 4$.
Then we write $\gamma $ in the form\\
$\gamma =
\sigma^{a_1}_{\epsilon(1)} \ldots \sigma^{2}_{\epsilon(i)}
\sigma_{\epsilon(i)} \sigma_{\epsilon '(i)}^0
\sigma_{\epsilon(i)}^{a_i-3} \ldots \sigma_{\epsilon(p)}^{a_p}$
where $$\epsilon '(i) = \left\{\begin{array}{ll} 1 & \textrm{iff }
\epsilon(i)=2 \\ 2 & \textrm{iff } \epsilon(i)=1
\end{array}\right. ,$$ and we put $a=2$, $b=0$.  Then
$\crs(\gamma_{a-2,b+2}) = \crs(\gamma)$ and $s(\gamma_{a-2, b+2})
= s(\gamma)-2$.

\item[(iii)] All $a_t>0$ and there exist $i,j$, $i\neq j$, such
that $a_i \geq 3$, $a_j \geq 2$.  Then we write $\gamma$ in the
form $\gamma = \sigma^{a_1}_{\epsilon(1)} \ldots
\sigma^{a_i}_{\epsilon(i)}\ldots \sigma_{\epsilon(j)}
\sigma_{\epsilon '(j)}^0 \sigma_{\epsilon(j)}^{a_i-1} \ldots
\sigma_{\epsilon(p)}^{a_p},$ and we put $a = a_i$, $b=0$.  Then
$\crs(\gamma_{a-2,b+2}) = \crs(\gamma)$ and $s(\gamma_{a-2, b+2})
= s(\gamma)-2$.

\item[(iv)] All $a_t$ are equal to 2 or 1, $p>2$ and in the
cyclically reduced word $\gamma$ there exist $i$ and $j$, $i, i+1
\neq j$, such that $a_i = a_{i+1} = a_j = 2$.  Then $$\gamma =
\sigma^{a_1}_{\epsilon(1)} \ldots \sigma^{a_i}_{\epsilon(i)}
\sigma^{a_{i+1}}_{\epsilon(i+1)} \sigma^{a_{i+2}}_{\epsilon(i+2)}
\ldots \sigma_{\epsilon(j)} \sigma^{\circ}_{\epsilon(j)}
\sigma_{\epsilon(j)} \ldots \sigma^{a_p}_{\epsilon(p)}$$ and we
put$a=a_{i+1} = 2$, $b=0$.  Then $\crs(\gamma_{a-2, b+2}) =
\crs(\gamma)$, $s(\gamma_{a-2, b+2}) = s(\gamma)$ and
$d(\gamma_{a-2, b+2}) = d(\gamma)-2$.

\item[(v)] All $a_t$ are equal to 1 or 2 and in the cyclically
reduced word $\gamma$ there is no $i$ such that $a_i = a_i+1 =2$.
Then $\gamma$ is conjugated in $B_3$ to a word with a smaller
complication or to a word with the same complication but which
satisfies (3) above.  (We use the following equalities in $B_3$:
$$\sigma^2_1\sigma_2 \sigma^2_1 =
\sigma_1\sigma_2\sigma_1\sigma_2\sigma_1,~ \sigma^2_1\sigma_2
\sigma_1 = \sigma_1\sigma_2\sigma_1^2.)$$

It excludes all possibilities of $\gamma$ with $\crs(\gamma) >
e(\gamma)$ or $\crs(\gamma) = e(\gamma)$ and $s(\gamma)>2$.
\end{enumerate}

Now we can use Lemma \ref{Lemma 3.30} exactly in the same way as in
the proof of Theorem \ref{Theorem 3.29}.

\end{proof}

There is a reasonable hope that Lemma \ref{Lemma 3.30} can be used
to show that many pairs of links (not necessarily closed 3-braids)
cannot be distinguished by the invariant yielded by any Conway
algebra (for example the closures of four string braids $K_a =
\sigma^{-a}_2 \overline{\sigma}_3 \sigma_1 \overline{\sigma}_2
\sigma^{a-1}_3 \sigma_1 \overline{\sigma}_2 \sigma_3$ and their
mirror images; $K_3 = \overline{9}_{42}$ (in the Rolfsen notation
\cite{Ro}), $K_4 = \overline{11}_{449}$ (in the Thistlethwaite
notation \cite{Thist_2})).

\begin{prob}\begin{enumerate}\item[(a)] Do there exist two links $L_1$
and $L_2$ which have the same values of the Jones-Conway
polynomial but are not algebraically equivalent (i.e. can be
distinguished by some invariant yielded by a Conway algebra).

\item[(b)] Do there exist two links which are not algebraically
equivalent but which have the same value of the invariant yielded
by any finite Conway algebra.

\item[(c)] Let $\gamma = \sigma^{-p_1}_1 \sigma^{q_2}_2 \ldots
\sigma^{-p_k}_1 \sigma^{q_k}_2$ ($p_i, q_i, k>0$) be an
alternating 3-braid with $e(\gamma) =0$.  Whether $\hat{\gamma}$
is skein equivalent to its mirror image?
\end{enumerate}\end{prob}

There is known many algebraic properties of the Jones-Conway
polynomial.  They relate, mainly, special substitutions in
polynomial with old invariants of links (\cite{Li_M_1},
\cite{Li_M_2}, \cite{Mur-1}, \cite{Mo_2}, \cite{F_W}).  Here we will
state two elementary properties of the Jones-Conway polynomial which
will be useful later.

\begin{lem}\label{Lemma 3.36}\begin{enumerate} \item[(a)] If $L$ is a link of
odd number of components then all monomials in the Jones-Conway
polynomial $P_L(x,y)$ are of even degree.  If $L$ has an even
number of components then these monomials are of odd degree.

\item[(b)] For every link $L$, $x+y-1$ divides $P_L(x,y) -1$.  In
particular, $P_L(x,y)$ cannot be always equal to
0.\end{enumerate}\end{lem}

\begin{proof}The conditions (a) and (b) hold easily for trivial
links.  Then it is enough to verify that if they hold for $L_-$ and
$L_\circ$ (respectively, $L_+$ and $L_\circ$) then they hold for
$L_+$ (respectively, $L_-$).

\end{proof}

\setcounter{equation}{36}

It is reasonable to try to generalize the Jones-Conway polynomial
by considering the following equation instead of the equation
\ref{Equation 2.1.9}, \begin{equation}\label{Equation 3.37} xw_1
+yw_2 = w_0 - z.\end{equation}

In fact it leads to 3-variable polynomial invariant of links but
this polynomial does not distinguish anything more than the
original Jones-Conway polynomial (it was observed by the referee
of \cite{P_T_1} and later but independently by O.Ya.~Viro
\cite{Vi}). Namely:

\setcounter{thm}{37}

\begin{prop}\label{Prop 3.38} \begin{enumerate}\item[(a)] The following
$\mathcal{A} = \{A, a_1, a_2, \ldots, |, \ast\}$ is a Conway
algebra.  $A = \ints[x^{\mp}, y^{\mp}, z]$, $a_1=1$, $a_2 = x+y+z$,
$\ldots$, $a_i = (x+y)^{i-1} + z(z+y)^{i-2} + \ldots + z(x+y) + z$,
$\ldots$~.

We define $|$ and $\ast$ as follows:  $w_2|w_0  =w_1$, and $w_1 \ast
w_0 = w_2$ where $xw_1 + yw_2 = w_0-z$, $w_1, w_2, w_3 \in A$.

\item[(b)] The invariant of links $w_L(x,y,z)$ yielded by the
Conway algebra $\mathcal A$ satisfies $$w_L(x,y,z) = w_L(x,y,0) +
z \left(\frac{w_L(x,y,0) -1}{x+y-1}\right), \textrm{ and}$$
$$w_L(x,y,0) = P_L(x,y).$$
\end{enumerate}\end{prop}

\begin{proof}\begin{enumerate} \item[(a)] We check conditions C1-C7 of Conway algebra
(compare Examples \ref{Example 2.1.8} and \ref{Example 4.5}).

\item[(b)] $a_i = (x+y)^{i-1} +z
\left(\frac{(x+y)^{i-1}-1}{x+y-1}\right)$, so for trivial links
the equalities (from (b)) hold.  Then, as usual, we can easily
verify that if they hold for $L_-$ and $L_\circ$ (respectively
$L_+$ and $L_\circ$) then they hold for $L_+$ (respectively
$L_-$).
\end{enumerate}\end{proof}

\begin{rem}Each invariant of links can be used to build a better
invariant which will be called weighted simplex of the invariant.
Namely, if $w$ is an invariant and $L$ is a link of $n$ components
$L_1, \ldots, L_n$, then we consider an $n-1$ dimensional simplex
$\Delta^{n-1}=(q_1, \ldots, q_n)$. We associate with each face
$(q_{i_1}, \ldots, q_{i_k})$ of $\Delta^{n-1}$ the value $w_{L'}$,
where $L'= L_{i_1} \cup \ldots \cup L_{i_k}$.  We say that two
weighted simplicies are equivalent if there exists a bijection of
their vertices which preserves weights of faces.  Of course, the
weighted simplex of an invariant of isotopy classes of oriented
links is also an invariant of isotopy classes of oriented
links.\end{rem}

\begin{example}\label{Example 3.40}\begin{enumerate}\item[(a)] Two links
shown in Figure \ref{Fig 3.1} are skein equivalent but they can be
distinguished by weighted simplices of the global linking numbers
(see Example \ref{Example 3.5}).

\item[(b)] The link (closed 3-braid) $\hat{\gamma}$ (see Figure
\ref{Fig 3.21}) where $$\gamma = \sigma^{-2}_1 \sigma^{3}_2
\sigma^{-2}_1 \sigma_2 $$ ($8^3_2$ in Rolfsen \cite{Ro} notation) is
algebraically equivalent to its mirror image
$\hat{\overline{\gamma}}$ (see Theorem \ref{Theorem 3.29}) and has
the same signature as $\hat{\overline{\gamma}}$.  However,
$\hat{\gamma}$ and $\hat{\overline{\gamma}}$ can be distinguished by
weighted simplices of the global linking numbers.

\begin{figure}
\centering
\includegraphics[width=0.6\textwidth]{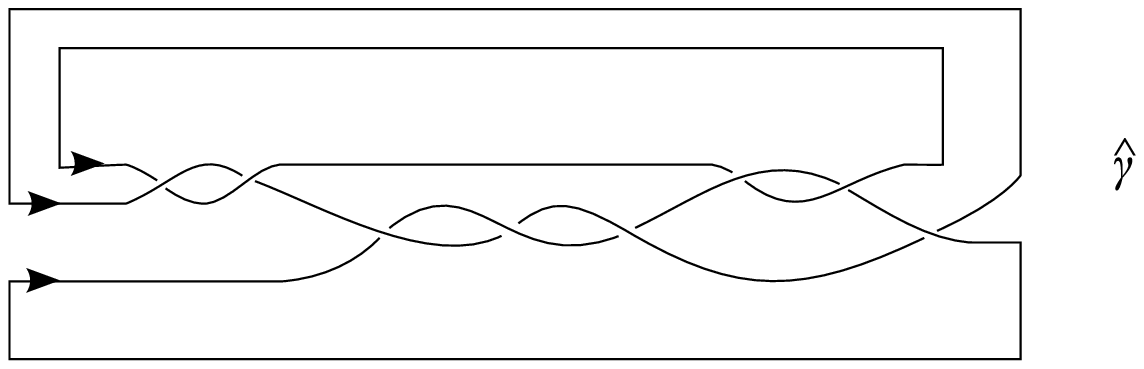}
\caption{}\label{Fig 3.21}
\end{figure}

\item[(c)] J.~Birman \cite{Bi_2} has found three-braids $$\gamma_1
= \sigma^{-2}_1 \sigma^{3}_2 \sigma^{-1}_1 \sigma^4_2
\sigma^{-2}_1\sigma^4_2 \sigma^{-1}_1 \sigma_2$$
$$\gamma_2 = \sigma^{-2}_1 \sigma^{3}_2 \sigma^{-1}_1 \sigma^4_2
\sigma^{-1}_1\sigma_2 \sigma^{-2}_1 \sigma^4_2$$ which closures
are algebraically equivalent and have the same signature but which
can be distinguished by weighted simplices of the global linking
numbers.

\end{enumerate}\end{example}

Another method of distinguishing knots was analyzed by Morton and
Short \cite{Mo_S}.  They considered the Jones-Conway polynomial of
$(2,k)$-cables along knots (2 was chosen because of limited
possibility of computers).  They made many calculations and got very
interesting experimental material.  In particular, they found that
using their method they were able to distinguish some Birman and
Lozano-Morton examples (all which they tried) and the $9_{42}$ knot
from its mirror image.  On the other hand, they were unable to
distinguish the Conway knot and the Kinoshita-Teresaka knot.  Other
pairs of mutants were tried with similar result.  The above finding
of Morton and Short was the motivation for the author of the survey
to prove the following theorem.

\begin{thm}\label{Theorem 3.41}  Consider the tangles (a)~~\tanglea,
and (b)~~\tangleb.  Let the diagram $L'$ of a knot be any mutation
along the tangle (a) or by a mutation along the tangle (b) such that
it consists of a rotation of angle $180^{\circ}$ about the central
axis (perpendicular to the plane of the diagram).  Then the
$(2,k)$-cable along $L$ is skein equivalent to the $(2,k)$-cable
along $L'$ for any $k$.
\end{thm}

For the proof, we refer to \cite{P_2}.

Despite the above theorem we still feel confident to propose the
following conjecture.

\begin{conj}\label{Conjecture 3.42} For any two non-isotopic
prime\footnote{Added for e-print: it should be ``simple". Already
in the final version of \cite{P_2} we proved that cables of $K_1\#K_2$,
and $K_1\#-K_2$, which can be different prime knots, have the same
 Jones-Conway polynomial. With the substitution of simple in place of
prime, Conjecture 3.42 remains open.}
knots there exist numbers $p$ and $q$ such that the $(p,q)$-cables
along these knots can be distinguished by the Jones-Conway
polynomial.\end{conj}

\section{Partial Conway algebras}

It can be observed that in order to get a link invariant it is not
necessary to have the operations $|$ and $\ast$ defined on the whole
product $A\times A$ and relations C3--C5 need not be satisfied by
all elements of $A \times A\times A\times A$.  We refer here results
from \cite{P_T_1} and \cite{P_T_2}.

Let us begin with the following definition:

\begin{defn}\label{Definition 4.1} A partial Conway algebra
$\mathcal A$ is a quadruple $(A, B_|, B_\ast, D)$, $B_|$ and
$B_\ast$ being subsets of $A\times A$, and $D$ of $A \times A\times
A\times A$ together with $0$-argument operations $a_1, a_2, \ldots$,
and two 2-argument operations $|$ and $\ast$ defined on $B_|$ and
$B_\ast$ respectively, satisfying conditions C1--C7 whenever both
sides of equations are defined and $(a,b,c,d)\in D$ in case of
relations C3--C5.\end{defn}

We would like to construct invariants of links using such partial
algebras.

\begin{defn} \label{Definition 4.2} We say that a partial Conway algebra $\mathcal{A}=(A,
B_|, B_\ast, D; a_1, a_2, \ldots, |, \ast)$ is geometrically
sufficient iff the following two conditions are satisfied:
\begin{enumerate} \item[(i)] for every resolving tree of a link all the
operations that are necessary to compute the root value are defined,

\item[(ii)] let $p_1$ and $p_2$ be two crossings of a diagram $L$;
consider the diagrams $$L^{p_1, p_2}_{\eps_1, \eps_2}, L^{p_1,
p_2}_{\eps_1, 0}, L^{p_1, p_2}_{0, \eps_2}, \textrm{ and }L^{p_1,
p_2}_{0,0},$$ where $\eps_i = -\sgn(p_i)$ ($\sgn(p_i)$ denotes the
sign of the crossing $p_i$ in the original diagram $L$), and
choose for them resolving trees $T_{p_1,p_2}$, $T_{p_1,0}$,
$T_{0,p_2}$, and $T_{0,0}$, respectively.  Denote the root values
of these trees by $w_{p_1, p_2}$, $w_{p_1, 0}$, $w_{0,p_2}$, and
$w_{0,0}$, respectively. Then $(w_{p_1, p_2},w_{p_1,
0},w_{0,p_2},w_{0,0}) \in D$. (The condition (ii) means that the
resolving trees of $L$ from Figure \ref{Fig 4.1} give the same
value at the root of the tree.)
\end{enumerate}\end{defn}
\begin{figure}
\centering

\subfigure[]{\xymatrix@C=0pt{ & & & \sgn p_1 \ar@{-}[dll] \ar@{-}[drr] & & & \\
 & \sgn p_2 \ar@{-}[dl] \ar@{-}[dr] & & & & \sgn p_2 \ar@{-}[dl] \ar@{-}[dr] & \\
 T_{p_1, p_2} & & T_{p_1,0} & & T_{0,p_2} & & T_{0,0} }}
\vspace{1cm}
 \subfigure[]{\xymatrix@C=0pt{ & & & \sgn p_2 \ar@{-}[dll] \ar@{-}[drr] & & & \\
 & \sgn p_1 \ar@{-}[dl] \ar@{-}[dr] & & & & \sgn p_1 \ar@{-}[dl] \ar@{-}[dr] & \\
 T_{p_1, p_2} & & T_{0,p_2} & & T_{p_1,0} & & T_{0,0} }}

\caption{}\label{Fig 4.1}
\end{figure}

The proof of Theorem \ref{Theorem 2.1.1} can be used, without
changes, in the case of a geometrically sufficient partial Conway
algebra.

\begin{thm} Let $\mathcal A$ be a geometrically sufficient partial
Conway algebra.  There exists a unique invariant $w$ attaching to
each skein equivalence class of links an element of $A$ and
satisfying the following conditions: \begin{enumerate}
\item $w_{T_n} = a_n$

\item $w_{L_+} = w_{L_-}|w_{L_\circ}$

\item $w_{L_-} = w_{L_+}\ast w_{L_\circ}$
\end{enumerate}\end{thm}

The conditions C1--C7 in a partial Conway algebra are not totally
independent of one another.  We can prove, similarly as Lemma
\ref{Lemma 2.1.3} the following fact.

\begin{lem} Let $(A, B_|, B_\ast, a_1, a_2, \ldots, |, \ast)$ be a
partial algebra, such that:

\begin{enumerate}
\item[(i)] The property (i) of Definition \ref{Definition 4.2} is
satisfied.

\item[(ii)] The property (ii) of Definition \ref{Definition 4.2}
is satisfied for each pair of crossings of positive sign; i.e. the
resolving trees of the diagram $L$ (Figure \ref{Fig 4.1}) gives
the same value $w$ at the roots if $\sgn p_1 = \sgn p_2 = +$.

\item[(iii)] The conditions C1, C6, and C7 are satisfied if both
sides of the equations are defined. \end{enumerate}

Define $D$ to be the subset of $A\times A\times A\times A$ for which
the conditions C3--C5 are satisfied.  Then $\mathcal A = (A, B_|,
B_\ast, D, a_1, a_2, \ldots, |, \ast)$ is a geometrically sufficient
partial Conway algebra.\end{lem}

Now we will describe three examples of geometrically sufficient
partial Conway algebras and we will discuss the knot invariant s
yielded by them.

Let us start with the example which gives a strict generalization
of the Jones-Conway polynomial.  The idea is to use instead of the
equations \ref{Equation 2.1.9} or \ref{Equation 3.37} the family of
equations (it depends on the number of components of $L_+$, $L_-$,
$L_\circ$ which an equation is used).

\begin{example}\label{Example 4.5}  The following partial algebra $\mathcal A$ is a
geometrically sufficient partial Conway algebra.

$$A = \nats \times \ints[x^{\mp 1}_1, z_1, x^{\mp 1}_2, z_2, x^{\mp
1}_3, z_3,\ldots, y^{\mp 1}_1, x'^{\mp 1}_2, z'_2],$$

$$B_| = B_\ast = \left\{ ((n_1,w_1), (n_2, w_2)) \in A\times A ~:~
|n_1 - n_2|=1 \right\},$$

$$D = A\times A \times A \times A; ~~ a_1 = (1,1), a_2 = (2, x_1
+y_1 +z_1), \ldots,$$

$$a_n=\left(n, \prod^{n-1}_{i=1}(x_i+y_i) + z_1 \prod^{n-1}_{i=2}(x_i
+y_i)+\ldots + z_{n-2}(x_{n-1}+y_{n-1}) + z_{n-1}\right)$$

where $y_i = x_i \frac{y_1}{x_1}$. To define the operations $|$
and $\ast$ consider the following system of equations:
\begin{enumerate}
\item[($1$)] $x_1w_1 + y_1w_2 = w_0 -z_1$

\item[($2$)] $x_2w_1 +y_2w_2=w_0-z_2$

\item[($2'$)] $x'_2w_1+y'_2w_2 = w_0 -z'_2$

\item[($3$)] $x_3w_1 + y_3w_2 = w_0 -z_3$

\item[($3'$)] $x'_3w_1 + y'_3w_2 = w_0 - z'_3$

\end{enumerate}

\hspace{1.5cm}$\ldots\ldots\ldots\ldots$

\begin{enumerate}

\item[($i$)]  $x_iw_1 +y_iw_2 = w_0 - z_i$

\item[($i'$)] $x'_iw_1 + y'_iw_2 = w_0-z'_i$

\end{enumerate}

\hspace{1.5cm}$\ldots\ldots\ldots\ldots$\medskip

where $y'_i = \frac{x'_iy_1}{x_1}$, $x'_i =
\frac{x'_2x_1}{x_{i-1}}$, and $z'_i$ are defined inductively to
satisfy
$$\frac{z'_{i+1}-z_{i-1}}{x_1x'_2} =
\left(1+\frac{y_1}{x_1}\right)\left(\frac{z'_i}{x'_i}
-\frac{z_i}{x_i}\right).$$

we define $(n,w) = (n_1, w_1)|(n_2, w_2)$ (respectively,
$(n,w)=(n_1,w_1)\ast (n_2, w_2)$) as follows: $n=n_1$ and if $n_1
= n_2-1$ then we use equation ($n$) to get $w$; namely $x_nw +
y_nw_1 = w_2 - z_n$ (respectively, $x_nw_1 + y_nw = w_2 - z_n$).
If $n_1 = n_2 +1$ then we use the equation ($n'$)  to get $w$,
namely $x'_n w +y'_nw_1 = w_2 - z'_n$ (respectively, $x'_n w_1 +
y'_n w = w_2-z'_n$).

Now we will show that $\mathcal A$ is a geometrically sufficient
partial Conway algebra.

It is an easy task to check that the first coordinate of elements
from $A$ satisfies C1--C7 (compare Example \ref{Example 2.1.5})
and to check the relations C1, C2, C6, and C7 so we will
concentrate our attention on the relations C3, C4, and C5.

It is convenient to use the following notation:  if $w\in A$ then
$w = (|w|, F_w)$ and for $$w_1|w_2 = (|w_1|, F_{w_1})|(|w_2|,
F_{w_2}) = (|w|, F_w) = w$$ to use the notation $$F=\left\{
\begin{array}{ll} F_{w_1}|_n ~ F_{w_2} & \textrm{ if }
n=|w_1|=|w_2|-1 \\ F_{w_1} |_{n'} ~ F_{w_2} & \textrm{ if } n =
|w_1| = |w_2|+1 .\end{array}\right. $$

Similar notation we use for the operation $\ast$.

In order to verify relations C3--C5 we have to consider three main
cases: \begin{enumerate}

\item $|a|=|c|-1 = |b|+1 =n$.

Relations C3--C5 make sense iff $|d| = n$.  The relation C3 has the
form: $$ (F_a|_{n'} F_b)|_n (F_x|_{(n+1)'} F-d) = (F_a |_n F_c)
|_{n'} (F_b |_{n-1} F_d).$$  From this we get
\begin{eqnarray*}\frac{1}{x_n x'_{n+1}}F_d - \frac{y'_{n+1}}{x_n
x'_{n+1}}F_c -\frac{y_n}{x_n x'_n}F_b + \frac{y_n y'_n}{x_n x'_n}
F_a - \frac{z'_{n+1}}{x_n x'_{x+1}} - \frac{z_n}{x_n} + \frac{y_n
z'_n}{x_n x'_n} \\ = \frac{1}{x'_n x_{n-1}}F_d - \frac{y_{n-1}}{x'_n
x_{n-1}}F_b - \frac{y'_n}{x_n x'_n}F_c + \frac{y_n y'_n }{x_n
x'_n}F_a - \frac{z_{n-1}}{x'_n x_{n-1}}- \frac{z'_n}{x'_n} +
\frac{y'_n z_n}{x_n x'_n}\end{eqnarray*}

Therefore, \begin{enumerate}

\item[(i)] $x_{n-1}x'_n = x_n x'_{n+1}$,

\item[(ii)] $\displaystyle \frac{y'_{n+1}}{x'_{n+1}} =
\frac{y'_n}{x'_n}$,

\item[(iii)] $\displaystyle \frac{y_n}{x_n} =
\frac{y_{n-1}}{x_{n-1}}$, and

\item[(iv)] $\displaystyle \frac{z'_{n+1}}{x_n
x'_{n+1}}+\frac{z_n}{x_n} - \frac{y_n z'_n}{x_n x'_n} =
\frac{z_{n-1}}{x'_n x_{n-1}} +\frac{z'_n}{x'_n} -\frac{y'_n
z_n}{x_n x'_n}.$

\end{enumerate}

When checking conditions C4 and C5 we get exactly the same
conditions (i)--(iv).

\item  $|a| = |b|-1 = |c|-1 =n$.

\begin{enumerate}

\item[(I)] $|d| = n.$ \end{enumerate}

\noindent The relation C3 has the following form:  $$(F_a |_n F_b)
|_n (F_c |_{(n+1)'} F_d) = (F_a |_n F_c ) |_n (F_b |_{(n+1)'}
F_d).$$ We get after some calculations that it is equivalent to

\item[(v)] $\displaystyle \frac{y_n}{x_n} =
\frac{y'_{n+1}}{x'_{n+1}}.$

\noindent The relations C4 and C5 reduce to the same condition,
(v).
\begin{enumerate}
\item[(II)] $|d| = n+2$.
\end{enumerate}
Then the relations C3-C5 reduce to the condition (iii).

\item $|a| = |b| +1 = |c| +1 = n$
\begin{enumerate}
\item[(I)] $|d| = n-2$
\item[(II)] $|d| = n$.
\end{enumerate}
\end{enumerate}
We get, after some computations, that the relations 3(I) and 3(II)
follow from the conditions (iii) and (v).

Conditions (i) -- (v) are equivalent to the conditions on $x'_i,
y_i, y'_i,$ and $z'_i$ described in Example \ref{Example 4.5}.
Therefore the partial algebra $\mathcal A$ from Example
\ref{Example 4.5} satisfies the relations C1--C7.  Furthermore, if
$L$ is a diagram and $p$ --- its crossing, then the number of
components of $L^p_\circ$ is always equal to the number of
components of $L$ plus or minus one, so the sets $B_|$, $B_\ast
\subset A\times A$ are sufficient to define the link invariant
associated with $\mathcal A$.

Therefore $\mathcal A$ is a geometrically sufficient partial Conway
algebra.  It yields the invariant of links second coordinate of
which is a polynomial in an infinite number of variables.

\end{example}

\begin{prob}\label{Problem 4.6} \begin{enumerate}\item[(a)] Do there
exist two oriented links which have the same Jones-Conway polynomial
but which can be distinguished by the polynomial of infinitely many
variables?\footnote{Added for e-print: Adam Sikora proved in his
Warsaw master degree thesis written under direction of P.Traczyk, that
the answer to Problem 4.6 is negative.}

\item[(b)] Do there exist two oriented links which are
algebraically equivalent (i.e. the value of the invariant yielded
by any Conway algebra is the same for both links) but which can be
distinguished by the polynomial of infinitely many
variables?\end{enumerate}\end{prob}

We were unable to solve the above problem, partially due to the lack
of many candidates to be tested.  In particular, the examples of
Birman, which are algebraically equivalent but not skein equivalent,
are not helpful.

\begin{prop}\label{Proposition 4.7} Consider a geometrically
sufficient partial Conway algebra $\mathcal A$ such that $D=A\times
A\times A\times A$ and $B_|$ and $B_\ast$ includes all pairs $(u,v)$
such that the first letter of $u$ is $a_i$ and the first letter of
$v$ is $a_{i \mp 1}$ (in particular the partial algebra of Example
\ref{Example 4.5} satisfies these conditions) then Lemma \ref{Lemma
3.30} and Theorem \ref{Theorem 3.29} and \ref{Theorem 3.34} are
valid for $\mathcal A$.\end{prop}

Still the knots $9_{42}$, $10_{71}$ (in the Rolfsen notation),
$11_{394}$ and $11_{449}$ in the Thistlethwaite notation) and their
mirror images should be tested.

The next example of a geometrically sufficient partial Conway
algebra is related to the classical (Murasugi) signature of links.

It was (more or less) shown by Conway \cite{Co} (also Giller
\cite{Gi}) that the signature of knots is a skein equivalence
invariant.  We will show it in a more general context.  Our approach
is based on an observation that the Tristram-Levine signature is
related to Conway polynomial in just the same way as classical
signature to determinant invariant.  One can hope for an analogous
invariant (supersignature) related to the Jones-Conway polynomial.

\begin{defn}The following partial algebra $\mathcal A_{u,v}$ will be
called the supersignature algebra ($u$, $v$ real numbers, $u\cdot v
> 0$):  $$\begin{array}{rl} A =& (R\cup iR) \times (\ints \cup \infty)
\\ B_| = B_\ast = & \left\{ ((r_1, z_1), (r_2, z_2)) \in A \times A
~:~\textrm{if } 0 \neq r_1 \in R \textrm{, then } r_2 \in iR;\right.
\\ & \textrm{ and if } 0\neq r_1 \in iR \textrm{, then } r_2 \in R;
\textrm{if }z_1, z_2 \neq \infty\textrm{, then } |z_1 -z_2| = 1;\\ &
\left. r_i = 0 \textrm{ if and only if }z_i = \infty
\right\}.\end{array}$$

\noindent $|$ and $\ast$ are defined as follows:

\setcounter{equation}{8}

\noindent The first coordinates $r_1$, $r_2$, $r_0$ of elements
$w_1$, $w_2$, $w_0$ such that $w_1 = w_2 | w _0$, $w_2 = w_1 \ast
w_0$ are related as in the case of the Jones-Conway polynomial by
the equation
\begin{equation}\label{Equation 4.9} -ur_1 +v r_2 = i r_0.
\end{equation}

In particular, the first coordinate of the result depends only on
the first coordinates, so we write simply $r_1=r_2 | r_0$ and $r_2 =
r_1 \ast r_0$.  The second coordinate of the result is defined by
the equalities \begin{enumerate}

\item $i^z = \frac{r}{|r|}$ if $r\neq 0$,
\item $|z_i - z_\circ| = 1$ if $r_i \neq 0$, $r_0 \neq 0$, $i= 1,2$,
\item $z_1 = z_2$ if $r_0 = 0$,
\item $z= \infty$ if $r=0$.

\end{enumerate}

\noindent The 0-argument operations are defined as follows: $a_1 =
(1,0), \ldots, $ $$a_k = \left( \left(\frac{v-u}{i}\right)^{k-1},
\left\{
\begin{array}{ll} -(k-1) & \textrm{if } u<v \\  & \textrm{if } u=v
\\k-1 & \textrm{if } u>v \end{array}\right.\right), \ldots ~.$$

$D$ is defined to be the subset of $A\times A \times A \times A$
consisting of these elements for which the relations C3 -- C5 are
satisfied.
\end{defn}
We conjecture that $\mathcal A_{u,v}$ is a geometrically sufficient
partial Conway algebra.  If it is so it defines an invariant second
coordinate of which will be called the supersignature
$(\sigma_{u,v})$.  In fact the conjecture is true for $u=v \in
\left(\infty, -\frac{1}{2}\right] \cup \left[\frac{1}{2},
\infty\right)$ giving the Tristram-Levine signature.

\setcounter{thm}{9}

\begin{thm}\label{Theorem 4.10} For $u=v \in
\left(\infty, -\frac{1}{2}\right] \cup \left[\frac{1}{2},
\infty\right)$, $\mathcal A_{u,v}$ is a geometrically sufficient
partial Conway algebra.
\end{thm}
\setcounter{equation}{10}

\begin{proof}  Relations C1, C2, C6, and C7 follow immediately from
definition.  Concerning C3--C5, we will show that for links with
non-zero value $r_L(u,u)$ the second coefficient $z_L$ --  the
supersignature -- coincides with the Tristram-Levine signature (it
is the classical signature for $u=v= \frac{1}{2}$,
$r_L\left(\frac{1}{2}, \frac{1}{2}\right) \neq 0$).  It will follow
that the relations C3--C5 are satisfied in geometrically realizable
situations.

To prove this let us recall the definition of Tristram-Levine
signature \cite{Tr}, \cite{Le}, \cite{Go}.

Let $A$ be a Seifert matrix of a link $L$.  For each complex number
$\zeta$ ($\zeta \neq -1$) consider the Hermitian matrix $A(\zeta) =
(1-\overline{\zeta})A + (1-\zeta)A^T.$

The signature of this matrix, $\sigma_L(\zeta)$, is called the
Tristram-Levine signature of the link $L$.

Assume that $i(1-\overline{\zeta}) = -\frac{1}{i(1-\zeta)}$ (which
means that $1-\zeta$ lies on the unit circle).  Then $\det
iA(\zeta)$ is equal to the Conway potential $\Omega(-i(1-\zeta))$
(using Kauffman notation \cite{K_1}) and therefore we have an
equality
\begin{equation}\label{Equation 4.11} \det iA_{L_+}(\zeta) + \det
iA_{L_-}(\zeta) = (2-\zeta -\overline{\zeta})i\det
iA_{L_\circ}(\zeta)\end{equation} where $A_{L_+}$, $A_{L_-}$, and
$A_{L_\circ}$ are Seifert matrices of $L_+$, $L_-$, and $L_{\circ}$
respectively and $|1-\zeta| = 1$.

To complete the proof of Theorem \ref{Theorem 4.10} we need the
following lemma.

\setcounter{thm}{11}

\begin{lem}\label{Lemma 4.12}For $|i-\zeta| =1$ we have \begin{enumerate}

\item[(a)] $\displaystyle i^{\sigma(A(\zeta))} = \frac{\det
iA(\zeta)}{|\det iA(\zeta)|}$ if $\det A(\zeta) \neq 0$,\medskip

\item[(b)] $|\sigma_{L_+}(\zeta) - \sigma_{L_\circ}(\zeta)|=1$
(respectively, $|\sigma_{L_-}(\zeta) - \sigma_{L_\circ}(\zeta)|=1$)
if $\det A_{L_\circ}(\zeta) \neq 0$ and $\det A_{L_+}(\zeta)\neq 0$
(respectively $\det A_{L_-}(\zeta) \neq 0$),\medskip

\item[(c)] $\sigma_{L_+}(\zeta) = \sigma_{L_-}(\zeta)$ if $\det
A_{L_\circ}(\zeta)\neq 0$ and $\det A_{L_+}(\zeta), \det
A_{L_-}(\zeta) \neq 0$.\end{enumerate}

\end{lem}

\begin{proof}
To prove Lemma \ref{Lemma 4.12}(a) let us diagonalize $A(\zeta)$ to
get $A'(\zeta)$.  The matrix $iA'(\zeta)$ has $\pm 1$ on the
diagonal.  Now, $\sigma(A(\zeta)) = \sigma(A'(\zeta))$ is the number
of $i$'s in $iA'(\zeta)$ minus the number of $-i$'s, while
$\frac{\det iA(\zeta)}{|\det iA(\zeta)|}= \frac{\det
iA'(\zeta)}{|\det iA'(\zeta)|}$ is equal to the product of $i$'s and
$-i$'s, which implies (a).  To prove (b) and (c) let us recall the
Seifert matrices of $L_+$,  $L_-$, and $L_\circ$ may be chosen to
be $$A_{L_+} = \left[
                 \begin{array}{cc}
                   A_{L_\circ} & \alpha \\
                   \beta & \mu \\
                 \end{array}
               \right], ~ A_{L_-} = \left[
                 \begin{array}{cc}
                   A_{L_\circ} & \alpha \\
                   \beta & \mu+1 \\
                 \end{array}
               \right],$$ and $A_{L_\circ}$ respectively, where
               $\alpha$ is a column, and $\mathcal B$ is a row
               \cite{K_1}.  Then we get

$$A_{L_+}(\zeta) = \left[
                 \begin{array}{cc}
                   A_{L_\circ}(\zeta) & a \\
                   \overline{a}^T & m \\
                 \end{array}
               \right], ~ A_{L_-}(\zeta) = \left[
                 \begin{array}{cc}
                   A_{L_\circ}(\zeta) & a \\
                   \overline{a}^T & m+2-\zeta-\overline{\zeta} \\
                 \end{array}
               \right],$$ where $a = (1-\overline{\zeta})\alpha+
               (1-\zeta)\beta^T$ and $m =
               ((1-\overline{\zeta})+(1+\zeta))\mu$.  Now the properties
               (a) and (b) follow immediately:  just diagonalize
               $A_{L_\circ}(\zeta)$ first. Finally we have $z(u,u) =
               \sigma(\zeta)$ for $u=(2-\zeta -\overline{\zeta})^{-1}$
               and $r(u,u)=\det iA(\zeta)\neq 0$. \end{proof}

\end{proof}

We were unable to extend Theorem \ref{Theorem 4.10} for another $u$
and $v$ (however we are convinced that it is possible).  The main
obstacle is that the conditions C3--C5 are not always satisfied
(but probably they are satisfied where needed).  It is not difficult
to find an example when C3--C5 are not satisfied even for the
signature ($u=v=1/2$) (see Example 4.10 of \cite{P_1}).

On the other hand the existence of this example follows from
Proposition \ref{Proposition 4.7} because if the supersignature
algebra satisfies always the conditions C3--C5 then it satisfies
the assumptions of Proposition \ref{Proposition 4.7} but it
contradicts the fact that the signature distinguishes the Birman
links.

\begin{lem}\label{Lemma 4.13}The supersignature $\sigma_{u,v}$ (if exists for a given
$u,v$) satisfies the following conditions:  \begin{enumerate}
\item[(a)] $\sigma_{u,v}(L) = - \sigma_{v,u}(\overline{L})$

\item[(b)] $\sigma_{u,v}(L_1 \sharp ~ L_2) = \sigma_{u,v}(L_1) +
\sigma_{u,v}(L_2)$

\item[(c)] $\sigma_{u,v}(L_1 \sqcup L_2) = \sigma_{u,v}(L_1)
+\sigma_{u,v}(L_2) + \epsilon(u,v)$, where $$\epsilon(u,v) =
\left\{\begin{array}{ll} 1 & \textrm{if } u>v \\ \infty & \textrm{if
}u=v \\ -1 & \textrm{if } u<v \end{array}\right.$$

\item[(d)] $\sigma_{u,v}(L_+) \leq \sigma_{u,v}(L_-)$ if
$\sigma_{u,v}(L_+)\neq \infty$ and $u,v >0$.
\end{enumerate}\end{lem}

\begin{proof}

In (a), (b), and (d) we use the standard induction on the number
of crossings in a diagram of a link and on the number of bad
crossings (for some choice of base points).  In the proof of (a)
we use additionally Lemma \ref{Lemma 3.15} which gives us the
formula \setcounter{equation}{13}

\begin{equation}\label{Equation 4.14} r_L(u,v) = r_{\overline{L}}(-v, -u)
= \left\{\begin{array}{ll}r_{\overline{L}}(v,u) & \textrm{if
}L\textrm{ has an odd number of components} \\
-r_{\overline{L}}(v,u) & \textrm{if }L\textrm{ has an even number of
components.}\end{array}\right.\end{equation}

To prove (b) we use Corollary \ref{Cor 3.20} which gives us the
formula $$r_{L_1 \sharp ~ L_2}(u,v) = r_{L_1}(u,v)\cdot
r_{L_2}(u,v).$$

(c) follows from (b) if one observes that $L_1 \sqcup L_2$ can be
obtained as a connected sum of $L_1 \sharp~ T_2$ and $L_2$, where
$T_2$ is a trivial link of two components (Figure \ref{Fig 4.2}) and
that $\epsilon(u,v) = \sigma_{u,v}(T_2)$.

\end{proof}

\begin{figure}[htbp]
\centering
\includegraphics{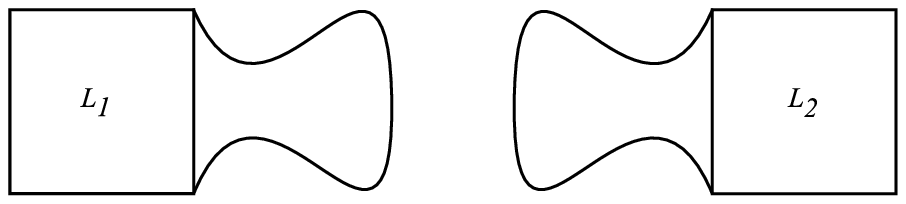}
\caption{$(L_1 \sharp~ T_2)\sharp~ L_2 = L_1 \sqcup L_2$}\label{Fig
4.2}
\end{figure}

The supersignature, if it exists, is a stronger link invariant than
the signature or the Tristram-Levine signature.

\setcounter{thm}{14}
\begin{example}\label{Example 4.15} The slice knots $8_8$ and $\overline{8}_8$ (Figure \ref{Fig
3.9}) can be distinguished using the supersignature $\sigma_{u,v}$
for some $u$ and $v$ however for $u=v$ the supersignature of both
knots is always equal to 0.

To see this, we find first that $$r_{8_8} = -uv^{-1}+2+v^{-2} +
u^{-1}v-2u^{-1}v^{-1}-u^{-2}v^2 -2u^{-2}+u^{-2}v^{-2}+u^{-3}v +
u^{-3}v^{-1}.$$

It follows from the above formula that $$r_{8_8}(u,u)>0\textrm{, so
}\sigma_{8_8}(u,u) \equiv 0 (\mod 4)$$
$$r_{8_8}(u,2u)<0 \textrm{ for } u>>0 \textrm{, so
}\sigma_{u,2u}(8_8) \equiv 2 (\mod 4) \textrm{, for } u>>0$$
$$r_{8_8}(2u,u) >0\textrm{ for } u>>0\textrm{, so
}\sigma_{2u,u}(8_8) \equiv 0 (\mod 4) \textrm{, for } u>>0.$$

Therefore, by Lemma \ref{Lemma 4.13}(a), $\sigma_{u,2u}(8_8) \neq
\sigma_{u,2u}(\overline{8}_8)$ for $u>>0$.  The equality
$\sigma_{u,u}(8_8) = 0$ follows for $u \in \left(-\infty,
-\frac{1}{2}\right] \cup \left[\frac{1}{2},\infty\right)$ from the
fact that $\sigma_{u,u}$ is the Tristram-Levine signature and $8_8$
is a slice knot.  Generally the equality $\sigma_{u,u}(8_8) = 0$
follows by considering properly chosen $L_\circ$ (for $L=8_8$);
compare Example \ref{Example 3.11}.

We refer to \cite{P_1} for detailed analysis of $\sigma_{u,v}(8_8)$.

\end{example}

\begin{cor}\label{Corollary 4.16} Assume that for a given pair $u,v$ the supersignature
$\sigma_{u,v}$ exists and that for a given $L$, $\sigma_{u,v}(L)
\neq \infty$.  Then the minimal height of a resolving tree of $L$ is
not less than $\displaystyle \frac{|\sigma_{u,v}(L)|}{2} -
\epsilon{L}$ where $$\epsilon(L) = \left\{\begin{array}{ll} 0 &
\textrm{if } u=v \\ n(L) -1 & \textrm{if } u \neq
v\end{array}\right.$$ where $n(L)$ denotes the number of components
of $L$.\end{cor}

\begin{proof} It follows from the definition of the supersignature
(by using induction on the minimal height of a resolving tree of a
link). \end{proof}

\begin{rem}\label{Remark 4.17} The signature (and the Tristram-Levine signature) is a
good tool for studying the unknotting number $(u(L))$ of knots and
links.  Namely $u(L) \geq \frac{|\sigma_3(L)}{2}$.  We hope for
similar formula for any supersignature because $|\sigma_{u,v}(L_+) 0
\sigma_{u,v}(L_-)| \leq 2$.  Unfortunately the last inequality holds
only if $r_{L_+}(u,v)$, $r_{L_-}(u,v) \neq 0$ (i.e.
$\sigma_{u,v}(L_+), \sigma_{u,v}(L_-) \neq \infty$).  So the problem
is to put for $\sigma_u,v(L)$ different value than $\infty$ (in the
case $r_L(u,v) =0$).  One possible solution uses the fact that the
Jones-Conway polynomial (and also $r_L(u,v)$) is not identically
equal to 0 (Lemma \ref{Lemma 3.36}(b)).  Namely if $r_L(u_0, v_0) =
0$ then for each neighborhood of $u_0, v_0$, $r_(u,v)$ is different
than 0 almost everywhere.  Now we put for $\sigma_L(u_0, v_0)$ the
``average'' value of $\sigma_L(u,v)$ from the neighborhood.  Of
course this idea is far from being complete.  In particular the
meaning of the ``average'' should be made more precise.\end{rem}

\begin{conj}\label{Conjecture 4.18} Assume that for a given $u,v,L :
\sigma_{u,v}(L)\neq \infty$.  Then $$u(L) \geq
\left\{\begin{array}{ll}\frac{|\sigma_{u,v}(L) -
\sigma_{u,v}(T_n)|}{2} & \textrm{if } u\neq v \\ &  \\
\frac{|\sigma_{u,v}(L)|}{2} & \textrm{if } u=v
\end{array}\right.$$  $T_n$ denotes the trivial link of $n$
components.\end{conj}

Consider the equation which defines the Jones polynomial
$$-tV_{L_+}(t) +\frac{1}{t}V_{L_-}(t) = \left(\sqrt(t) -
\frac{1}{\sqrt{t}}\right)V_{L_\circ}(t).$$  If we substitute
$\sqrt{t} = -iw$ then we get the equation
\setcounter{equation}{18}
\begin{equation}\label{Equation 4.19} -w^2V_{L_+}(w) + \frac{1}{w^2}
V_{L_-}(w) = i\left(w + \frac{1}{w}\right)
V_{L_\circ}(w).\end{equation}  This equation can be used, for $w\in
\reals - \{0\}$, to define the supersignature associated with the
Jones polynomial: $$\sigma_w(L) = \sigma_{u,v}(L)\textrm{, where }
u=\frac{w^2}{w+\frac{1}{w}} ~
v=\frac{1}{w^2}\frac{1}{w+\frac{1}{w}}.$$  In particular we get the
classical signature for $w=1$ ($V_L(w) \neq 0$).\footnote{T.
Przytycka has recently shown that the supersignature does not always
exist.  In particular it does not exist if $u=2v>2$.  However the
existence of the supersignature associated with the Jones polynomial
is the open problem.}

\setcounter{thm}{19}

Now we can generalize into $\sigma_w(L)$ the classical result of
Murasugi \cite{Mu_4}.

\begin{thm}\label{Theorem 4.20} Assume that for a given $w$,
$\sigma_w$ exists.  Then $\sigma_w(L) + lk(L)$ does not depend on
the orientation of $L$.  \end{thm}

\begin{proof}To prove the theorem, the following lemma is needed.

\begin{lem}[Reversing result]\label{Lemma 4.21}Suppose that a component $L_i$ of an
oriented link $L$ has a linking number $\lambda$ with the union of
the other components.  Let $L'$ be $L$ with the direction of $L_i$
reversed.  Then $\sigma_w(L') = \sigma_w(L) + 2\lambda$. \end{lem}

The theorem follows from the lemma because the equality
$$\sigma_w(L) + lk(L) = \sigma_w(L')  + lk(L') $$ is equivalent to
$$\sigma_w(L') = \sigma_w(L) + lk(L) - lk(L') \textrm{ and }lk(L) - lk(L')
= 2\lambda.$$

To prove Lemma \ref{Lemma 4.21}, we need the Jones reversing result
($V_{L'}(w) = (-1)^\lambda w V_L(w)$) or rather the
Lickorish-Millett method of its proof (see Lemma \ref{Lemma 5.15}).

We refer to \cite{P_1} for details.
 \end{proof}

 One can try to construct a more general supersignature modeled on
 the polynomial in an infinitely many variables (with $z_i = z'_i =
 0$).  We didn't try to pursue this concept any further.

 The next example describes a universal geometrically sufficient
 partial Conway algebra.

 \begin{example}\label{Example 4.22} Skein equivalence classes of
 oriented links form a geometrically sufficient partial Conway
 algebra $\mathcal{A}_u$.  $a_n$ is a skein equivalence class of a
 trivial link of $n$ components.  The operation $|$ (respectively,
 $\ast$) is defined on a pair of classes of links if they have
 representatives of the form $L_-$ and $L_\circ$ (respectively $L_+$
 and $L_\circ$).  The result is the class of $L_+$ (respectively
 $L_-$).  The definition of the skein equivalence is chosen in such
 a way that the conditions C1--C7 are satisfied when needed.  Notice
 that $\mathcal{A}_u$ is the universal geometrically sufficient
 partial Conway algebra, that is, for any geometrically sufficient
 partial Conway algebra $\mathcal A$, there is a unique homomorphism
 $\mathcal{A}_u \to \mathcal{A}$.

 \end{example}
 \newpage

\newpage \thispagestyle{empty}


AMS 57M25  \hrulefill[0pt]{} ISSN 0239-8885

{\center{ \vspace{1cm}

UNIWERSYTET WARSZAWSKI

 INSTYTUT MATEMATYKI

\vspace{2in}

J\'{o}zef H. Przytycki

Survey on recent invariants on classical knot theory

III. Kauffman approach

\vspace{4in}

%


Preprint 9/86

``na prawach r{\c e}kopisu''

\vspace{1cm}

Warszawa 1986

}}
\newpage

\section{Kauffman approach}

It is the natural question whether the three diagrams $L_+$, $L_-$,
and $L_\circ$ which have been used to build Conway type invariants
could be replaced by another diagram.  On the edge of December 1984
and January 1985, K.~Nowinski has suggested consideration of the
fourth diagram obtained by smoothing $L_+$ in the way which does not
agree with the orientation of $L_+$ (Figure \ref{Fig 5.1}) but we
didn't make any effort to get an invariant of links.

\begin{figure}[htbp]
\centering
\includegraphics{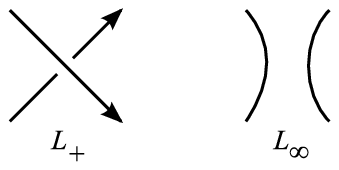}
\caption{}\label{Fig 5.1}
\end{figure}

At the early spring of 1985, R.~Brandt, W.B.R.~Lickorish and
K.C.~Millett \cite{B_L_M} and independently C.F.~Ho \cite{Ho} have
proven that if one considers non-oriented links then four
non-oriented diagrams from Figure \ref{Fig 5.2} (of course we do
not distinguish $L_+$ from $L_-$ and $L_\circ$ from $L_\infty$)
can lead to the construction of a new invariant of isotopy of
non-orientable links.

\begin{figure}[htbp]
\centering
\includegraphics{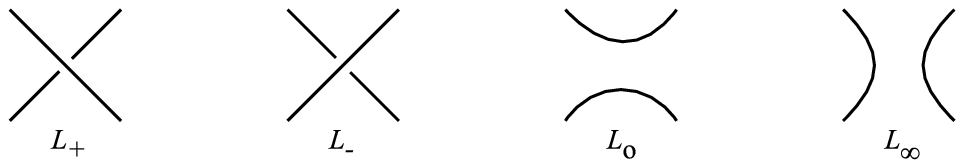}
\caption{}\label{Fig 5.2}
\end{figure}

\begin{thm}\label{Theorem 5.1} There exists a uniquely determined invariant $Q$ which
attaches an element of $\ints[x^{\mp 1}]$ to every isotopy class of
unoriented links and satisfies the following conditions:

\begin{enumerate}

\item $Q_{T_1}(x)=1$ where $T_1$ is a trivial knot,

\item $Q_{L_+}(x) + Q_{L_-}(x) = x (Q_{L_\circ}(x) +
Q_{L_\infty}(x))$ where $L_+$, $L_-$, $L_\circ$, and $L_\infty$ are
diagrams of links which are identical, except near one crossing
point, where they look like on Figure \ref{Fig 5.2}.

\end{enumerate}\end{thm}

The proof of Theorem \ref{Theorem 5.1}\footnote{In the original,
the Theorem was referenced as Theorem 5.2} is very similar to that
of Theorem \ref{Theorem 2.1.1} (compare \cite{B_L_M}).  We will
show it later in more general context.

The polynomial $Q_L(x)$ shares many properties analogous to that of
Jones-Conway ($P_L(x,y)$).  In particular we have:

\begin{prop}\label{Proposition 5.2}\begin{enumerate}
\item[(a)] $Q_{L_1\sharp ~L_2}(x) = Q_{L_1}(x) \cdot Q_{L_2}(x)$

\item[(b)] $Q_{L_1 \sqcup L_2}(x) = \mu Q_{L_1}(x) \cdot
Q_{L_2}(x)$ where $\mu = 2x^{-1}-1$ is the value of the invariant
for a trivial link of 2 components.

\item[(c)] $Q_L(x) = Q_{\overline{L}}(x)$ where $\overline{L}$ is the mirror
image of $L$

\item[(d)] $Q_{L}(x) = Q_{m(L)}(x)$ where $m(L)$ is a mutant of
$L$. \end{enumerate}\end{prop}

Proof is easy and we omit it.

The polynomial $Q_L(x)$ can sometimes distinguish knots which are
skein equivalent.

\begin{prop}\label{Proposition 5.3}The knots $8_8$,
$\overline{10}_{129}$, and $13_{6714}$ have different $Q_L(x)$
polynomials however they are skein equivalent.\end{prop}

\begin{proof} Just perform the calculations and use Example \ref{Example
3.11}.\end{proof}

Now we will describe the Kauffman approach which allows in
particular to generalize the $Q_L{(x)}$ polynomial to polynomial
which often distinguishes a link from its mirror image.  The
Kauffman approach to invariants of links bases on the idea of
considering diagrams up to the relation which does not use the first
Reidemeister move.  In this way we will not get an invariant of a
link but often after some correction an invariant of links can be
achieved.

\begin{defn}\label{Definition 5.4} Two diagrams are regularly isotopic iff one can be
obtained from the other by a sequence of Reidemeister moves of
type $\Omega^{\mp 1}_2$ or $\Omega^{\mp 1}_3$.  This definition
makes sense for orientable and non-orientable diagrams as
well.\end{defn}

Working with regular isotopy we are able to take into account some
properties of a diagram which are eliminated by the Reidemeister of
the first type.

\begin{lem}\label{Lemma 5.5} Let $\tw(L) = \sum_i p_i$ where the sum is taken over
all the crossings of $L$ (we call it twist or writhe number). Then
$\tw(L)$ is an invariant of the regular isotopy of diagrams, and
$\tw(-L)=\tw(L)$.  \end{lem}

\begin{proof}The Reidemeister move of the second type creates or
kills two crossings of the opposite signs, and the move of the third
type does not change the signs of crossings.  Furthermore, the
change of $L$ to $-L$ does not change the signs of
crossings.\end{proof}

Now the idea of Kauffman uses the fact that the trivial knot (up to
isotopy) has many representants in the regular isotopy category.
Therefore each of these representants can have different value of an
invariant. Kauffman associates with a diagram $T_1$ representing a
trivial knot the monomial $a^{\tw(T_1)}$.  Then the Kauffman
definition of invariants reminds that of Conway, Jones, $P_L(x,y)$,
or $Q(x)$.  When one wants to go from invariants of regular isotopy
to invariants of isotopy, the following lemma is useful.

\begin{lem}\label{Lemma 5.6}Consider the following elementary move
on a diagram of a link (denoted $\Omega_{0.5}^{\mp 1}$ and called
the weakened first Reidemeister move); i.e. the move which allows us
to create or to kill the pair of curls of the opposite signs (Figure
\ref{Fig 5.3}).  Observe that the sign of the crossing in the curl
does not depend on the orientation of a diagram.

\begin{figure}[htbp]
\centering
\includegraphics{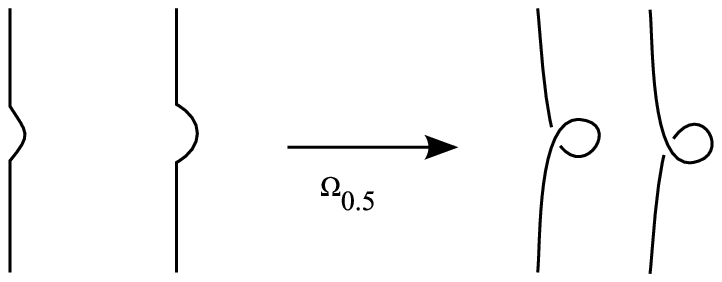}
\caption{}\label{Fig 5.3}
\end{figure}

Then we can obtain the diagram $L_1$ from $L_2$ by regular isotopy
and $\Omega^{\mp 1}_{0.5}$ moves iff $\tw(L_1) = \tw(L_2)$ and $L_1$
is isotopic to $L_2$.\end{lem}

\begin{proof}It becomes clear if one observes that the move
$\Omega^{\mp 1}_{0.5}$ allows us to carry a curl to any place in the
diagram.\end{proof}

Now we will show how, using the Kauffman approach, the Conway
polynomial can be generalized into Jones-Conway polynomial and
$Q_L(x)$ into polynomial which we will call the Kauffman polynomial.

\begin{thm}[\cite{K_4}]\label{Theorem 5.7}\begin{enumerate}\item[(a)]
There exists a uniquely determined invariant of regular isotopy of
oriented diagrams ($R_L(a,z) \in \ints[a^{\mp 1}, z^{\mp 1}]$) which
satisfies the following conditions:
\begin{enumerate}\item[(1)] $R_{T_1}(a,z) = a^{\tw(T_1)}$, where
$T_1$ is isotopic to the trivial knot, \item[(2)] $R_{L_+}(a,z) -
R_{L_-}(a,z) = z R_{L_\circ}(a,z)$ \end{enumerate}

\item[(b)] Let us define for a given diagram $L$: $G_L(a,z) =
a^{-\tw(L)}R_L(a,z)$ then $G_L(a,z)$ is an invariant of isotopy of
oriented links, equivalent to the Jones-Conway polynomial; i.e.
$G_L(a,z) = P_L(x,y)$ for $x=\frac{a}{z}$ and $y= -\frac{1}{az}$.
\end{enumerate}\end{thm}

\begin{proof}The method of proof of Theorem \ref{Theorem 2.1.1}
can be used, but because we have already proven the existence of
the Jones-Conway polynomial so we use it instead.  If we
substitute in $P_L(x,y)$, $x=\frac{a}{z}$, $y = -\frac{1}{az}$ we
will get the polynomial invariant of isotopy classes of oriented
links, $\widetilde{G}_L(a,z)$ which satisfies:

\begin{enumerate}
\item $\widetilde{G}_{T_1}(a,z)= 1$

\item $a \widetilde{G}_{L_+}(a,z)-\frac{1}{a}\widetilde{G}_{L_-}(a,z)=
a\widetilde{G}_{L_\circ}(a,z)$.
\end{enumerate}

Now let us define
$\widetilde{R}_L(a,z)=a^{\tw(L)}\widetilde{G}_L(a,z)$ for an
oriented diagram $L$.

It is easy to see that the first Reidemeister move changes the
value of $\widetilde{R}_L(a,z)$ depending on the sign of the curl
as follows:
$$\begin{array}{l} \widetilde{R}_{[\includegraphics{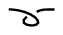}]}(a,z) = a
\widetilde{R}_{[\includegraphics{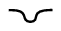}]}(a,z) \\
\widetilde{R}_{[\includegraphics{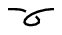}]}(a,z) = a^{-1}
\widetilde{R}_{[\includegraphics{Bump.eps}]}(a,z).\end{array}$$

The second and the third type of Reidemeister moves do not change
$\widetilde{R}_L(a,z)$. Therefore $\widetilde{R}_L(a,z)$ is an
invariant of regular isotopy and it satisfies
$\widetilde{R}_{T_1}(a,z) =a^{\tw(T_1)}$.  Now we should only check
that $\widetilde{R}_L(a,z)$ satisfies the equality (2) from Theorem
\ref{Theorem 5.7}.

But from the equality (2) of the proof we get
$$a\widetilde{R}_{L_+}(a,z)a^{-\tw(L_+)} -
\frac{1}{a}\widetilde{R}_{L_-}(a,z) a^{-\tw(L_-)} =
z\widetilde{R}_{L_\circ}(a,z)a^{-\tw(L_\circ)}$$ and it reduces to
$$\widetilde{R}_{L_+}(a,z) - \widetilde{R}_{L_-}(a,z) = z
\widetilde{R}_{L_\circ}(a,z).$$

Because every diagram possesses a resolving tree so polynomial
$R_L(a,z)$ if exists, is unique, therefore we can put $R_L(a,z) =
\widetilde{R}_L(a,z)$ and $G_L(a,z) = \widetilde{G}_L(a,z)$ which
completes the proof of Theorem \ref{Theorem 5.7}. \end{proof}

\begin{thm}\cite{K_5}\begin{enumerate}\item[(a)] There exists a
uniquely determined invariant $L$, which attaches an element of
$\ints[a^{\mp 1}, z^{\mp 1}]$ to every regular isotopy class of
oriented diagrams and satisfies the following conditions:

\begin{enumerate}

\item[(1)] $L_{T_1}(a,z) = a^{\tw(T_1)}$

\item[(2)] $L_{L_+}(a,z) +
L_{L_-}(a,z)=z(L_{L_\circ}(a,z)+L_{L_\infty}(a,z)).$
\end{enumerate}

\item[(b)] Let us define for a given diagram $D$, $F_D(a,z) =
a{-\tw(D)}L_D(a,z)$.  Then $F_D(a,z)$ is an invariant of isotopy
classes of oriented links and it generalizes the polynomial $Q$
($Q_L(x) = F_L(1,x)$).\end{enumerate}\end{thm}

\begin{proof} Part (a) will be proved later in more general
context. Part (b) follows from (a) if we notice that the first
Reidemeister move changes $F_L(a,z)$ as follows:
$$\begin{array}{l}L_{[\includegraphics{Curl1.eps}]}(a,z) = a L_{[\includegraphics{Bump.eps}]}(a,z)
\textrm{and} \\ \, \\
L_{[\includegraphics{Curl2.eps}]}(a,z)=a^{-1}L_{[\includegraphics{Bump.eps}]}(a,z).\end{array}$$\end{proof}

The polynomial $F_L(a,z)$ is called the Kauffman polynomial. Now we
will state some elementary properties of the Kauffman polynomial.

\begin{lem}\label{Lemma 5.9}\begin{enumerate}\item[(a)] $F_{L_1
\cs L_2}(a,z) = F_{L_1}(a,z)\cdot F_{L_2}(a,z)$

\item[(b)] $F_{L_1\sqcup L_2}(a,z) = \mu F_{L_1}(a,z) \cdot
F_{L_2}(a,z)$ where $\mu = \frac{a+ a^{-1}}{z} -1$ is the value of
the invariant of a trivial link of 2 components,

\item[(c)] $F_L(a,z) = F_{-L}(a,z)$,

\item[(d)] $F_{\overline{L}}(a,z) = F_L(a^{-1},z)$,

\item[(e)] $F_L(a,z)=F_{m(L)}(a,z)$ where $m(L)$ is a mutant of
$L$.\end{enumerate}\end{lem}

The proof is very similar to that for the Jones-Conway polynomial.

The polynomial $L(a,z)$ does not depend on an orientation of
components of a diagram $D$.  Therefore the Kauffman polynomial
$F(a,z)$ depends on an orientation of components of $D$ in a simple
manner (because $F_D(a,z)$ differs from $L_D(a,z)$ only by a power
of $a$).

\begin{lem}\label{Lemma 5.10} Let $D=\{D_1, D_2, \ldots, D_i,
\ldots, D_n\}$ be a diagram of an oriented link of $n$ components
and let $D'=\{D_1, D_2, \ldots, -D_i, \dots, D_n\}$.  Let $\lambda =
lk(D_i, D-D_i) = \frac{1}{2} \sum \sgn p_j$ where the sum is taken
over all crossings between $D_i$ and $D-D_i$.  Then $F_{D'}(a,z) =
a^{4\lambda}F_D(a,z)$.\end{lem}

\begin{proof}$L_{D'}(a,z) = L_D(a,z)$ so $a^{\tw(D')}F_{D'}(a,z) =
a^{\tw(D)}F_D(a,z)$ therefore, \newline $F_{D'}(a,z) = a^{\tw(D) -
\tw(D')}F_D(a,z) = a^{4\lambda}F_D(a,z)$.\end{proof}

We can comment on the lemma as follows: the Kauffman polynomial says
about different orientations of $D$ as much as linking numbers of
its components.  The Kauffman polynomial is much more useful for
testing amphicheirality of links.  We have however an example of a
link of two components which is a mutant of its mirror image but is
not isotopic (Figure \ref{Fig 3.7}) but Kauffman has conjectured
that for knots such a case is impossible.

\begin{conj}[\cite{K_5}]\label{Conjecture 5.11} If the knot $K$ is
not isotopic to its mirror image ($\overline{K}$ or $-\overline{K}$)
then $F_K(a,z) \neq F_{\overline{K}}(a,z)$.\end{conj}

The knots $9_{42}$ and $10_{71}$ (in the Rolfsen \cite{Ro} notation,
see Figure \ref{Fig 5.4}) contradict the conjecture, however weak
version of the conjecture can still be true (see Problem
\ref{Problem 5.29}(d) and Conjecture \ref{Conjecture 5.30}).

\begin{prob}\label{Problem 5.12} \begin{enumerate}\item[(a)] Is it
possible to distinguish $9_{42}$ from $\overline{9}_{42}$ using any
invariant yield by a Conway algebra ($9_{42}$ and
$\overline{9}_{42}$ are not skein equivalent because they have
different signatures)? (Figure \ref{Fig 5.4}.)

\item[(b)] Is the knot $10_{71}$ skein equivalent to its mirror
image? (Figure \ref{Fig 5.4}.)\end{enumerate}\end{prob}

\begin{figure}[htbp]
\centering
\includegraphics{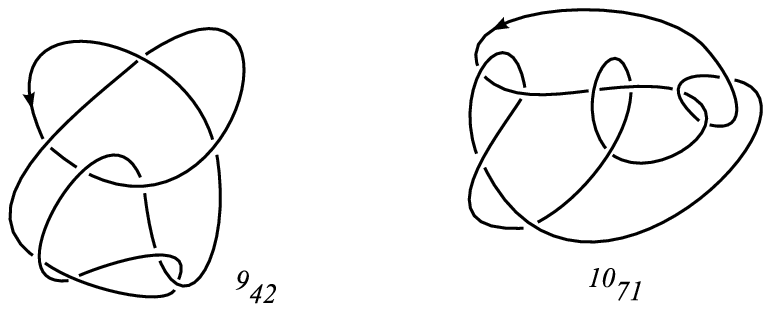}
\caption{}\label{Fig 5.4}
\end{figure}

The Kauffman polynomial is also the generalization of the Jones
polynomial.

\begin{thm}[\cite{Li_2}]\label{Theorem 5.13} $V_L(t)= F_L\left(t^{3/4},
-\left(t^{-1/4}+t^{1/4}\right)\right)$.\end{thm}
\begin{proof} First we will describe the Kauffman polynomial not
using regular isotopy.

\begin{lem}\label{Lemma 5.14} The Kauffman polynomial is uniquely
determined by the following conditions:
\begin{enumerate}

\item $F_{T_1}(a,z) = 1$

\item \begin{enumerate}

\item[(i)] for $c(L_+)<c_(L_\circ)$ , where $c(L)$ denotes the
number of components:  $$ aF_{L_+}(a,z) + \frac{1}{a}F_{L_-}(a,z)
=z(F_{L_\circ}(a,z) + a^{-4\lambda} F_{L_\infty}(a,z))$$ where we
give $L_\infty$ one of the two possible orientations and $\lambda =
\lk(L_i, L_\circ, -L_i)$ where $L_i$ is the component of $L_\circ$
which orientation does not agree with the orientation of the
corresponding component of $L_\infty$.

\item[(ii)] for $c(L_+)>c(L_\circ)$ $$ aF_{L_+}(a,z) +
\frac{1}{a}F_{L_-}(a,z) =z(F_{L_\circ}(a,z) + a^{-4\lambda+2}
F_{L_\infty}(a,z))$$ where we give $L_\infty$ one of the two
possible orientations and $\lambda = \lk(L_i, L_+, -L_i)$ where
$L_i$ is the component of $L_+$ which orientation does not agree
with orientation of the corresponding component of $L_\infty$.

\end{enumerate}\end{enumerate}\end{lem}
\begin{proof} It follows from the definitions of $F_L(a,z)$ and
from Lemma \ref{Lemma 5.10}.  We will show as an example how to get
the formula (2)(ii).

By definition we have $$L_{L_+}(a,z) + L_{L_-}(a,z) =
z(L_{L_\circ}(a,z) + L_{L_\infty}(a,z)$$ therefore:
$$a^{\tw(L_+)}F_{L_+}(a,z)+a^{\tw(L_-)}F_{L_-}(a,z) = z
(a^{\tw(L_\circ)}F_{L_\circ}(a,z)+a^{\tw(L_\infty)}F_{L_\infty}(a,z))$$
so $$a F_{L_+}(a,z)+\frac{1}{a}F_{L_-}(a,z) =
z(F_{L_\circ}(a,z)+a^{\tw(L_\infty) -
\tw(L_\circ)}F_{L_\infty}(a,z))$$ what reduces to the formula
(2)(ii).\end{proof}

In the next part of the proof of Theorem \ref{Theorem 5.13} we need
an additional characterization of the Jones polynomial. Remind that
the Jones polynomial was uniquely defined by the following
conditions:

\begin{enumerate}
\item $V_{T_1}(t)=1$

\item $\displaystyle -tV_{L_+}(t)+\frac{1}{t} V_{L_-}(t) =
\left(\sqrt{t} -\frac{1}{\sqrt{t}}\right) V_{L_\circ}(t)$.

\end{enumerate}

\begin{lem}[Jones]\label{Lemma 5.15} Suppose that a component
$L_i$ of an oriented link $L$ has linking number $\lambda$ with the
union of the other components.

Let $L'$ be $L$ with the direction of $L_i$ reversed.  Then
$V_{L'}(t)=t^{3\lambda}V_L(t)$.\end{lem}

\begin{proof} We present the proof of Lickorish and Millett
\cite{Li_M_3}, another elementary proofs have been found by Morton
and Kauffman.

The proof is in five sections.

\begin{enumerate}

\item Lemma \ref{Lemma 5.15} is true for the two links of Figure
\ref{Fig 5.5}.  This is an easy computation.

\begin{figure}[htbp]
\centering
\includegraphics[width=.6\textwidth]{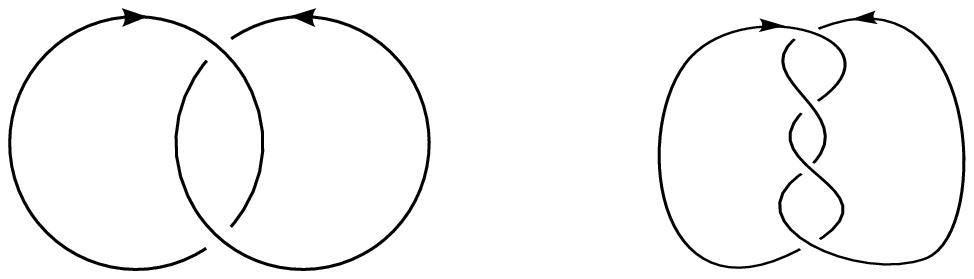}
\caption{}\label{Fig 5.5}
\end{figure}

\item If the orientation of every component of $L$ is reversed
then $V_L(t)$ is unchanged.  Further, $V_{K \cs L}(t) =
V_K(t)V_L(t)$ where $K\cs L$ is any connected sum of oriented links
$K$ and $L$, and also $V_{\overline{L}}(t)=V_L(1/t)$, where
$\overline{L}$ is the mirror image of $L$.  Thus if Lemma \ref{Lemma
5.15} is true for $K$ and $L$ it is true for $\overline{K}$ and for
$K \cs L$.

\item Consider the self-crossings of the component $L_i$ in some
presentation of $L$.  Induction (as repeatedly used in section
three of \cite{Li_M_1} or in \ref{Theorem 2.1.1}) on the number of
these crossings and on the number of them that have to be switched
to unknot $L_i$ shows that $L_i$ may be assumed to be unknotted.

\item Let the unknotted component $L_i$ bound a disc that meets
the remained of $L$ in $n$ points. Proceed by induction on $n$. The
start of the induction will be given in (5); for the moment assume
that $n\geq 4$.  Figure \ref{Fig 5.6} depicts a skein triple in
which $L_i$ is $L_\circ$.  The disc bounded by $L_i$ is shown
meeting the remained of $L$ in $n$ points shown as crosses. In
$L_-$, $L_i$ has become two unlinked curves $\gamma^-_1$ and
$\gamma^-_2$ that bound discs that meet the remainder of $L_-$ with
linking numbers $\lambda_1$ and $\lambda_2$ respectively. The
situation of $L_+$ is exactly similar except that $\gamma^+_1$ and
$\gamma^+_2$ are linked as shown.

\begin{figure}[htbp]
\centering
\includegraphics{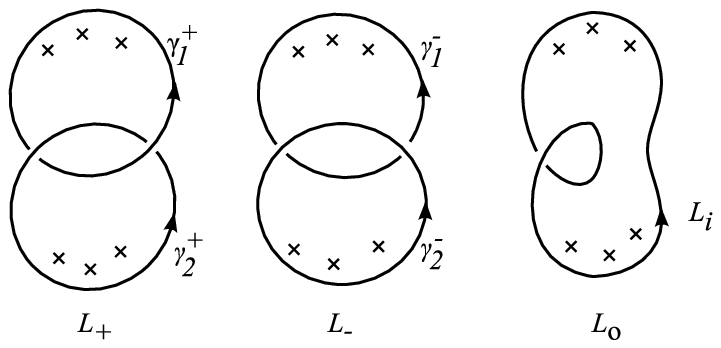}
\caption{}\label{Fig 5.6}
\end{figure}

Thus $n_1+n_2 = n$ and $\lambda_1 + \lambda_2 = \lambda$.  Choose
$n_1$ and $n_2$ so that each is at most $n-2$, (recall $n\geq 4$).
Let $L'_+$, $L'_-$, and $L'_\circ$ be the same links but with the
$\gamma^{\mp}_1$, $\gamma^{\mp}_2$ and $L_i$ all reversed.  Then
$$\begin{array}{l}tV_{L_+}(t) - t^{-1}V_{L_-}(t) + (t^{1/2}-t^{1/2})V_{L_\circ}(t)
= 0 \\ tV_{L'_+}(t) - t^{-1}V_{L'_-}(t) +
(t^{1/2}-t^{1/2})V_{L'_\circ}(t) = 0.\end{array}$$ But, by the
induction on $n$, reversing $\gamma^-_1$ and then $\gamma^-_2$ gives
$$t^{3\lambda_2}t^{3\lambda_1}V_{L_-}(t) = V_{L'_-}(t)$$ and
reversing $\gamma^+_1$ and then $\gamma^+_2$ gives
$$t^{3(\lambda_2-1)}t^{3(\lambda_1+1)}V_{L_+}(t)=V_{L_+}(t).$$  It
follows immediately that $t^3V_L(t) = V_{L'}(t)$.

This argument extends a little further when $n=3$.  If $\lambda$ is
also 3, choose $n_1 =1$ and $n_2=2$, then the above argument holds
if the theorem is known for $n=3$ and $\lambda =1$ and for $n\leq
2$.  Similarly when $n=3$, $\lambda =-3$.

\item Suppose that $n=3$ and $\lambda = \pm 1$.  It is required to
show that whatever tangle is inserted into the room (the rectangle)
of Figure 5.\ref{Fig 5.7} to give $L$, Lemma \ref{Lemma 5.15} holds
true and $t^3 V_L(t) = V_{L'}(t)$.

However, the standard induction on the number of crossings in the
room and on the number of bad crossings in the room for some choice
of base points allows us to consider the room filled on Figure
5.\ref{Fig 5.8}.

\begin{figure}[htbp]
\centering

\subfigure[]{\includegraphics[scale = 0.9]{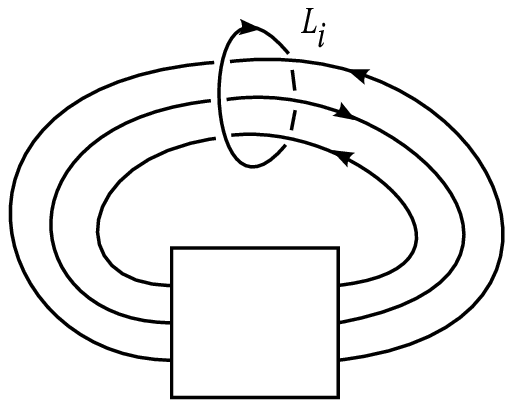}\label{Fig
5.7}} \hspace{2cm} \subfigure[]{\includegraphics[scale =
0.9]{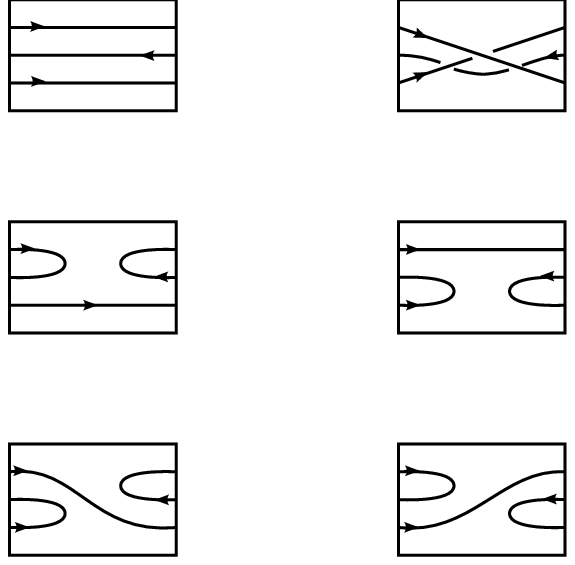}\label{Fig 5.8}}

\caption{}

\end{figure}
\setcounter{figure}{8}

\noindent Thus all that is required is to check that whichever of
these pictures is inserted into the room to give $L$ the theorem
holds. This follows at once from (1) and (2).
\end{enumerate}

A simplified version of this proof works when $n=2$.

The case $n=1$ is immediate from (1) and (2) and $n=0$ is trivial.

This completes the proof of Lemma \ref{Lemma 5.15}.
\end{proof}

We have remarked in Theorem \ref{Theorem 4.20} that Lemma \ref{Lemma
5.15} can be used to extend the result of Murasugi \cite{Mu_4} of
signature into supersignature related to the Jones polynomial.  The
next step to prove Theorem \ref{Theorem 5.13} is so called
$V_\infty$-formula \cite{Bi_3}, \cite{Bi_4}.

\begin{lem}[Birman]\label{Lemma 5.16}\begin{enumerate}

\item[(i)] $c(L_+)<c(L_\circ)$ where $c(L)$ is the number of
components of $L$.  Let us give $L_\infty$ me orientation (which
agrees with that of $L_+$ if possible), and let $\lambda = \lk(L_i,
L_\circ-L_i)$ where $L_i$ is the component of $L_\circ$ which
orientation does not agree with the orientation of the corresponding
component of $L_\infty$.  Then $$\sqrt{t} V_{L_+}(t) -
\frac{1}{\sqrt{t}}V_{L_-}(t) = \left(\sqrt{t} -
\frac{1}{\sqrt{t}}\right) t^{-3\lambda}V_{L_\infty}(t)$$

\item[(ii)] $c(L_+)<c(L_\circ)$

Let us give $L_\infty$ one of the two possible orientations and let
$\lambda = \lk(L_i, L_+-L_i)$, where $L_i$ is the component of $L_+$
which orientation does not agree with the orientation of the
corresponding component of $L_\infty$.  Then $$\sqrt{t} V_{L_+}(t) -
\frac{1}{\sqrt{t}}V_{L_-}(t) = \left(\sqrt{t} -
\frac{1}{\sqrt{t}}\right) t^{-3(\lambda - 1/2)}V_{L_\infty}(t).$$

\end{enumerate}
\end{lem}
\begin{proof}
We will follow Lickorish and Millett \cite{Li_M_3}.

\begin{enumerate}

\item[(i)] $c(L_+) < c(L_\circ)$

Consider the diagram $X$ with 2 crossings $p$ and $q$ as on Figure
\ref{Fig 5.9} such that $$L_\circ = X^{pq}_{-+},~L_+ =
X^{pq}_{+\circ},~ L_- = X^{pq}_{-\circ}.$$

If we consider the crossing $q$ we get:
\begin{enumerate}

\item $-tV_{L_\circ}(t) + \frac{1}{t}V_X(t) =
\left(t-\frac{1}{\sqrt{t}}\right)V_{L_-}(t)$.

\begin{figure}
\centering
\includegraphics{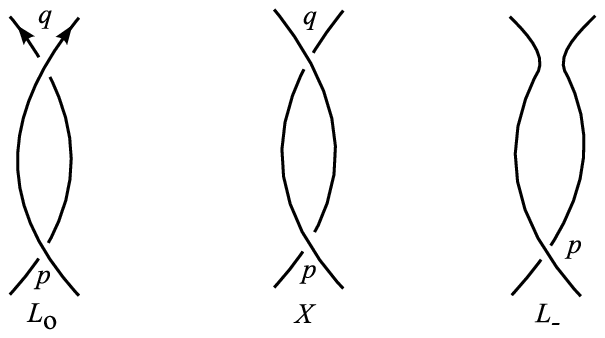}
\caption{}\label{Fig 5.9}
\end{figure}

Now let us change the orientation of the component of $L_\circ$
which contains the upper right corner of the diagram (Figure
\ref{Fig 5.10}).  We get the link $L'_\circ$.  We change similarly
$X$ into $X'$.  Now let us choose the orientation of $L_\infty$ so
it agrees with the orientation of $L'_\circ$ (Figure \ref{Fig 5.10};
$L_\infty = X^{pq}_{-\infty} = X^{pq}_{\infty -}
=X'^{pq}_{+\circ}$).  From the diagrams of Figure \ref{Fig 5.10}
(considering the crossing $q$) one gets:

\begin{figure}
\centering
\includegraphics{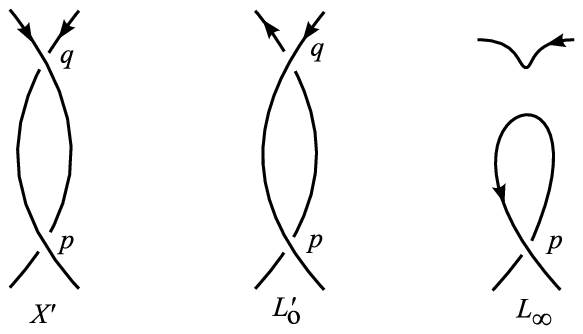}
\caption{}\label{Fig 5.10}
\end{figure}

$$-tV_{X'}(t) +
\frac{1}{t}V_{L'_\circ}(t) =
\left(t-\frac{1}{\sqrt{t}}\right)V_{L_\infty}(t)$$ and because of
Lemma \ref{Lemma 5.15} $V_{L'_\circ}(t) =
t^{3\lambda}V_{L_\circ}(t)$ and $V_{X'}(t) =
t^{3(\lambda-1)}V_X(t)$, therefore we get:

\item $-\frac{1}{t^2} V_X(t) +
\frac{1}{t}V_{L_\circ}(t)=\left(\sqrt{t}-\frac{1}{\sqrt{t}}\right)t^{-3\lambda}V_{L_\infty}(t).$
The triple $L_+$, $L_-$, and $L_\circ$ gives us the equation

\item $-tV_{L_+}(t)+\frac{1}{t}V_{L_-}(t) = \left(t -
\frac{1}{t}\right)V_{L_\circ}(t)$.
\end{enumerate}

The equation (b)$+\frac{1}{t}$(a)$-\frac{1}{t}$(c) gives us the
$V_\infty$-formula

\item[(ii)] $c(L_+) > c(L_\circ)$.

Let $L'_+$ be the link obtained from $L_+$, by changing the
orientation of $L_i$.  Similarly $L'_-$ is get from $L_-$.  Now the
smoothing of $L'_+$ is exactly $L'_\circ = L_\infty$.  Let us use
the defining equation for the Jones polynomial into the triple
$L'_-$, $L'_+$, $L'_\circ$ we get: $-t V_{L'_-}(t) +
\frac{1}{t}V_{L'_+}(t)=\left(\sqrt{t}-\frac{1}{\sqrt{t}}\right)V_{L'_\circ}$,
now from Lemma \ref{Lemma 5.15} we get: $-tV_{L'_+}(t) =
t^{3\lambda} V_{L_+}(t), ~ V_{L_-}(t)= t^{3(\lambda -
1)}V_{L_-}(t)$.  Therefore

\begin{enumerate}

\item[(d)] $-\frac{1}{t^2} V_{L_-}(t) +
\frac{1}{t}V_{L_+}(t)=\left(\sqrt{t}-\frac{1}{\sqrt{t}}\right)t^{-3\lambda}V_{L_\infty}(t)$
\end{enumerate} what ends the proof of (ii) and of Lemma \ref{Lemma 5.16}.
\end{enumerate}
\end{proof}

Now we are ready to finish the proof of Theorem \ref{Theorem 5.13}.
We follow in this the paper of Lickorish \cite{Li_2}.

As usual we consider two cases: \begin{enumerate}

\item[(i)]  $c(L_+)<c(L_\circ)$

Consider the formula which defines the Jones polynomial and
$V_\infty$-formula.  We get $$\begin{array}{l} -tV_{L_+}(t) +
\frac{1}{t}V_{L_-}(t) = \left(\sqrt{t} -
\frac{1}{\sqrt{t}}\right)V_{L_\circ}(t) \\
\sqrt{t}V_{L_+}(t)-\frac{1}{\sqrt{t}} V_{L_-}(t) = \left(\sqrt{t} -
\frac{1}{\sqrt{t}}\right)t^{-3\lambda}V_{L_\infty}(t).\end{array}$$

If we add this two formulas we get the formula (i) from Lemma
\ref{Lemma 5.14}(2) for $a=t^{3/4}$, $z=-(t^{-1/4}+t^{1/4})$.

\item[(ii)]  $c(L_+)>c(L_\circ)$

We proceed in the same manner as in the case (i) to get the formula
(ii) from Lemma \ref{Lemma 5.14}(2).
\end{enumerate}

It completes the proof of Theorem \ref{Theorem 5.13}.

\end{proof}

Kauffman \cite{K_6} has found a nice characterization of the Jones
polynomial which is of a great importance for alternating links
\cite{Mu_2}, \cite{K_6}, \cite{K_7}, \cite{Mu_3}.

This characterization follows easily from the $V_\infty$-formula.

\begin{cor}\label{Corollary 5.17}  Consider the polynomial
invariant of the regular isotopy $\widetilde{V}_L(t) =
t^{\frac{3}{4}\tw(L)}V_L(t)$.  Then $\widetilde{V}_L(t)$ is uniquely
determined by the following conditions:
\begin{enumerate}

\item $\widetilde{V}_{T_1}(t) = t^{\frac{3}{4}\tw(T_1)}$ where
$T_1$ is isotopic to the trivial knot,

\item $\widetilde{V}_{L_+} =
-t^{\frac{1}{4}}\widetilde{V}_{L_\circ}(t) -
t^{-\frac{1}{4}}\widetilde{V}_{L_\infty}(t)$,

\item $\widetilde{V}_{L_-} =
-t^{-\frac{1}{4}}\widetilde{V}_{L_\circ}(t) -
t^{\frac{1}{4}}\widetilde{V}_{L_\infty}(t)$.\end{enumerate}\end{cor}

\begin{proof} $\widetilde{V}_L(t)$ is by the definition of the
invariant of unoriented diagrams ($V_L$ is a special case of the
Kauffman polynomial).  To prove the corollary it is enough to show
the properties (2) and (3).

From the formula which defines the Jones polynomial and from
$V_\infty$-formula one gets:

$$\begin{array}{l}-t^{\frac{1}{4}}\widetilde{V}_{L_+}(t) +
t^{-\frac{1}{4}}\widetilde{V}_{L_-}(t)=\left(\sqrt{t} -
\frac{1}{\sqrt{t}}\right)\widetilde{V}_{L_\circ}(t)\textrm{, and}\\
t^{-\frac{1}{4}}\widetilde{V}_{L_+}(t) -
t^{\frac{1}{4}}\widetilde{V}_{L_-}(t)=\left(\sqrt{t} -
\frac{1}{\sqrt{t}}\right)\widetilde{V}_{L_\infty}(t).\end{array}$$

Now one gets, eliminating $\widetilde{V}_{L_-}(t)$:

$$\begin{array}{l}-t^{\frac{3}{4}}\widetilde{V}_{L_+}(t) +
t^{-\frac{1}{4}}\widetilde{V}_{L_+}(t)=\left(\sqrt{t} -
\frac{1}{\sqrt{t}}\right)(\sqrt{t}\widetilde{V}_{L_\circ}(t)+\widetilde{V}_{L_\infty}(t))\textrm{,
so}\\ \widetilde{V}_{L_+}(t) \left(-\sqrt{t} +
\frac{1}{\sqrt{t}}\right)=\left(\sqrt{t} -
\frac{1}{\sqrt{t}}\right)(t^{\frac{1}{4}}\widetilde{V}_{L_\circ}(t)
+ t^{-\frac{1}{4}}\widetilde{V}_{L_\infty}(t))\end{array}$$ what is
equivalent into the formula (2).

Similarly, we get formula (3).

\end{proof}

Kauffman reformulated the condition (2) and (3) into the form
\begin{enumerate}

\item[($2'$)] $\widetilde{V}_{\includegraphics{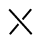}}(t) =
-t^{\frac{1}{4}}\widetilde{V}_{\smilefrownSmall}(t) -
t^{-\frac{1}{4}}\widetilde{V}_{\infinityfig}(t)$ and

\item[($3'$)] $\widetilde{V}_{\includegraphics{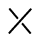}}(t) =
-t^{-\frac{1}{4}}\widetilde{V}_{\smilefrownSmall}(t) -
t^{\frac{1}{4}}\widetilde{V}_{\infinityfig}(t)$.

\end{enumerate}

This approach allowed Kauffman to give different proof of the Jones
reversing result and the Birman $V_\infty$-formula \cite{K_6}.

K.~Murasuqi \cite{Mu_2}, \cite{Mu_3} (see also \cite{K_6} and
\cite{K_7}) has used the above Corollary to prove the classical Tait
\cite{Ta} conjecture about alternating links.  Namely: A link
projection $\widetilde{L}$ is called proper if $\widetilde{L}$ does
not contain ``removable'' double points like
$\includegraphics{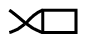}$~. The reduced degree $r$-deg of the
polynomial $ \displaystyle w=\sum_m^n a_it^i$ (where $a_m, a_n \neq
0$) is defined to be $r$-deg $w = n=m$.

\begin{thm}\label{Theorem 5.18} \begin{enumerate}\item[(a)] If
$\widetilde{L}$ is a connected and proper alternating projection of
an alternating link $L$, then $r$-deg $V_L(t) = \crs(\widetilde{L})$
where $\crs(\widetilde{L})$ denotes the number of crossings of $L$.

\item[(b)] If $L$ is a prime link, then for any non-alternating
projection $\widetilde{L}$ of $L$, $r$-deg $V_L(t) <
\crs(\widetilde{L})$...

\item[(c)] Two (connected and proper alternating projections of an
alternating link have the same number of crossings.
\end{enumerate}\end{thm}

For the proof we refer to \cite{Mu_3}.

We can introduce a relation on diagrams of links which naturally
limits the possible use of the Kauffman method (similarly as skein
equivalence is the limit for Conway type invariants).

\begin{defn}\label{Definition 5.19}Consider the space
$\mathcal{S}$ of partially oriented diagrams (i.e. some components
of a diagram are oriented) up to regular isotopy.  The Kauffman
equivalence relation ($\sim_K$) is the smallest equivalence relation
on $\mathcal{S}$ which satisfies the following condition:

Let $L'_1$ (respectively $L'_2$) be a diagram of a link $L_1$
(respectively $L_2$) with a given crossing $p_1$ (respectively
$p_2$) and \begin{enumerate}

\item[(i)] $(L'_1)^{p_1}_{-\sgn p_1} \sim_K (L'_2)^{p_2}_{-\sgn
p_2}$ where $L^p_{-\sgn p}$ denotes the link obtained from $L$ by
interchanging the bridge and the tunnel at $p$ (it does not depend
an orientation or lack of orientation of $L$).

\item[(ii)]  $(L'_1)^{p_1}_\circ \sim_K (L'_2)^{p_2}_\circ$ and
$(L'_1)^{p_1} \sim_K (L'_2)^{p_2}$ where $p_1$ is a crossing of
oriented components of $L'_1$ or $p_1$ is a self-crossing or some
component of $L'_1$ ( in the case of a self-crossing no orientation
is needed to distinguish $(L'_1)^{p_1}_\circ$ from
$(L'_1)^{p_1}_\infty$).

\item[(iii)]
$\left\{(L'_1)^{p_1}_\circ,(L'_1)^{p_1}_\infty\right\} =
\left\{(L'_2)^{p_2}_\circ,(L'_2)^{p_2}_\infty\right\}$ (equality of
the pairs of Kauffman equivalence classes) if $p_1$ is a crossing of
components of $L'_1$ one of which is not oriented.\end{enumerate}

Then $L_1 \sim_K L_2$. \end{defn}

\begin{cor}\label{Corollary 5.20} \begin{enumerate}\item[(a)] If the
oriented diagram $L_1$ is a mutant of the oriented diagram $L_2$
then $L_1 \sim_K L_2$.

\item[(b)] If $L_1 \sim_K L_2$ ($L_1$, $L2$ oriented) then
$\tw(L_1) = \tw(L_2)$, $P_{L_1}(x,y) = P_{L_2}(x,y)$ and
$F_{L_1}(a,z) = F_{L_2}(a,z)$.\end{enumerate}\end{cor}

\begin{proof}(a) We can build the same resolving tree for $L_1$
and $L_2$ (compare Lemma \ref{Lemma 3.7}).

(b) It follows from the definition of $\sim_K$ and from the fact
that $\tw L_+ = \tw L_\circ +1 = \tw L_- +2$.
\end{proof}

Now we will show how invariants of links got by the Kauffman method
can be described by an algebraic structure (similarly as Conway
algebra yielded invariants of Conway type).  We will also construct
a polynomial invariant of oriented links which generalizes at once
Jones-Conway and Kauffman polynomials (however it does not give more
information than these two polynomials).  We proceed similarly as in
the case of the Conway algebra but we consider diagrams up to
regular isotopy.  There is no need to distinguish positive crossing
from negative one so we work (in principle) with one 3-argument
operation $\ast$ which allows us to recover the value of invariant
for $L_+$ (respectively $L_-$) from its values for $L_-$, $L_\circ$
and $L_\infty$ (respectively $L_+, L_\circ, L_\infty$).

However we have to solve one important problem:  if we change
$L_+$ into $L_\infty$ then new component of $L_\infty$ does not
have any natural orientation.  One possibility is to consider
partially oriented link $L_\infty$.  We will not follow this way
from practical reasons.  Namely, we would like to use the same
scheme of the proof as in Theorem \ref{Theorem 2.1.1} and to make
it we need the equality $L^{pq}_{\epsilon_1
\epsilon_2}=L^{qp}_{\epsilon_2 \epsilon_1}$, $\epsilon_1,
\epsilon_2 \in \{+,-,\circ, \infty\}$ i.e. if we perform some
surgeries on two crossings $p$ and $q$ the result should not
depend on the order of performing these operations. If we consider
partially oriented links it is not always the case (see Figure
\ref{Fig 5.11}).

\begin{figure}[htbp]
\centering
\includegraphics{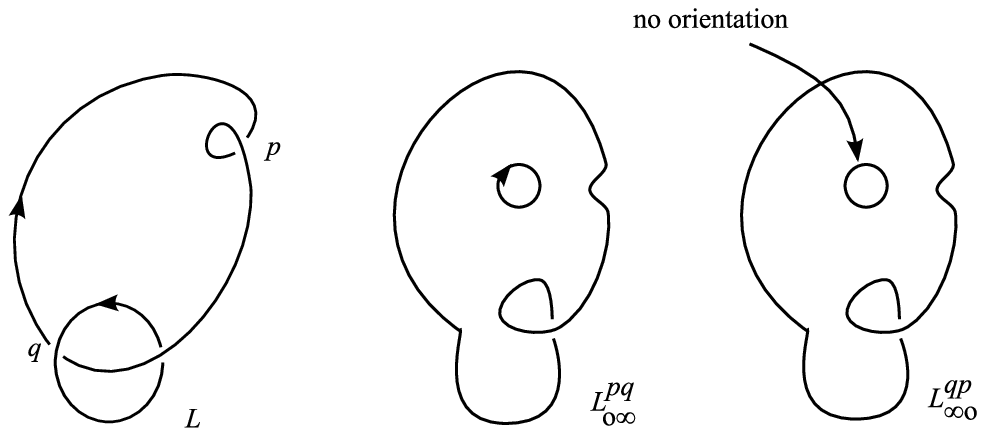}
\caption{}\label{Fig 5.11}
\end{figure}

Therefore we will limit ourself to the case of oriented and
non-oriented links (i.e. all the components are oriented or no
component is oriented).  In the last case we do not distinguish
$L_\circ$ from $L_\infty$.

Consider the following general situation.  Assume we are given an
abstract algebra $\mathcal A$ with two universal (sets) $A$ and
$A'$, a countable number of $0$-argument operations in $A$ and
$A':\{a_{ij}\}_{i\in \nats, j\in \ints}, ~ \{a'_{ij}\}_{i\in
\nats, j\in \ints}$, two 3-argument operations $\ast: A\times A
\times A' \to A$ and $\ast':A' \times A' \times A' \to A'$ and
1-argument operation $\varphi:A\to A'$.  We would like to
construct invariants of classes of regular isotopy of oriented and
non-oriented diagrams which satisfy the following conditions:
\begin{enumerate}

\item[(a)] If $L$ is an oriented link then the value of the
invariant $w_L \in A$, and if $L'$ is non-oriented then $w_{L'} \in
A'$.

\item[(b)] If $L'$ is the non-oriented diagram obtained from an
oriented diagram $L$ by ignoring the orientation the $w_{L'} =
\varphi(w_L)$.

\item[(c)] $w_{T_{i,j}} = a_{i,j}$ where $T_{i,j}$ is an oriented
diagram of the trivial link of $i$ components and $\tw(T_{i,j}) =
j$.

\item[(d)] $w_{T'_{i,j}} = a'_{i,j}$.

\item[(e)] $w_{L^p} = w_{L^p_{-\sgn p}}\ast (w_{L^p_\circ},
w_{L^p_\infty})$ where $L^p_{-\sgn p}$ denotes the diagram obtained
from $L$ by interchanging the bridge and tunnel at $p$.

\end{enumerate}

\setcounter{thm}{19}

\begin{defn}\label{Definition 5.20} We say that $\mathcal{A} =
\{A, A', \{a_{ij}\}, \{a'_{ij}\}, \ast, \ast', \varphi\}$ is a
Kauffman algebra if the following conditions are satisfied:

\begin{enumerate}
\item[K1] $\varphi(a_{i,j})=a'_{i,j}$

\item[K2] $\varphi(a\ast(b,c))=\varphi(a)\ast' (\varphi(b),c)$
where the operation $\ast$ on $(a,b,c)$ is denoted by $a \ast
(b,c)$; similarly for the operation $\ast'$.

\item[K3] $a_{i,j-1}\ast (a_{i+1,j}, a'_{i,j}) = a_{i, j+1}$.

\item[K4] $(a \ast (b,c))\ast(d\ast(e,f), g\ast ' (h,i))= (a\ast
(d,g))\ast(b \ast (e,h), c \ast'(f,i))$ where $a,b,d,e \in A$ and
$c,f,g,h,i \in A'$

\item[K5] $(a\ast (b,c))\ast(b,c)=a$

\item[K6] $a \ast'(b,c) = a \ast'(c,b)$.

\end{enumerate}\end{defn}

\begin{thm}\label{Theorem 5.21} For a given Kauffman algebra
$\mathcal A$ there exists a uniquely determined invariant of regular
isotopy $w$, which attaches an element $w_L$ from $A$ to every
oriented diagram $L$ and an element $w_{L'}$ of $A'$ to every
non-oriented  diagram $L'$ and satisfies the following conditions:

\begin{enumerate}

\item $w_{T_{i,j}} = a_{i,j}$

\item $w_{L'} = \varphi(w_L)$ where $L'$ is obtained from oriented
diagram $L$ by ignoring orientation.

\item $w_{L^p} = w_{L^p_{-\sgn p}} \ast (w_{L^p_\circ},
w_{L^p_\infty})$.

\end{enumerate}\end{thm}

\begin{proof}  The proof of Theorem \ref{Theorem 5.21} is very
similar to that of Theorem \ref{Theorem 2.1.1}.  Therefore we will
give details only in parts in which bigger differences occur.

We can build a resolving tree for each diagram (oriented or not) in
such a way that each vertex represents a diagram and
\begin{enumerate}

\item[(i)] Descending diagrams lie in leaves (descending for some
choice of base points and orientation in the case of a
non-orientable diagram)

\item[(ii)] The situation at each vertex (except leaves) looks
like on Figure \ref{Fig 5.12}

\begin{figure}[htbp]
\centering

$\xymatrix { & & L \ar@{-}[dll] \ar@{-}[d] \ar@{-}[drr] & &
\\  L^p_{-\sgn p} & & L^p_\circ & & L^p_\infty  }$

\caption{}\label{Fig 5.12}
\end{figure}

Such a tree can be used to compute the invariant of the root
diagram.
\end{enumerate}

We start the proof of Theorem \ref{Theorem 5.21} by constructing the
function $w$ on diagrams and then show that it is not changed by the
Reidemeister moves $\Omega^{\mp 1}_{0.5}$, $\Omega^{\mp 1}_2$, and
$\Omega^{\mp 1}_3$.  We use induction on the number $\crs(L)$ of
crossings points in the diagram.  For each $k \geq 0$ we define a
function $w_k$ assigning an element of $A$ (respectively $A'$) to
each oriented (respectively non-oriented) diagram with no more than
$k$ crossings.  Then $w$ will be defined for every diagram by $w_L =
w_k(L)$ where $k \geq \crs(L)$. Similarly as in the proof of Theorem
\ref{Theorem 2.1.1} we define $w_0(L) = a_{n,0}$ if $L$ is a trivial
oriented diagram of $n$ components, and $w_0(L') = a'_{n,0}$ if $L'$
is obtained from $L$ by ignoring the orientation.  Then we formulate
the Main Inductive Hypothesis (M.I.H.):  We assume that we have
already defined $w_k$ attaching an element of $A$ (respectively
$A'$) to each diagram $L$ for which $\crs(L)\leq k$ and that $w_k$
has the following properties:
\begin{enumerate}\item[5.22]\label{Condition 5.22} $w_k(U_{n,j}) =
a_{n,j}$, where $U_{n,j}$ is an oriented descending (for some choice
of base points) diagram of $n$ components and $\crs(U_{n,j})\leq k$,
$\tw(U_{n,j}) = j$

$w_k(U'_{n,j})=a'_{n,j}$ where $U'_{n,j}$ is obtained from $U_{n,j}$
by ignoring the orientation.

\item[5.23]\label{Condition 5.23} $w_k(L)=w_k(L^p_{-\sgn p})\ast
(w_k(L^p_\circ), w_k(L^p_\infty))$ if $L$ is an orientable
diagram, and

$w_k(L) = w_k(L^p_{-\sgn p})\ast' (w_k(L^p_\circ), w_k(L^p_\infty))$
if $L$ is a non-orientable diagram.

\item[5.24]\label{Condition 5.24} $w_k(L) = w_k(R(L))$ where $R$
is a Reidemeister move of type $\Omega^{\mp 1}_{0.5}$,
$\Omega^{\mp 1}_2$, or $\Omega^{\mp 1}_3$ and $\crs(L), \crs(R(L))
\leq k$.

\end{enumerate}

\setcounter{thm}{24}

Then we want to make the Main Inductive Step (M.I.S.) in order to
obtain the existence of a function $w_{k+1}$ with analogous
properties defined on diagrams with at most $k+1$ crossings.  It
will complete the proof of Theorem \ref{Theorem 5.21} analogously
as in the case of Theorem \ref{Theorem 2.1.1}

The proof of M.I.S.~ begins, as in Theorem \ref{Theorem 2.1.1}, by
defining a function $w_b$ which, for diagrams with $\crs(L)=k+1$,
depends on the choice of base points $b=(b_1, \ldots, b_n)$ and on
the choice of the orientation of $L$ in the case $L$ was
non-oriented.  $w_b(L)=w_k(L)$ if $\crs(L) \leq k$.  For
$\crs(L)=k+1$, we define $w_b$ by induction on the number of bad
crossings ($b(L)$) of the diagram $L$ using condition 5.1 or the
formula 5.2 to the first bad crossing.  The we show that formula 5.2
holds for every crossing.  The proof in this point does not differ
from the analogous point in the proof of Theorem \ref{Theorem 2.1.1}
(K4 is used instead of the conditions C3--C5).

The next step of the proof is to show, that $w_b$ does not depend on
the choice of $b$ (for a given orientation and order of components).
We proceed as in \ref{Theorem 2.1.1} choosing base points $b$ and
$b'=(b_1, b_2, \ldots ,b'_i, \ldots ,b_n)$ in such a way that $b'_i$
lies after $b_i$ in the $i$th component $L_i$ of $L$ and there is
exactly one crossing point between $b_i$ ad $b'_i$.

We use induction on $B(L) = \max(b(L), b'(L))$.  If $B(L)=0$ then
$L$ is descending with respect to both choices of base points,
therefore $w_b(L) = w_{b'}(L) = a_{n, \tw(L)}$.  If $B(L)>1$ or
$B(L) = b(L) = b'(L)=1$ then $L$ has a bad crossing with respect to
$b$ and $b'$.  We use then the inductive hypothesis resolving the
diagram along this crossing (i.e. using condition 5.2).  It remains
to consider the case $B(L)=1$, $b(L)\neq b(L')$.  The proof in this
case is little more involved than in analogous place of \ref{Theorem
2.1.1}.  Namely:

Let $p$ be the only bad crossing of $L$ with respect to $b$ or $b'$.
$p$ is a self-crossing of a component $L_i \subset L$.

Assume, for simplicity, that $L$ is oriented and $b(L) =1$,
$b'(L)=0$, $\sgn p=+$. Therefore $L$ is a descending diagram with
respect to $b'$ and $$w_{b'}(L) = a_{n, \tw(L)}.$$  From the
property 5.2, $$w_b(L) = w_b(L^p_{-})\ast (w_b(L^p_\circ),
w_b(L^p_\infty)).$$ $b(L_-^p) = 0$, so $w_b(L_-^p) = a_{n,
\tw(L)-2}$.

$L^p_\circ$ is a descending diagram with respect to a proper choice
of base points therefore $w_b(L^p_\circ) = a_{n+1, \tw(L)-1}$.  We
need the equality $w_b(L^p_\infty) = a'_{n, \tw(L)-1}$ in order to
use K3 and to get $w_b(L) = a_{n,\tw(L)}$. We cannot get it
immediately. In fact $L^p_\infty$ does not need to be a descending
diagram with respect to any choice of base points.  We can use
however the fact that $L^p_\infty$ has only $k$ crossings.
Furthermore $L^p_\infty$ consists of two parts one is descending and
the second ascending (with respect to proper choice of base points
and orientation) and these parts may be put on different levels
(Figure \ref{Fig 5.13}).  In order to show that $w_b(L^p_\infty) =
a'_{n, \tw(L)-1}$ we will use the following trick:

\begin{figure}[htbp]
\centering
\includegraphics{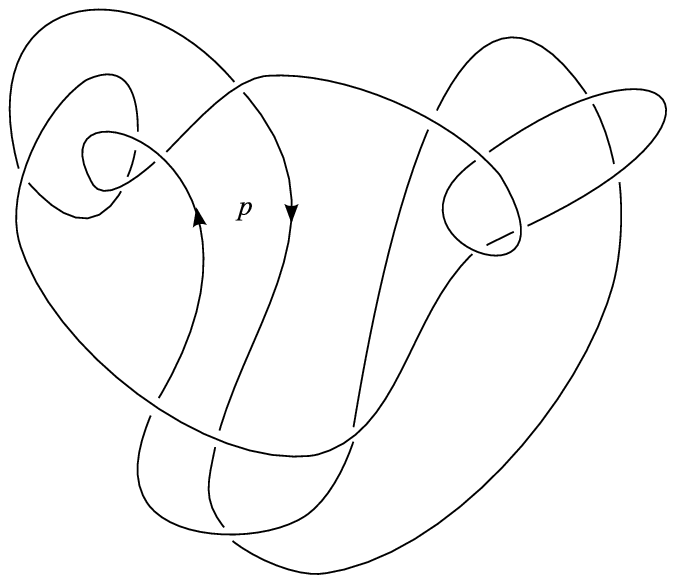}
\caption{}\label{Fig 5.13}
\end{figure}

Rotate the ascending part of the diagram $L^p_\infty$ $180^\circ$
with respect to the vertical (N-S) axis and then change the
orientation of this part of $L^p$ (we make some kind of mutation).
We get the descending diagram $\widetilde{L}$.  Therefore
$w_b(\widetilde{L}) = w_k(\widetilde{L}) = a'_{n, \tw(L)-1}$.  On
the other hand, we can build for $L^p_\infty$ and $\widetilde{L}$
the same resolving tree, each vertex of which corresponds to a
diagram with no more than $k$ crossings (analogy with mutation is
complete). Now we conclude that $ w_b(L^p_\infty) = w_k(L^p_\infty)
= w_k(\widetilde{L}) = a'_{n, \tw(L)-1}$.  This completes the proof
of this part (compare \cite{B_L_M}).

The rest of the proof of Theorem \ref{Theorem 5.21} is almost the
repetition of the analogous part of the proof of Theorem
\ref{Theorem 2.1.1}.  We change Reidemeister move $\Omega^{\mp 1}_1$
by $\Omega^{\mp 1}_{0.5}$.  Then the Lemma \ref{Lemma 2.2.14}
remains valid and it can be additionally used to show that $w_b(L)$
does not depend on the orientation of $L$ (if $L$ is not oriented).
Thus we can complete the proof of Theorem \ref{Theorem 5.21}.

\end{proof}

\begin{example}[Jones-Conway-Kauffman polynomial]\label{Example
5.25} The following $\mathcal A$ is a Kauffman algebra.

$$A= \ints[a^{\mp 1}, t^{\mp 1}, a],~ A' = \ints[a^{\mp 1}, t^{\mp 1}],$$

$$a_{i,j} =\left({\frac{a+a^{-1}}{t}}\right)^{i-1}\left(1-\frac{z}{t}\right)a^j
+ \frac{z}{t}\left(\frac{a+a^{-1}}{t} -1\right)^{i-1}a^j,$$
$$a'_{i,j} = \left(\frac{a+a^{-1}}{t}-1\right)^{i-1}a^j,$$ $b\ast
(c,d)$ is defined by the equation $b\ast(c,d) + b =tc+zd$, $b\ast'
(c,d)$ is defined by the equation $b\ast' (c,d) + b = tc +td$ and
finally $\varphi$ is defined on generators by $\varphi(a) =a$,
$\varphi(t)=t$, $\varphi(z)=t$.

We will check now that $\mathcal A$ is a Kauffman algebra.  The
conditions K1, K2, K5, and K6 follow immediately from the definition
of $\mathcal A$.

The condition K3 follows from the equality $$\begin{array}{l}
\left({\frac{a+a^{-1}}{t}}\right)^{i-1}\left(1-\frac{z}{t}\right)a^{j+1}
+ \frac{z}{t}\left(\frac{a+a^{-1}}{t} -1\right)^{i-1}a^{j+1} + \\
+
\left({\frac{a+a^{-1}}{t}}\right)^{i-1}\left(1-\frac{z}{t}\right)a^{j-1}
+ \frac{z}{t}\left(\frac{a+a^{-1}}{t} -1\right)^{i-1}a^{j-1} =
\\ = t\left(
\left({\frac{a+a^{-1}}{t}}\right)^{i}\left(1-\frac{z}{t}\right)a^j +
\frac{z}{t}\left(\frac{a+a^{-1}}{t} -1\right)^{i}a^j\right) + z
\left(\frac{a+a^{-1}}{t}-1\right)^{i-1}a^j\end{array}$$

It remains to show the condition K4: $$\begin{array}{l} (a
\ast(b,c))\ast(d\ast(e,f), g \ast'(h,i)) = -(a
\ast(b,c))+t(d\ast(e,f))+z(g\ast (h, i))= \\
=-(-a+tb+zc)+t(-d+te+zf)+z(-g+th+ti)= \\ = a-tb-zc-td+t^2e + tzf
-zg+zth+zti = \\ =(a\ast(d,g))\ast(b \ast(e,h),c \ast'
(f,i)).\end{array}$$

The invariant of regular isotopy of oriented diagrams, $J_L(a,t,z)$,
yielded by the algebra $\mathcal A$ is called the
Jones-Conway-Kauffman (or JCK) polynomial.  It can be modified into
invariant of oriented links:
$$\widetilde{J}_L(a,t,z)=J_L(a,t,z)a^{-\tw(L)}.$$

\end{example}

\begin{example}\label{Example 5.26}  We will compute the value of
the Jones-Conway-Kauffman polynomial for the diagram of the right
handed trefoil knot (Figure \ref{Fig 5.14}).  We get, using the
resolving tree from Figure \ref{Fig 5.14}, that in any Kauffman algebra
$w_L = a_{1,1}\ast(a_{11}, a'_{1,-1})a'_{1,-2})$.

\begin{figure}[htbp]
\centering
\includegraphics{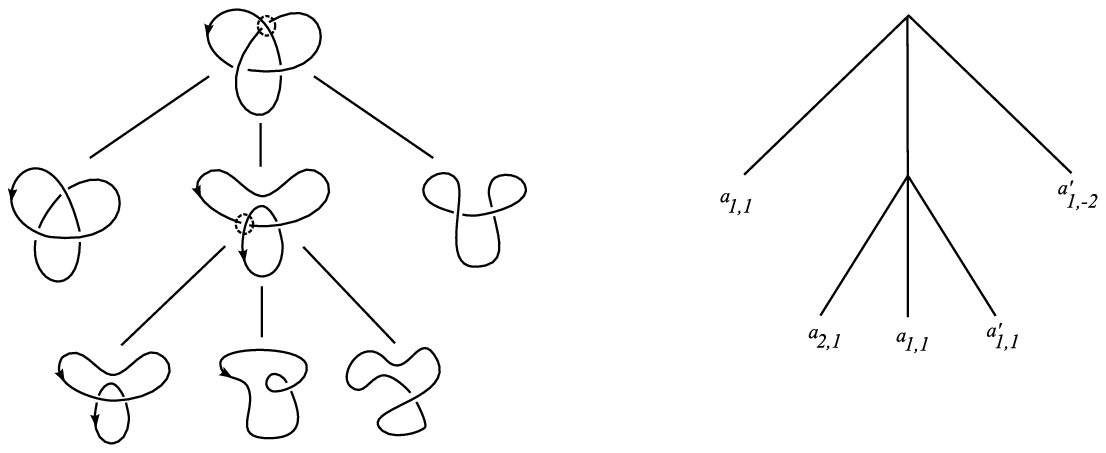}
\caption{}\label{Fig 5.14}
\end{figure}

Therefore we get $$\begin{array}{l}J_L(a,t,z,) =
-a+t\left(-\frac{a+a^{-1}-z}{t}+ta+za^{-1}\right)+za^{-2} = \\=
-a^{-1}-2a+t^2a+z(1+a^{-2}+ta^{-1})\textrm{, and}\end{array}$$

$$\widetilde{J}_L(a,t,z) =
-a^{-4}-2a^{-2}+t^2a^{-2}+z(a^{-3}+a^{-5}+ta^{-4}).$$

\end{example}

\begin{lem}\label{Lemma 5.27} \begin{enumerate}

\item[(a)] $J_L(a,t,z) = J_L(a,t,0) + z\left(\frac{J_L(a,t,t) -
J_L(a,t,0)}{t}\right)$

\item[(b)] $\widetilde{J}_L(a,t,0) = P(\frac{a}{t},
\frac{1}{at})$, it is some variant of the Jones-Conway polynomial.

\item[(c)] $J_L(a,t,t)=G_L(a,t)$, it is Kauffman polynomial for
regular isotopy.\end{enumerate}\end{lem}

\begin{proof} (a) It is true for diagrams representing trivial
links.  Then we proceed by induction on the height of the resolving
tree of the diagram.

(b) and (c).  It is enough to check the initial conditions and
compare the 2- or 3-argument operations used in the definitions.

Lemma 5.27 shows that the Jones-Conway-Kauffman polynomial is
equivalent to Jones-Conway and Kauffman polynomials.  There is a
remarkable similarity in it with Proposition 3.38.  It
follows immediately that any invariant yielded by a Kauffman
algebra (e.g. Jones-Conway-Kauffman polynomial) is an invariant of
$\sim_K$ equivalence of oriented or non-oriented diagrams.
\end{proof}

\begin{rem}\label{Remark 5.28} The theory of invariants yielded by Kauffman
algebras can be developed similarly as the theory of invariants
yielded by Conway algebras.  In particular:

\begin{enumerate}

\item[(a)] One can look for involutions $\tau$ on $A$ and $\tau'$ on
$A'$ such that $\tau(a_{i,j})=a_{i,-j}$, $\varphi(\tau(w)) =
\tau'(\varphi(w))$ where $w\in A$, $\tau(a\ast(b,c)) =
\tau(a)\ast(\tau(b), \tau'(c))$.  Then $A_{\overline{L}}=\tau(A_L)$
where $A_L$ is the value of the invariant for an oriented diagram
$L$ and $\overline{L}$ denotes the mirror image of $L$ (compare
Lemma \ref{Lemma 3.16}).  For the Kauffman algebra which yields the
Jones-Conway-Kauffman polynomial $\tau:\ints[a,t,z] \to
\ints[a,t,z]$ exists and is given on the generators by $\tau(a) =
a^{-1}$, $\tau(t)=t$, $\tau(z)=z$.

\item[(b)] It is possible to build the universal Kauffman algebra
(using terms) and to show that for such the universal Kauffman
algebra the involutions $\tau$ and $\tau'$ exist.

\item[(c)] It is sensible to look for an operation $o:A\times
A\times A' \to A$ which for orientable diagrams will recover the
value of the invariant for $L_\circ$ from its values for $L_+$,
$L_-$, and $L_\infty$.  The operation $o$ exists for the Kauffman
algebra which yields the [J-C-K] polynomial.

\item[(d)]  One can look for conditions which a Kauffman algebra
should satisfy if we want simple formulas for the value of
invariants of connected and disjoint sums of diagrams.

\item[(e)] One can look for conditions which a Kauffman algebra
should satisfy if we want to modify the invariant of regular
isotopy of diagrams yielded by the algebra into invariant of
isotopy of links (e.g. if there exist two bijections $\beta : A
\to A$ and $\beta':A'\to A'$ such that $\beta(a_{i,j}) = a_{i,
j-1}$, $\varphi(\beta(a)) = \beta'(\varphi(a))$ and $\beta(a\ast
(b,c)) = \beta(a) \ast (\beta(b), \beta(c))$ then
$\beta(\beta(\ldots \beta(A_L)\ldots))$ (where $\beta$ is applied
$\tw(L)$-times) is an invariant of isotopy of $L$).

\item[(f)] We can consider geometrically sufficient partial Kauffman
algebras (we modify Kauffman algebras in the same way as Conway
algebras -- Definition \ref{Definition 4.2}) which will yield
regular isotopy invariants of oriented or nonoriented diagrams.

\item[(g)] We can build a polynomial of infinitely many variables
which will generalize the J-C-K polynomial (similarly as in the case
of Jones-Conway polynomial; Example \ref{Example 4.5}).

\item[(h)] One can show that the invariant yielded by a
geometrically sufficient partial Kauffman algebra is invariant under
mutation of oriented or non-oriented diagrams (see Corollary
\ref{Corollary 5.20}).

\end{enumerate} \end{rem}

Many of which we formulated before for invariants of Conway type may
be considered also for invariants got by the Kauffman method.

\begin{prob}\label{Problem 5.29}

\begin{enumerate}

\item[(a)] Do there exist two oriented diagrams, which have the same
Jones-Conway-Kauffman polynomial but which can be distinguished by
some invariant yielded by a Kauffman algebra?

\item[(b)] Do there exist two oriented diagrams which have the same
value of invariant yielded by any Kauffman algebra but which can be
distinguished by some invariant yielded by a geometrically
sufficient partial Kauffman algebra?

\item[(c)] Do there exist two oriented diagrams which are not
$\sim_K$ equivalent but which cannot be distinguished by the
invariant yielded by any geometrically sufficient partial Kauffman
algebra?

\item[(d)] Assume that an oriented diagram of a knot $L$ satisfies
$L \sim_K \overline{L}$.  Does it follow that $L$ is isotopic to
$\overline{L}$ or $-\overline{L}$?

\item[(e)] Assume that oriented knots $L_1$ and $L_2$ have the same
value of the Kauffman polynomial.  Can it happen that these knots
have different Jones-Conway polynomials? (In particular is it
possible if $L_2 = \overline{L_1}$?)
\end{enumerate}

\end{prob}

The knot $9_{42}$ (in the Rolfsen \cite{Ro} notation) ha the same
value of the JCK polynomial $\widetilde{J}(a,t,z)$ as its mirror
image but different signature.  The signature of knots is a skein
invariant and it can be yielded (together with the determinant) by
some geometrically sufficient partial Kauffman algebra.

The problem (d) is a weak version of the Kauffman conjecture
(\ref{Conjecture 5.11}).  It is true for knots up to 9 crossings and
the only knots up to 11 crossings for which it still should be
verified are $10_{71}$ (in the Rolfsen notation) and $11_{449}$ (in
the Thistlethwaite \cite{Thist_2} notation).  The second part of the
problem (e) is true for knots up to 11 crossings.

The Kauffman polynomial seems to be powerful in distinguishing
closed 3-braids.

\begin{conj}\label{Conjecture 5.30} Let $\gamma$ be a closed 3-braid
which closure is not isotopic to the mirror image.  The
$$\widetilde{J}_{\hat{\gamma}}(a,t,z) \neq
\widetilde{J}_{\hat{\overline{\gamma}}}(a,t,z).$$\end{conj}

\begin{prob}\label{Problem 5.31} When we have defined invariants of
diagrams using Kauffman algebras or we have defined the relation
$\sim_K$ we have had the problem with orientation of $L^p_\infty$.
New component of $L^p_\infty$ inherits from $L$ two different
orientations on its pieces (Figure \ref{Fig 5.15})\end{prob}

\begin{figure}[htbp]
\centering
\includegraphics{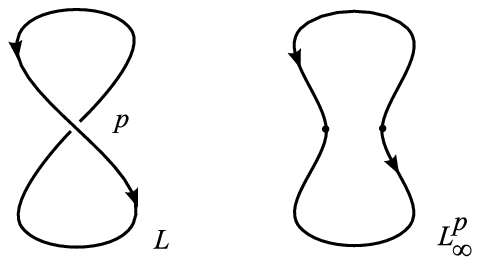}
\caption{}\label{Fig 5.15}
\end{figure}

It seems to be the reasonable idea to consider diagrams each
component of which can have different orientations (i.e. each
components is divided into arcs and every arc is oriented).  The
author tried a polynomial invariant and his computations show that
the problem is difficult but hopefully not impossible to solve (we
suggest to consider the simple diagram from Figure \ref{Fig 5.16}
and to build a resolving tree starting at first from $p$ and then
starting from $q$.

\begin{figure}[htbp]
\centering
\includegraphics{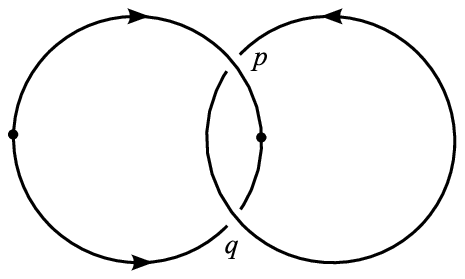}
\caption{}\label{Fig 5.16}
\end{figure}

\section*{Table (made by T.~Przytycka)}

The following table gives a braid expression, the value of the
Jones-Conway-Kauffman polynomial ($\widetilde{J}_K(a,t,z)$) and the
value of the supersignatures $\sigma_{0.5,0.5}$, $\sigma_{2,2}$,
$\sigma_{1.6, 0.1}$, and $ \sigma_{0.1,1.6}$ for some knots which
were considered in the survey.  For knots up to 10 crossings the
Rolfsen \cite{Ro} notation is used, for knots with 11 or more
crossings we use the notation of Thistlethwaite \cite{Thist_2} or
Perko \cite{Pe} (for the meaning of $K_a$ see the remark before
Problem 3.35).  $\sigma_{0.5, 0.5}$ is the classical
(Murasugi) signature, $\sigma_{2,2}$ is a Tristram-Levine signature
and $\sigma_{1.6, 0.1}$ and $\sigma_{0.1,1.6}$ are supersignatures
associated with the Jones polynomial.\bigskip

\begin{equation*} \begin{array}{ll} 8_8 & \scriptstyle \sigma^{2}_{1} \sigma^{2}_{2}
\sigma^{-2}_{3} \sigma^{-1}_{1} \sigma^{}_{2} \sigma^{-1}_{3} \ \ \
\ \ \ \  -a^{-4}-a^{-2}+2 + a^2 + t^2(a^{-4} + 2a^{-2} -2
-a^2)+t^4(-a^{-2}+1) + \\  & \scriptstyle+
z(2a^{-5}+3a^{-3}+a^{-1}-a-a^3 +
t(3a^{-4}+3a^{-2} +1 -a^2)+t^2(-3a^{-5} -5a^{-3} -3a^{-1}+a^3)+ \\
& \scriptstyle  +t^3(-6a^{-4}-8a^{-2}-2 +2a^2) +
t^4(a^{-5}a^{-1}+2a^1)+t^5(2a^{-4} + 4a^{-2}
+2)+t^6(a^{-3}+a^{-1})), \\ & \scriptstyle  0,0,0,0.
\end{array}\end{equation*}

\begin{equation*} \begin{array}{ll} \overline{10}_{129} & \scriptstyle
\sigma^{}_{1} \sigma^{2}_{2} \sigma^{}_{1} \sigma^{-2}_{3}
\sigma^{-1}_{2} \sigma^{}_{3} \sigma^{}_{1} \sigma^{-2}_{2} \ \ \ \
\ \ \ -a^{-4}-a^{-2}+2 + a^2 + t^2(a^{-4} + 2a^{-2} -2 -a^2) + \\  &
\scriptstyle +t^4(-a^{-2}+1) + z(a^{-5}-a^{-3} -5a^{-1} -5a
-2a^3+t(2a^{-4}-2-2a^{-2})+t^2(-3a^{-5}+ \\  & \scriptstyle +
4a^{-2} +15a^{-1} + 9a +a^3)+t^3(-6a^{-4} + a^{-2} +7
+2a^2)+t^4(a^{-5}-6a^{-3} -11a^{-1}-4a)+\\  & \scriptstyle
+t^5(2a^{-4}-a^{-2}-4)+t^6(2a^{-3}3a^{-1}+a) +t^7(a^{-2}+1)); \\  &
\scriptstyle  0,0,0,0. \end{array}\end{equation*}

\begin{equation*} \begin{array}{ll} {13}_{6714} & \scriptstyle
\sigma^{-2}_{2} \sigma^{-1}_{3} \sigma^{-1}_{1} \sigma^{}_{2}
\sigma^{}_{3} \sigma^{}_{1} \sigma^{-2}_{2} \sigma^{}_{1}
\sigma^{}_{3} \sigma^{-1}_{2} \sigma^{-1}_{3} \sigma^{}_{2}
\sigma^{-1}_{3} \ \ \ \ \ \ \ -a^{-4}-a^{-2}+2 + a^2 + \\  &
\scriptstyle t^2(a^{-4} + 2a^{-2} -2 -a^2)+t^4(-a^{-2}+1) +
z(3a^{-5}+7a^{-3} + 7a^{-1} + 3a + \\  & \scriptstyle
+t(4a^{-4}+6a^{-2}+4) +t^2(-4a^{-5}-18a^{-3}
-27a^{-1}-13a)+t^3(-7a^{-4} -20a^{-2} - \\  & \scriptstyle -14+a^2)
+t^4(a^{-5} + 15a^{-3} +31a^{-1} 17a) +t^5 (2a^{-4} + 19a^{-2} +17)
+t^6(-6a^{-3} - \\  & \scriptstyle -
13a^{-1}-7a)+t^7(-7a^{-2}-7)+t^8(a^{-3}+2a^{-1}+a)+t^9(a^{-2}+1)),
\\  & \scriptstyle 0,0,0,0. \end{array}\end{equation*}

\begin{equation*} \begin{array}{ll} \overline{11}_{388} &
\scriptstyle \sigma^{5}_{1} \sigma^{}_{3} \sigma^{-1}_{2}
\sigma^{-2}_{1} \sigma^{-2}_{3} \sigma^{2}_{2}, \, a^{-6}+4a^{-4}
+5a^{-2} +3+t^2(-5a^{-4}-10a^{-2}-4)+ \\  & \scriptstyle
+t^4(a^{-4}+6a^{-2}+1)-t^6a^{-2}+z(a^{-5}+3a^{-3} +2a^{-1}+t(-a^{-6}
-8a^{-4} -13a^{-2}-7)+ \\  & \scriptstyle
+t^2(-a^{-5}-12a^{-3}-11a^{-1})+t^3(14a^{-4}24a^{-2}+14)t^4(15a^{-3}+15a^{-1})+
\\  & \scriptstyle +t^5(
-7a^{-4}-13a^{-2}-7)+t^6(-7a^{-3}-7a^{-1})+t^7(a^{-4}+2a^{-2}+1)+t^8(a^{-3}+a^{-1})),
\\  & \scriptstyle -4, 0, -4, -4. \end{array}\end{equation*}

\begin{equation*} \begin{array}{ll} \overline{9}_{42}=K_3=K_{-2} &
\scriptstyle \sigma^{-3}_{2} \sigma^{-1}_{3} \sigma^{}_{1}
\sigma^{-1}_{2} \sigma^{2}_{3} \sigma^{}_{1} \sigma^{-1}_{2}
\sigma^{}_{3}, \, -2a^{-2}-3-2a^2+t^2(a^{-2}+4+a^2)-t^4 + \\  &
\scriptstyle +z(-2a^{-1}-2a +t(5a^{-2}+8+5a^2)+t^2(6a^{-1}+6a)
+t^3(-5a^{-2}-9-5a^2)+ \\  & \scriptstyle
t^4(-5a^{-1}-5a)+t^5(a^{-2}+2 +a^2)+t^6(a^{-1}+a), \\  &
\scriptstyle 2,0,2,2.\end{array}\end{equation*}

\begin{equation*} \begin{array}{ll} 10_{71} &
\scriptstyle \sigma^{-2}_{3} \sigma^{-1}_{1} \sigma^{-1}_{4}
\sigma^{}_{2} \sigma^{-1}_{3} \sigma^{-1}_{4} \sigma^{}_{2}
\sigma^{-1}_{1} \sigma^{}_{2} \sigma^{}_{3} \sigma^{2}_{2}
\sigma^{}_{4}, \, -a^{-4} -3a^{-2}-3-3a^2-a^4+ \\  & \scriptstyle +
t^2(a^{-4}+4a^{-2}+5+4a^2 +a^4)+t^4(-2a^{-2}-3
-2a^2)+t^6+z(a^{-5}+a^{-3}-a^{-1}-a + \\  & \scriptstyle +a^3 +a^5
t(3a^{-4}+6a^{-2}+7+6a^2+3a^4)+t^2(-2a^{-5}+7a^{-1}+7a -2a^5)+ \\  &
\scriptstyle +t^3(-6a^{-4}-10a^2 -9
-10a^2-6a^4)+t^4(a^{-5}-5a^{-3}-15a^{-1}-15a -5a^3 +a^5)+\\  &
\scriptstyle t^5(3a^{-4}+2a^{-2}-3 + 2a^2+3a^4) +
t^6(4a^{-3}+8a^{-1}+8a 4a^3)+t^7(3a^{-2}+6+ \\  & \scriptstyle
+3a^2)+t^8(a^{-1}+a)), \\  & \scriptstyle 0,0,0,0.
\end{array}\end{equation*}

\begin{equation*} \begin{array}{ll} \overline{11}_{394} &
\scriptstyle \sigma^{-2}_{1} \sigma^{-2}_{3} \sigma^{}_{2}
\sigma^{}_{3} \sigma^{-1}_{1} \sigma^{-1}_{2} \sigma^{}_{3}
\sigma^{2}_{2}, \, 2a^{-2} +5 +2a^2 +t^2(-3a^{-2}-8
-3a^2)+t^4(a^{-2}+ \\  & \scriptstyle +5
+a^2)-t^6+z(-3a^{-3}-7a^{-1}-7a -3a^3 +t(-5a^{-2} -8
-5a^2)+t^2(7a^{-3}+16a^{-1} + \\  & \scriptstyle +16a +7a^3)+
t^3(13a^{-2}+23+ 13a^2
)+t^4(-5a^{-3}-8a^{-1}-8a-5a^3)+t^5(-10a^{-2}- \\  & \scriptstyle
-19
-10a^2)+t^6(a^{-3}-2a^{-1}-2a+a^3)+t^7(2a^{-2}+4+2a2)+t^8(a^{-1}+a)),
\\  & \scriptstyle 2,0,2,2.\end{array}\end{equation*}

\begin{equation*} \begin{array}{ll} \overline{11}_{449}=K_4=K_{-3} &
\scriptstyle \sigma^{-4}_{2} \sigma^{-1}_{3} \sigma^{}_{1}
\sigma^{-1}_{2} \sigma^{3}_{3} \sigma^{}_{1} \sigma^{-1}_{2}
\sigma^{}_{3}, \, a^{-2}+3+a^2+t^2(-3a^{-2}-7-3a^2)+ \\  &
\scriptstyle
+t^4(a^{-2}+5+a^2)-t^6+z(a^{-5}+2a^{-3}+a^{-1}-a-a^3+t(a^{-4}-3a^{-2}-8-5a^2)+\\
& \scriptstyle
+t^2(-3a^{-3}-7a^{-1}+2a+6a^3)+t^3(7a^{-2}+20+16a^2)+t^4(a^{-3}+8a^{-1}+2a-5a^3)+\\
& \scriptstyle
+t^5(-5a^{-2}-15-11a^2)+t^6(-5a^{-1}-4a+a^3)+t^7(a^{-2}+3+2a^2)+t^8(a^{-1}+a)),
\\  & \scriptstyle 2,0,2,2.\end{array}\end{equation*}

\begin{equation*} \begin{array}{ll} 10_{48} &
\scriptstyle \sigma^{4}_{1} \sigma^{-3}_{2} \sigma^{}_{1}
\sigma^{-2}_{2}, \,
4a^{-2}+9+4a^2+t^2(-8a^{-2}-20-8a^2)+t^4(5a^{-2}+18a+5a^2)+ \\  &
\scriptstyle
+t^6(-a^{-2}-7-a^2)+t^8+z(2a^{-5}-7a^{-1}-9a-3a^3+a^5+t(2a^{-4}-3a^{-2}-7-
\\  & \scriptstyle
-5a^2+a^4)+t^2(-3a^{-5}-a^{-3}+12a^{-1}+21a+8a^3-3a^5)+t^3(-5a^{-4}+4a^{-2}+19+\\
& \scriptstyle +13a^2 -5a^4)+t^4(a^{-5}-3a^{-3}-5a^{-1}-11a
-9a^3+a^5)+t^5(2a^{-4}-4a^{-2}-13-\\  & \scriptstyle
-10a^2+2a^4)+t^6(2a^{-3}+a+3a^3)+t^7(2a^{-2}+4+3a^2)+t^8(a_{-1}+a)),
\\  & \scriptstyle 0,0,0,0.\end{array}\end{equation*}

\begin{equation*} \begin{array}{ll} 10_{104} &
\scriptstyle \sigma^{2}_{1} \sigma^{-3}_{2} \sigma^{2}_{1}
\sigma^{-1}_{2} \sigma^{}_{1} \sigma^{-1}_{2}, \,
a^{-2}+3+a^2+t^2(-5a^{-2}-11-5a^2)+t^4(4a^{-2}+13+4a^2)+ \\  &
\scriptstyle t^6(-a^{-2}-6 -a^2) +t^8 +z(-2a^{-3}-4a^{-1} -2a +a^3
+a^5 +t(2a^{-4}-a^{-2}-4+a^2 +3a^4)+\\  & \scriptstyle +t^2(-2a^{-5}
+8a^{-3}+13a^{-1}+4a-a^3-2a^5)+t^3(-6a^{-4}+8a^{-2}+14 -a^2 -6a^4)+
\\  & \scriptstyle +t^4( a^{-5}-11a^{-3}-12a^{-1}-6a
-5a^3+a^5)+t^5(3a^{-4}-10a^{-2}-16 -4a^2 +3a^4)+\\  & \scriptstyle
+t^6(5a^{-3}3a^{-1}+2a +4a^3) +t^7(5a^{-2}+8
+4a^2)+t^8(2a^{-1}+2a)), \\  & \scriptstyle 0, 0, 0,
0.\end{array}\end{equation*}

\begin{equation*} \begin{array}{ll} 10_{125} &
\scriptstyle \sigma^{-3}_{1} \sigma^{-1}_{2} \sigma^{5}_{1}
\sigma^{-1}_{2} =\Delta^{-2}\sigma^{7}_{1} \sigma^{-1}_{2}, \,
3a^{-2}+7 +3a^2+t^2(-4a^{-2}-11-4a^2)+ \\  & \scriptstyle +
t^4(a^{-2}+6+a^2) -t^6 +z( a^{-5}-a^{-3} -6a^{-1}-8a -4a^3 +t(
a^{-4}-2a^{-2}-4-4a^2)+ \\  & \scriptstyle
+t^2(a^{-2}+8a^{-1}_17a+10a^3)+t^3(a^{-2}+7+10a^2)+t^4(-5a^{-1}-11a-6a^3)+
\\  & \scriptstyle +t^5(-5-6a^2) +t^6(a^{-1}+2a+a^3)+t^7(1+a^2)), \\  &
\scriptstyle -2, 0, -2, -2.\end{array}\end{equation*}

\ \\
\ \\
\footnotesize{Z-d Ma{\l}ej Poligr. U.W. zam. 981/86; 100 egz.}

\newpage
\begin{thebibliography}{99}

\bibitem[B--M]{B_M} R.~Ball, M.L.~Metha, Sequence of invariants for knots
and links, J.~Physique 42(1981), 1193--1199.

\bibitem[B--S]{B_S} S.~Bleiler, M.~Scharlemann, Tangles, property $P$ and
a problem of Martin, Preprint 1985.

\bibitem[Bi--1]{Bi_1}J.S.~Birman, Braids, links and mapping class groups,
\emph{Ann. Math. Studies} 82, Princeton Univ. Press, 1974.

\bibitem[Bi--2]{Bi_2}J.S.~Birman, On the Jones Polynomial of closed
3-braids, \emph{Invent. Math.} 81(2), 1985, 287--294.

\bibitem[Bi--3]{Bi_3}J.S.~Birman, Jones plat-braid formulae,
\emph{Abstracts of AMS} 6(5), 1985, p.335.

\bibitem[Bi--4]{Bi_4}J.S.~Birman, Jones braid-plat formulae, and a new
surgery triples, preprint 1985.

\bibitem[B--L--M]{B_L_M} R.D.~Brandt, W.B.~Lickorish, K.C.~Millett, A
polynomial invariant for unoriented knots and links, preprint 1985.

\bibitem[B--Z]{B_Z}G.~Burde, H.~Zieschang, Knots, \emph{De Gruyter studies in
Math.} 5, Berlin, New York 1985.

\bibitem[Co]{Co} J.H.~Conway, An enumeration of knots and links, and
some of their algebraic properties, \emph{Computation problems in
abstract algebra} (J.~Leech, ed.), Pergamon Press, Oxford and New
York (1969), 329--359.

\bibitem[F--W]{F_W} J.~Franks, R.F.~Williams, Braids and the Jones
polynomial, preprint, 1985.

\bibitem[F--Y--H--L--M--O]{F_Y_H_L_M_O} P.~Freyd, D.~Yetter, J.~Hoste,
W.B.R.~Lickorish, K.~Millett, A.~Ocneanu, A new polynomial
invariant of knots and links, \emph{Bull. Amer. Math. Soc.} 12(2)
1985, 239--249.

\bibitem[Ga]{Ga}  D.~Gabai,
Foliations and genera of links, {\it Topology}, 23(1),
1984, 381-394 (this reference was missing in the original ``Survey" and
it is added for e-print).

\bibitem[Gi]{Gi}C.~Giller, A family of links and the Conway calculus,
\emph{Trans. Amer. Math. Soc.} 270(1982), 75--109.

\bibitem[Go]{Go} C.McA.~Gordon, Some aspects of classical knot theory.
In: \emph{Knot theory}, L.N.M. 685, 1978, 1--160.

\bibitem[Ho]{Ho} C.F.~Ho, A new polynomial invariant for knots and
links -- preliminary report, \emph{abstracts of AMS} 6(1985), p.300.

\bibitem[Hod]{Hod}C.O.~Hodgson, Involutions and isotopies of lens spaces,
MS thesis, Univ. of Melbourne (1981)

\bibitem[Hos-1]{Hos-1} J.~Hoste, A polynomial invariant of knots and links,
preprint 1985.

\bibitem[Jo--1]{Jo_1} V.F.R.~Jones, Letter to J.~Birman (May 31, 1984).

\bibitem[Jo--2]{Jo_2} V.F.R.~Jones, A polynomial invariant for knots via
von Neumann algebras, \emph{Bull. Amer. Math. Soc.} 12(1) 1985,
103-111.

\bibitem[Ka--1]{Ka_1} T.~Kanenobu, Infinitely many knots with the same polynomial
invariant, \emph{Proc. Amer. Math. Soc.} (to appear).

\bibitem[Ka--2]{Ka_2} T.~Kanenobu, Letter to P.~Traczyk (Nov. 13 1985).

\bibitem[Ka--3]{Ka_3} T.~Kanenobu, Examples of polynomial invariants of
knots and links, preprint 1985.

\bibitem[Ka--M]{Ka_M} T.~Kanenobu, H.~Murakami, Two-bridge knots with
unknotting number one, prepring 1985.


\bibitem[K--1]{K_1} L.H.~Kauffman, The Conway polynomial, \emph{Topology}
20 1980, 101--108.

\bibitem[K--2]{K_2} L.H.~Kauffman, Combinatorics and knot theory,
\emph{Contemporary Mathematics}, Vol. 20, 1983, 181--200.

\bibitem[K--3]{K_3} L.H.~Kauffman, Knots, Lecture notes, Zaragoza, Spring
1984.

\bibitem[K--4]{K_4} L.H.~Kauffman, A geometric interpretation of the
generalized polynomial, preprint, 1985.

\bibitem[K--5]{K_5} L.H.~Kauffman, An invariant of regular isotopy,
preprint 1985.

\bibitem[K--6]{K_6} L.H.~Kauffman, State models for knot polynomials,
preprint, 1985.

\bibitem[K--7]{K_7} L.H.~Kauffman, Chromatic polynomial (Potts model),
Jones polynomial, preprint 1985.

\bibitem[K--T]{K_T} S.~Kinoshita, H.~Terasaka, On unions of knots, \emph{Osaka
Math. J.} 9(1957), 131--153.

\bibitem[Le]{Le} J.~Levine, Knot cobordism groups in codimension two.
\emph{Comment. Math. Helv.} 44(1969) 229--244.

\bibitem[Li--1]{Li_1} W.B.R.~Lickorish, Prime knots and tangles, \emph{Trans.
Amer. Math. Soc.} 271(1), 1981, 321--332.

\bibitem[Li--2]{Li_2} W.B.R.~Lickorish, A relationship between link
polynomials, preprint 1985.

\bibitem[Li--M--1]{Li_M_1} W.B.R.~Lickorish, K.C.~Millett, A polynomial invariant
of oriented links, preprint 1985.

\bibitem[Li--M--2]{Li_M_2} W.B.R.~Lickorish, K.C.~Millett, The reversing result
for the Jones polynomial, \emph{Pacific J. Math.} (to appear).

\bibitem[Li--M--3]{Li_M_3} W.B.R.~Lickorish, K.C.~Millett, Some evaluations of
link polynomials, preprint 1985.

\bibitem[Mo--1]{Mo_1} H.R.~Morton, Closed braid representatives for a link,
and its 2-variable polynomial, preprint 1985.

\bibitem[Mo--2]{Mo_2} H.R.~Morton, Seifert circles and knot polynomials,
preprint 1985.

\bibitem[Mo--3]{Mo_3} H.R.~Morton, The Jones polynomial for unoriented
links, preprint 1985.

\bibitem[Mo--S]{Mo_S} H.R.~Morton, H.B.~Short, The 2-variable polynomial of
cable knots, preprint 1986.

\bibitem[Mur--1]{Mur-1} H.~Murakami, A recursive calculation of the Arf
invariant of a link, preprint 1984.

\bibitem[Mur--2]{Mur-2} H.~Murakami, A note on the first derivative of the
Jones polynomial, preprint 1984.

\bibitem[Mur--3]{Mur-3} H.~Murakami, A note on the second derivative of the
Jones polynomial, preprint 1985.

\bibitem[Mur--4]{Mur-4} H.~Murakami, Unknotting number and polynomial
invariants of a link, preprint 1985.

\bibitem[Mu--1]{Mu_1} K.~Murasugi, On closed 3-braids, \emph{Memoirs AMS}
151, 1974, Amer. Math. Soc. Providence, RI.

\bibitem[Mu--2]{Mu_2} K.~Murasugi, Jones polynomial of alternating links,
\emph{Trans. Amer. Math. Soc.} 295(1) 1986.

\bibitem[Mu--3]{Mu_3} K.~Murasugi, Jones polynomial and classical
conjectures in knot theory, preprint 1985.

\bibitem[Mu--4]{Mu_4} K.~Murasugi, On the signature of links,
\emph{Topology} 9(1970) 283--298.

\bibitem[Mu--5]{Mu_5} K.~Murasugi, On the certain numerical invariant of
link types, \emph{Trans. Amer. Math. Soc.} 117(1965), 387--422.

\bibitem[Oc]{Oc} A.~Ocneanu, A polynomial invariant for knots: a
combinatorial and algebraic approach, preprint 1985.

\bibitem[Pe]{Pe} K.A.~Perko, Invariants of 11-crossing knots,
\emph{Publications Math. d'Orsay}, 1980.

\bibitem[P--1]{P_1} J.H.~Przytycki, \emph{Knot theory}, Warsaw University Press
(in preparation); in Polish.

\bibitem[P--2]{P_2} J.H.~Przytycki, Skein equivalence of (2,k)-cables of
mutants of knots, preprint 1986 (Added for e-print: It was published as\ \
 Equivalence of cables of mutants of
knots, {\it Canad. J. Math.}, 26 (2) 1989, 250-478).

\bibitem[P--T--1]{P_T_1} J.H.~Przytycki, P.~Traczyk, Invariants of links of
Conway type, \emph{Kobe J. Math.} (to appear).

\bibitem[P--T--2]{P_T_2} J.H.~Przytycki, P.~Traczyk, Conway algebras and skein
equivalence of links, preprint 1985.

\bibitem[Re]{Re} K.~Reidemeister, Knotentheorie, \emph{Ergebn. Math.
Grenzgeb.} Bd.1; Berlin: Springer-Verlag, 1932.

\bibitem[Ri]{Ri} R.~Riley, Homomorphisms of knot groups on finite
groups, \emph{Math. Comp.} 25(1971), 603--619.

\bibitem[Ro]{Ro} D.~Rolfsen, \emph{Knots and links}, Publish or Perish,
Inc. Berkeley 1976; Math. Lect. Series 7.

\bibitem[Ta]{Ta} P.G.~Tait, On knots, Scientific paper I, Cambridge
University Press, 1898, London, 273--347.

\bibitem[Thist--1]{Thist_1} M.B.~Thistlethwaite, Knot tabulations and related
topics, \emph{Aspects of Topology}, Ed. I.M.~James and
E.H.~Kronheimer, \emph{LMS Lects. Notes} 93(1985), 1-76.

\bibitem[Thist--2]{Thist_2} M.B.~Thistlethwaite, Knots to 13-crossings,
\emph{Math. Comp.} (to appear).

\bibitem[Tra]{Tra} B.~Trace, On the Reidemeister moves of a classical
knot, \emph{Proc. Amer. Math. Soc.} 89(1983), 722--724.

\bibitem[Tr]{Tr} A.G.~Tristram, Some corbordism invariants for links,
\emph{Proc. Cambridge Phil. Soc.} 66(1969), 251--264.

\bibitem[Vi]{Vi} O.Ya.~Viro, Letter to J.~Przytycki (September 1985).

\end{thebibliography}
\end{document}